\documentclass[11pt]{amsart}

\allowdisplaybreaks
\usepackage{mathpazo}
\usepackage{amsmath,amsfonts,amssymb}
\usepackage{amsrefs}
\usepackage{amsthm}
\usepackage{latexsym,amsmath,amssymb,amsfonts}
\usepackage{rotating}
\usepackage{mathrsfs}
\usepackage{xypic} \xyoption{all}
\usepackage{amscd}
\usepackage{hyperref} %% pdflatex
\usepackage{euscript}
\usepackage{hhline}
\usepackage{graphicx,epstopdf}
\usepackage{epsfig}
\usepackage{xcolor}
\usepackage[all,color]{xy}

\newlength{\fighskip} \fighskip=2pt
\newlength{\figvskip} \figvskip=3pt

\usepackage{hyperref}
\newcommand*{\figbox}[2]{{
  \def\figscale{#1}
  \def\arraystretch{0.8}
  \arraycolsep=0pt
  \begin{array}{c}
    \vbox{\vskip\figscale\figvskip
      \hbox{\hskip\figscale\fighskip
        \includegraphics[scale=\figscale]{#2}}}
  \end{array}}}

\advance\textheight by \topskip
\numberwithin{equation}{section}

\newcommand{\op}{\operatorname}
\newcommand{\C}{\mathbb{C}}

\newcommand{\R}{\mathbb{R}}
\newcommand{\Q}{\mathbb{Q}}
\newcommand{\Z}{\mathbb{Z}}

\newcommand{\E}{{\mathcal E}}

\renewcommand{\H}{\mathbb{H}}

\newcommand{\g}{\mathfrak{g}}

\newcommand{\abracket}[1]{\left\langle#1\right\rangle}
\newcommand{\bbracket}[1]{\left[#1\right]}
\newcommand{\fbracket}[1]{\left\{#1\right\}}
\newcommand{\bracket}[1]{\left(#1\right)}

\newcommand{\cinfty}{C^{\infty}}
\newcommand{\pa}{\partial}

\renewcommand{\dbar}{\bar\pa}
\newcommand{\OO}{{\mathcal O}}

\newcommand{\BV}{Batalin-Vilkovisky }
\newcommand{\CE}{Chevalley-Eilenberg }

\newcommand{\into}{\hookrightarrow}
\newcommand{\Ol}{\mathcal O_{loc}}

\newcommand{\iso}{\cong}

\newcommand{\Kahler}{K\"{a}hler }
\newcommand{\MC}{Maurer-Cartan }

\renewcommand{\Im}{\op{Im}}
\renewcommand{\Re}{\op{Re}}
\DeclareMathOperator{\Sym}{Sym}
\DeclareMathOperator{\Hom}{Hom}

\DeclareMathOperator{\Tr}{Tr}

\DeclareMathOperator{\Jet}{Jet}
\DeclareMathOperator{\Res}{Res}
\DeclareMathOperator{\Jac}{Jac}
\newcommand{\A}{\mathcal A}

\renewcommand{\L}{\mathcal L}

\theoremstyle{plain}
\newtheorem{thm}{Theorem}[section]
\newtheorem{thm-defn}{Theorem/Definition}[section]
\newtheorem{lem}[thm]{Lemma}
\newtheorem{lem-defn}[thm]{Lemma/Definition}
\newtheorem{prop}[thm]{Proposition}
\newtheorem{cor}[thm]{Corollary}

\newtheorem*{claim}{Claim}

\theoremstyle{definition}
\newtheorem{defn}[thm]{Definition}
\newtheorem{notn}[thm]{Notation}
\newtheorem{eg}[thm]{Example}

\theoremstyle{remark}
\newtheorem{rmk}[thm]{Remark}

\allowdisplaybreaks[4]  %%%%%%%%%%% choose 1-4, where 4 is the strongest desire to breack

\begin{document}

 \title{{On the B-twisted topological sigma model and Calabi-Yau geometry}}
  \author{Qin Li}
   \address{Department of Mathematics, The Chinese University of Hong Kong, Shatin, Hong Kong;}
   \address{School of Mathematical Sciences, Wu Wen-Tsun Key Laboratory of Mathematics, University of Science and Technology of China\\
Hefei, China}
 \email{qli@math.cuhk.edu.hk}

  \author{Si Li}
   \address{Department of Mathematics and Statistics, Boston University, 111 Cummington Mall, Boston, U.S.A.
  }
 \email{sili@math.bu.edu}
  \date{}

  % ∑‚√Ê
  \maketitle

%%%%%%%%%%%%%%%%%%%%%%%%%%%%%%
%% «∞—‘≤ø∑÷
%%%%%%%%%%%%%%%%%%%%%%%%%%%%%%

\begin{abstract}
We provide a rigorous perturbative quantization of the B-twisted topological sigma model via a first order quantum field theory on derived mapping space in the formal neighborhood of constant maps. We prove that the first Chern class of the target manifold is the obstruction to the quantization via \BV formalism. When the first Chern class vanishes, i.e. on Calabi-Yau manifolds, the factorization algebra of observables gives rise to the expected topological correlation functions in the B-model. We explain a twisting procedure to generalize to the Landau-Ginzburg case, and show that the resulting topological correlations coincide with Vafa's residue formula.
\end{abstract}

  \tableofcontents

\section{Introduction}
Mirror symmetry predicts dualities between quantum geometries on Calabi-Yau manifolds. The two sides of the dual theories are called \emph{A-model}
and \emph{B-model} respectively. The A-model is related to symplectic geometry, which is mathematically established as \emph{Gromov-Witten} theory of
counting holomorphic maps. The B-model is attached to complex geometry, which could be understood via \emph{Kodaira-Spencer gauge} theory.  Such
gauge theory is proposed by Bershadsky, Cecotti, Ooguri and Vafa \cite{BCOV} as a closed string analogue of Chern-Simons theory \cite{Witten-CS} in the B-model, whose classical theory describes the deformation of complex structures. We refer to \cite{Barannikov-Kontsevich,Si-BCOV, Si-elliptic,Si-review, Si-gauge} for some recent mathematical
development of its quantum geometry.

Although Kodaira-Spencer gauge theory provides the geometry of B-model from the point of view of string/gauge duality, a direct mathematical approach
to B-model in the spirit of $\sigma$-model is still lacking. The main difficulty is the unknown measure of path integral on the infinite dimensional
mapping space. Thanks to supersymmetry, the physics of B-model path integral is expected to be fully encoded in the small neighborhood of constant
maps. This allows us to extract physical quantities via classical geometries, such as Yukawa couplings for  genus-$0$ correlation functions, etc (See \cite{mirror} for an introduction).
 It is thus desired to have
a mathematical theory to
reveal the above physics context in the vicinity of constant maps, parallel to the localized space of holomorphic maps in the A-model.

The main purpose of the current paper is to provide a rigorous geometric model to analyze B-model via mapping space. To illustrate our method, we will focus on topological
field theory in this paper, while leaving the topological string for coupling with gravity in future works. In the rest of the introduction, we will
sketch the main ideas and explain our construction. A closely related development of B-model in physics has been communicated recently to us by Losev \cite{Losev}.

The geometry of B-twisted $\sigma$-model (in the spirit of AKSZ-formalism \cite{AKSZ}) describes the mapping
$$
   (\Sigma_g)_{dR}\to T^\vee_X[1],
$$
where $(\Sigma_g)_{dR}$ is the ringed space with the sheaf of de Rham complex on the Riemann surface $\Sigma_g$, and $T^\vee_X[1]$ is the
super-manifold associated to the cotangent bundle of $X$ with degree one shifting in the fiber direction. The full mapping space
is difficult to analyze. Instead we
will consider the mapping space in the formal neighborhood of constant maps. Such consideration is proposed in \cite{Kevin-SUSY} to fit
into the effective renormalization method developed in \cite{Kevin-book}. Therefore the corresponding perturbative quantum field theory can be
rigorously analyzed, which is the main context of the current paper. As we have mentioned above, zooming into the neighborhood of constant maps in the B-model does not lose information in physics due to supersymmetry.

\noindent \textbf{Notations}: We will fix some notations that will be used throughout the paper. For a smooth manifold $M$, we
will let $\A_M$ denote the sheaf of de Rham complex of smooth differential forms on $M$, and let $\A_M^\sharp$ denote the sheaf of
smooth differential forms forgetting the de Rham differential:
$$
   \A_M:=(\A_M^\sharp, d_M).
$$
$D_M$ refers to the sheaf of smooth differential operators on $M$. When $M$ is a complex manifold, $\OO_M$ refers to the sheaf of
holomorphic
functions, and $T_M$ denotes either the holomorphic tangent bundle, or the sheaf of holomorphic tangent vectors,  while its
meaning should be clear from the context. Similarly for the dual $T_M^\vee$. We will use $\Omega_M^\bullet$ to denote the sheaf of
holomorphic de Rham complex on $M$, and $D_M^{hol}$ the sheaf of holomorphic differential operators. The
tensor product $\otimes$ without mentioning its ring means $\otimes_{\C}$.

\subsection{Calabi-Yau model}

The space of fields describing our B-twisted $\sigma$-model is given  by
$$
  \mathcal E:= \mathcal A_{\Sigma_g}\otimes (\mathfrak{g}_X[1]\oplus \mathfrak{g}_X^\vee),
$$
where $\mathfrak{g}_X$ is the sheaf of curved $L_\infty$-algebra on $X$ describing its complex geometry \cite{Kevin-CS}. As a sheaf itself,
$$
   \mathfrak{g}_X=\mathcal A^\sharp_X\otimes_{\OO_X} T_X[-1], \quad \mathfrak{g}_X^\vee =\mathcal A^\sharp_X\otimes_{\OO_X} T_X^\vee[1].
$$
The \CE complex $C^*(\mathfrak{g}_X)$ is a resolution of the sheaf $\OO_X$ of holomorphic functions on $X$, and the curved $L_\infty$-algebra $\mathfrak{g}_X\oplus \mathfrak{g}_X^\vee[-1]$ describes the derived geometry of $T^\vee_X[1]$ (see Section \ref{section-classical} for details).

The \CE differential and the natural symplectic pairing equip $T^\vee_X[1]$ (more precisely its $L_\infty$-enrichment) with the structure of QP-manifold \cite{AKSZ}. The action functional is constructed via the AKSZ-formalism in the same fashion as in \cite{Kevin-CS}, formally written as
$$
  S(\alpha+\beta)=\int_{\Sigma_g} \langle d_{\Sigma_g}\alpha, \beta\rangle +\sum_{k\geq 0}\abracket{{l_k(\alpha^{\otimes k})\over (k+1)!}, \beta}
$$
for $\alpha \in \mathcal A_{\Sigma_g}\otimes \mathfrak{g}_X[1], \beta\in \mathcal A_{\Sigma_g}\otimes \mathfrak{g}_X^\vee$. Here $l_k$'s are the
$L_\infty$-products for $\mathfrak{g}_X$. By construction, the action functional satisfies a version of classical master
equation (See section \ref{classical action}).  One interesting  feature is that $S$ contains only one derivative (coming from
$d_{\Sigma_g}$), and the first-order formulation has been used (e.g. \cite{Losev-I,CDC,Frenkel-Losev}) to describe the the twisted $\sigma$-model around the large volume limit. We follow the more recent formulation \cite{Kevin-CS,Owen-Ryan}, using $L_\infty$-algebra via jet bundle as a coherent way to do perturbative expansion over the target
manifold $X$. In fact, the terms involving $L_\infty$ products exactly represent the curvature of the target (see \cite{Kevin-CS}
for an explanation) in terms of jets.

We would like to do perturbative quantization via Feynman diagrams on the infinite dimensional space $\mathcal E$ analogous to the ordinary non-linear $\sigma$-model \cite{Friedan}. One convenient theory via effective \BV formalism is
developed by Costello \cite{Kevin-book}, and we will analyze the quantization problem via this approach.

\begin{thm}[Theorem \ref{thm:obstruction-quantization}, Theorem \ref{thm:one-loop-correction}]
Let $X$ be a complex manifold.
\begin{enumerate}
\item The obstruction to the existence of perturbative quantization of our B-twisted topological $\sigma$-model is given by $(2-2g)c_1(X)$, where $g$
is the genus of the Riemann surface $\Sigma_g$ and $c_1(X)$ is the first Chern-class of $X$.
\item If $c_1(X)=0$, i.e. X being Calabi-Yau, then there exists a canonical perturbative quantization associated to a choice of holomorphic volume
form $\Omega_X$.
\end{enumerate}
\end{thm}

We refer to section \ref{section-quantization} for the precise meaning of the theorem. The theorem is proved by analyzing Feynman diagrams
with the heat kernel on $\Sigma_g$ associated to the constant curvature metric, and this is consistent with physics that
B-twisting can only exist on Calabi-Yau manifolds. Similar results on half-twisted B-model and 2d holomorphic Chern-Simons theory
have been obtained in \cite{Yuan, Yuan-Owen} via background field method. Another approach to topological B-model via D-module
techniques is communicated to us by Rozenblyum \cite{Nick}.

Given a perturbative quantization, there exists a rich structure of factorization algebra for observables developed by Costello and Gwilliam
\cite{Kevin-Owen}. In our case of quantum field theory in two dimensions, the factorization product for local observables gives rise to the structure
of $E_2$-algebra. A perturbative quantization of a so-called cotangent field theory (where our Calabi-Yau model belongs to) can be
viewed as defining certain \emph{projective volume form} on the space of fields \cite{Kevin-CS}. It allows us to define correlation functions
for local observables via the
local-to-global factorization product. The next theorem concerns with the local and global obvervables in our model.

\begin{thm}
Let $X$ be a compact Calabi-Yau with holomorphic volume form $\Omega_X$.
\begin{enumerate}
\item The cohomology of local quantum observables on any disk $U\subset \Sigma_g$ is $H^*(X, \wedge^* T_X)[[\hbar]]$.
\item The complex of quantum observables on $\Sigma_g$ is quasi-isomorphic to the de Rham complex of a trivial local system on $X$ concentrated at
degree $(2g-2)\dim_{\mathbb C}X$.
\end{enumerate}
\end{thm}

See section \ref{section:global-observable} for the explanation.

Instead of the de Rham cohomology for observables in the Gromov-Witten theory, the observables in the B-model are described by polyvector fields. Let
$\mu_i\in H^*(X, \wedge^* T_X)$, and $ \amalg_iU_i\subset \Sigma_g$ be disjoint union of disks on $\Sigma_g$. Let
$O_{\mu_i,U_i}$ be a local observable in $U_i$ representing $\mu_i$ via the above theorem. Then the factorization product with
respect to the embedding
$$
   \amalg_i U_i \into \Sigma_g
$$
gives a global observable $O_{\mu_1,U_1}\star O_{\mu_2,U_2}\star\cdots\star O_{\mu_k,U_k}$. Following \cite{Kevin-CS}, the
correlation function of topological field theory is defined by the natural integration
$$
   \abracket{O_{\mu_1,U_1}, \cdots, O_{\mu_k,U_k}}_{\Sigma_g}:= \int_X \bbracket{O_{\mu_1,U_1}\star O_{\mu_2,U_2}\star\cdots\star
O_{\mu_k,U_k}} \in \mathbb C((\hbar)).
$$
Here $\bbracket{-}$ is the de Rham cohomology class represented by the quantum observable as in the second part of the above theorem. The degree
shifting implies that the correlation function is zero unless $\sum\limits_i \deg O_{\mu_i,U_i}=\sum\limits_i \deg \mu_i=(2-2g)\dim_{\mathbb C}X$.
Explicit calculation on the sphere gives

\begin{thm}[Theorem \ref{correlation-P1}]
Let $\Sigma_g=\mathbb P^1$, and let $X$ be a compact Calabi-Yau with holomorphic volume form $\Omega_X$. Then
$$
\abracket{O_{\mu_1,U_1}, \cdots, O_{\mu_k,U_k}}_{\mathbb P^1}= \hbar^{\dim_{\C}X}\int_X \bracket{\mu_1\cdots\mu_k \vdash \Omega_X}\wedge \Omega_X,
$$
where $\hbar$ is a formal variable.
\end{thm}

When $\Sigma_g$ is an elliptic curve, the only non-trivial topological correlation function is the partition function without
inputs.

\begin{thm}[Theorem \ref{correlation-elliptic}] Let $g=1$, then $\abracket{1}_{\Sigma_g}=\chi(X)$ is the Euler characteristic of $X$.
\end{thm}

To establish the above computation of correlation functions, we describe a formalism in the spirit of \BV Lagrangian integration,
which is equivalent
to the above definition of correlation functions for our model (Corollary \ref{cor:correlation-function-BV-def}). It not only simplifies the
computation, but also sheds light on the potential application to theories which are not cotangent. In fact, the
Landau-Ginzburg model to be described below is not a cotangent field theory, hence the definition of correlation function in
\cite{Kevin-CS} does not work in this case. However, the \BV Lagrangian integration still makes sense and gives rise to the expected
result (Proposition \ref{thm-LG-correlation}).

\subsection{Landau-Ginzburg model} The  Calabi-Yau model described above allows a natural generalization to the
Landau-Ginzburg model associated to a pair $(X,W)$, where $W$ is a holomorphic function on $X$ called the \emph{superpotential}.
This is accomplished by a \emph{twisting procedure}: at the classical level, the interaction is modified by adding a term $I_W$
(Definition \ref{LG-definition}); at the quantum level, this simple modification is still valid  (Proposition \ref{prop:QME-LG}).
In particular, a choice of holomorphic volume form $\Omega_X$ on $X$ leads to a quantization of our Landau-Ginzburg B-model.

Let us describe the corresponding observable theory. For simplicity, let us assume $X=\C^n$, and the critical set of the superpotential $Crit(W)$ is finite. We let
$\{z^i\}$ be the affine coordinates on $\C^n$, and choose $\Omega_X=dz^1\wedge\cdots\wedge dz^n$. We consider the quantization associated to
the pair $(X, \Omega_X)$ with the twisting procedure described above.

\begin{thm}[Proposition \ref{LG-local-observable}]
The cohomology of Landau-Ginzburg $B$-model local quantum observables on any disk $U\subset \Sigma_g$ is $\Jac(W)[[\hbar]]$.
\end{thm}

Similarly, we use $O_{f,U}$ to denote a local quantum observable representing $f\in \Jac(W)$ in the above theorem. Let
$ \amalg_iU_i\subset \Sigma_g$ be disjoint union of disks on $\Sigma_g$. Then the factorization product
$$
   O_{f_1,U_1}\star\cdots\star O_{f_k,U_k}
$$
defines a global quantum observable on $\Sigma_g$. However, the Landau-Ginzburg theory is no longer a cotangent theory in the sense of
\cite{Kevin-CS}, and the projective volume form interpretation of quantization breaks down. Instead, we directly construct an integration map on
quantum observables following the interpretation of \BV Lagrangian geometry described above. This allows us to define the correlation function (Definition \ref{def:correlation-function-LG})
$$
\abracket{ O_{f_1,U_1}\star\cdots\star O_{f_k,U_k}}_{\Sigma_g}^W
$$
in the Landau-Ginzburg case.
\begin{thm}[Proposition \ref{thm-LG-correlation}]
The correlation function of topological Landau-Ginzburg B-model is
$$
\abracket{O_{f_1,U_1}\star\cdots\star O_{f_k,U_k}}_{\Sigma_g}^W=\sum_{p\in Crit(W)}\Res_p\bracket{f_1\cdots f_k \det(\pa_i\pa_j W)^{g}dz^1\wedge\cdots\wedge dz^n \over \prod_i
\pa_i W},
$$
where $\Res_p$ is the residue at the critical point $p$ \cite{GH}.
\end{thm}
This coincides with Vafa's residue formula \cite{vafa}.

\bigskip

\noindent \textbf{Acknowledgement}: The authors would like to thank Kevin Costello, Ryan Grady, Owen Gwilliam and Yuan Shen for valuable discussions
on quantum field theories. Part of the work was done while the authors were visiting MSC at Tsinghua University in the summer of 2012 and 2013, and the first author
was visiting Boston University in 2012 and 2013. We would like to thank for their hospitality. The first author is partially supported by Chinese Universities
Scientific
Fund WK0010000030. The second author is partially supported by NSF DMS-1309118.

\section{The classical theory}\label{section-classical}
In this section we will describe the geometry of B-twisted topological $\sigma$-model and set up our theory at the classical level.

\subsection{The model}
Let $X$ be a complex manifold, and let $\Sigma_g$  be a closed Riemann surface of genus $g$.
Two-dimensional $\sigma$-models are concerned with the space of maps
$$
   \Sigma_g\to X.
$$
One useful way to incorporate interesting information about the geometry and topology of the target $X$ is to enhance ordinary
$\sigma$-models to supersymmetric ones and apply topological twists. There are two twisted supersymmetric theories that
have been extensively studied both in the mathematics and physics literature: the A-model and the B-model. It leads to the famous
mirror symmetry between symplectic and complex geometries. In this paper we will  mainly focus on the B-model.

One possible mathematical formulation of the quantum field theory of B-twisted $\sigma$-model is proposed by Costello
\cite{Kevin-SUSY} via formal derived geometry, and we will adopt this point of view.

\begin{defn}[\cite{Kevin-SUSY}]\label{B-model-definition}
The (fully twisted) B-model, with source a genus $g$ Riemann surface $\Sigma_g$ and target a complex manifold $X$, is  the cotangent theory to the elliptic
moduli problem of maps $$(\Sigma_g)_{dR}\rightarrow X_{\bar{\partial}}.$$
\end{defn}

In the subsequent subsections, we will explain all the notations and geometric data in the above definition. Basically, we have
enhanced the mapping as from a dg-space $\bracket{\Sigma_g}_{dR}$ to the $L_\infty$-space $X_{\dbar}$ to implement supersymmetry.
However, the full mapping space is complicated and hard to analyze. Instead, we will focus on the locus in the formal neighborhood
of constant maps. Under this reduction, we describe our classical action functional in section \ref{classical action}. From the
physical point of view,  the quantum field theory of B-twisted $\sigma$-model is fully encoded in the neighborhood of constant
maps, thanks to supersymmetry. Therefore we do not lose any information via this consideration.

\subsection{The spaces $\bracket{\Sigma_g}_{dR}$ and $X_{\dbar}$}
\subsubsection{The dg-space $\bracket{\Sigma_g}_{dR}$}
We use $\bracket{\Sigma_g}_{dR}$ to denote the dg-ringed space
$$
\bracket{\Sigma_g}_{dR}=\bracket{\Sigma_g, \A_{\Sigma_g}}
$$
on the Riemann surface $\Sigma_g$, where the structure sheaf is the sheaf of smooth de Rham complex. $\A_{\Sigma_g}$ is an elliptic complex, and we view $\bracket{\Sigma_g}_{dR}$
as an \emph{elliptic ringed space} in the sense of \cite{Kevin-SUSY}.

\subsubsection{The $L_\infty$-space $X_{\dbar}$} The space $X_{\dbar}$ is a derived version of the complex  manifold $X$ itself, which is introduced in \cite{Kevin-CS} to describe
holomorphic Chern-Simons theory. This is a suitable concept to discuss perturbative quantum field theory invariant under diffeomorphism group. It consists of a pair
$$
    X_{\dbar}=\bracket{X, \g_{X}},
$$
 where $\g_X$ is the sheaf of curved $L_\infty$-algebras on $X$ that we describe now. As a graded sheaf on $X$, $\g_X$ is defined by
$$
   \g_X:=\A_X^{\sharp}\otimes_{\OO_X} T_X[-1],
$$
where $T_X[-1]$ is the sheaf of holomorphic tangent vectors with degree shifting such that it is concentrated at degree $1$. To describe the curved
$L_\infty$-structure, we consider
$$
  C^*\bracket{\g_X}:=\widehat{\Sym}_{\A_X^{\sharp}}\bracket{\g_X[1]^\vee}=\prod\limits_{k\geq 0}\Sym^k_{\A_X^{\sharp}}\bracket{\g_X[1]^\vee},
$$
where
$$
\g_X[1]^\vee:=\A_X^{\sharp}\otimes_{\OO_X} T_X^\vee
$$
is the dual sheaf of $\g_X[1]$ over  $\A_X^{\sharp}$, and $\Sym^k_{\A_X^{\sharp}}\bracket{\g_X[1]^\vee}$ is the graded symmetric tensor product of $k$ copies of $\g_X[1]^\vee$ over
$\A_X^{\sharp}$. When $k=0$, we set
$
\Sym^0_{\A_X^{\sharp}}\bracket{\g_X[1]^\vee}\equiv \A_X^{\sharp}.
$

It is easy to see that
$$
C^*\bracket{\g_X}=\A_X^\sharp\otimes_{\OO_X} \widehat{\Sym}_{\OO_X}\bracket{T_X^\vee}.
$$
Thus $C^*\bracket{\g_X}$ is a sheaf of algebras over  $\A_X^{\sharp}$.
\begin{notn}\label{notation-basis}
Let $\{z^1,\cdots,z^n\}$ denote local holomorphic coordinates on $X$, we will let $\{\widetilde{\partial_{z^i}}\}$ denote the
corresponding basis of $\mathfrak{g}_X$ over $\A_X^{\sharp}$,
and let $\{\widetilde{dz^i}\}$ denote the corresponding basis of $\mathfrak{g}_X^\vee$ over $\A_X^{\sharp}$ similarly.
\end{notn}

A curved $L_\infty$-algebra structure on $\g_X$ is a differential on $C^*\bracket{\g_X}$ with which it becomes a dg-algebra over the dg-ring $\A_X$. Such a structure is obtained
in \cite{Kapranov}, which is called a weak Lie algebra there. We reformulate the construction for the application in B-twisted $\sigma$-model. Let us first recall

\begin{defn}\label{def:jet-bundle} Let $E$ be a holomorphic vector bundle on $X$. We define the holomorphic jet bundle $\Jet_X^{hol}(E)$ as follows: let
$\pi_1$ and $\pi_2$ denote
the projection of $X\times X$ onto the first and second component respectively,
$$
 \xymatrix{
   & X\times X \ar[dl]_{\pi_1} \ar[dr]^{\pi_2} & \\
   X && X
 }
$$
then
$$
    \Jet_X^{hol}(E):=\pi_{1*}\bracket{\widehat{\OO}_{\Delta}\otimes_{\OO_{X\times X}}\pi_2^* E},
$$
where $\Delta\into X\times X$ is the diagonal, and $\widehat{\OO}_{\Delta}$ is the analytic formal completion of $X\times X$
along $\Delta$. The jet bundle $ \Jet_X^{hol}(E)$ has a natural filtration defined by
$$
 F^k\Jet_X^{hol}(E):= I_\Delta^k  \Jet_X^{hol}(E),
$$
where $I_\Delta$ is the structure sheaf of $\Delta$.
\end{defn}

It is clear that $\Jet^{hol}_X(E)$ inherits a $D^{hol}_X$-module structure from $\widehat{\OO}_{\Delta}$, and we will let $\Omega_X^*\bracket{\Jet^{hol}_X(E)}$ be the corresponding holomorphic de
Rham complex. The natural embedding
$$
  E\into \Omega_X^*\bracket{\Jet^{hol}_X(E)}
$$
induced by taking Taylor expansions of holomorphic sections is a quasi-isomorphism.

Let us consider a smooth map
$$
   \rho: U\to X\times X,
$$
where $U\subset T_X$ is a small neighborhood of the zero section. We require that $\rho$ is a diffeomorphism onto its image, and if we write
$$
  \rho: (x, v)\mapsto (x, \rho_x(v)),
$$
then $\rho_x(-)$ is holomorphic if we  fix $x$. Such a diffeomorphism can be constructed from a K\"{a}hler metric on $X$ via the
K\"{a}hler normal coordinates. Note that in general $\rho_x(-)$ does not vary holomorphically with respect to $x$. Such a map
$\rho$ induces an isomorphism
$$
    \rho^*: \cinfty(X)\otimes_{\OO_X}\pi_{1*}\bracket{\widehat{\OO}_{\Delta}}\overset{\sim}{\rightarrow} \cinfty(X)\otimes_{\OO_X} \widehat{\Sym}\bracket{T_X^\vee}.
$$
Tensoring with $\A_X^\sharp$, we find the following identification
\begin{equation}\label{eqn:identification-CE-jet}
\rho^*: \A_X^\sharp\otimes_{\OO_X}\Jet_X^{hol}(\OO_X) \overset{\sim}{\rightarrow} C^*\bracket{\g_X}.
\end{equation}
Let $d_{D_X}$ be the de Rham differential on $\A_X^\sharp\otimes_{\OO_X}\Jet_X^{hol}(\OO_X)$ induced from the $D^{hol}_X$-module
structure on $\Jet_X^{hol}(\OO_X)$. We can define a differential $d_{CE}$ on $C^*\bracket{\g_X}$ by
$$
   d_{CE}=\rho^*\circ d_{D_X} \circ \rho^{*-1}.
$$

The differential $d_{CE}$ defines a curved $L_\infty$-structure on $\g_X$, under which $d_{CE}$ is the corresponding \CE differential. We remark that the use of the \Kahler metric is only auxiliary: any choice of smooth splitting of the projection $F^1 \Jet_X^{hol}(\OO_X)\to F^1\Jet_X^{hol}(\OO_X)/F^2 \Jet_X^{hol}(\OO_X)$ can be used to define a curved $L_\infty$-structure, and different choices  are homotopic equivalent \cite{Kevin-CS}. Therefore we will not refer to a particular choice.

\begin{defn}
$\g_X$ is the sheaf of curved $L_\infty$-algebras on $X$ defined by the  \CE complex $\bracket{C^*\bracket{\g_X}, d_{CE}}$. We
will denote the components of the structure maps  (shifted by degree 1) of $\g_X$  by
$$
  l_k: \Sym^k_{\A^\sharp_X}(\g_X[1])\to \g_X.
$$
\end{defn}

Therefore $l_1$ defines $\g_X$ as a dg-module over $\A_X$, $l_k$ are $\A_X^{\sharp}$-linear for $k>1$, and $l_0$ defines the curving. There is a natural quasi-isomorphic embedding
$$
   \bracket{X, \OO_X}\into \bracket{X, C^*\bracket{\g_X}}
$$
and $X_{\dbar}$ is viewed as the derived enrichment of $X$ in this sense.

Classical constructions of vector bundles can be  naturally extended to the $L_\infty$-space $X_{\dbar}$.

\begin{defn}
Let $E$ be a holomorphic vector bundle on $X$. The induced vector bundle $E_{\dbar}$ on the $L_\infty$-space $X_{\dbar}$ is defined by the $\g_{X}$-module whose sheaf of \CE complex $C^*\bracket{\g_X, E_{\dbar}}$ is the dg module
$$
    C^*\bracket{\g_X, E_{\dbar}}:=\A_X^{\sharp} \otimes_{\OO_X}\Jet^{hol}_X(E)
$$
over the dg algebra $C^*(\g_X)$.
\end{defn}

\begin{eg}The tangent bundle $TX_{\dbar}$ is given by the module $\g_X[1]$, with its naturally induced module structure over
$\g_X$. Similarly, the cotangent
bundle $T^*X_{\dbar}$ is given by the natural $\g_X$-module $\g_X[1]^\vee$. Symmetric and exterior tensor products of vector bundles are defined in the same
fashion. For example, $$
   \wedge^kT^*X_{\dbar}=\wedge^k\bracket{\g_X[1]^\vee}
$$
and a $k$-form on $X_{\dbar}$ is a section of the sheaf
$$
C^*\bracket{\g_X, \wedge^k\bracket{\g_X[1]^\vee}}=\A_X^{\sharp}\otimes_{\OO_X}\Jet_X^{hol}(\wedge^k T_X^\vee).
$$
\end{eg}

In Appendix \ref{appendix:L_infty}, we present the corresponding $L_\infty$ constructions in more details.

\subsubsection{Mapping space as $L_\infty$-space}\label{mapping-space}
Let $f:\Sigma_g\rightarrow X$ be a smooth map. The sheaf $$f^*\g_X\otimes_{f^*{\A_{X}}} \A_{\Sigma_g}$$ naturally inherits a curved
$L_\infty$-algebra on $\Sigma_g$ within which \MC elements are defined \cite{Kevin-CS} .
\begin{defn}
A map $\bracket{\Sigma_g}_{dR} \to X_{\dbar}$ consists of a smooth map  $f: \Sigma_g\to X$,  together with a \MC element
$$
   \alpha\in f^*\g_X\otimes_{f^*{\A_{X}}} \A_{\Sigma_g}.
$$
\end{defn}

We would like to consider those maps which are constant on the underlying manifold. As shown in \cite{Kevin-CS}, the space of such maps can be
represented by the $L_\infty$-space
$$
  \bracket{X, \A_{\Sigma_g}\otimes_{\C} \g_X},
$$
which is an enrichment of $X_{\dbar}$ by the information from the Riemann surface $\Sigma_g$.

\subsection{Classical action functional}\label{classical action}
As in Definition \ref{B-model-definition}, our model is defined as the cotangent theory to the elliptic moduli problem of maps
$$
(\Sigma_g)_{dR}\rightarrow
X_{\bar{\partial}}.
$$
The cotangent construction of perturbative field theory is described in \cite{Kevin-Owen} as a convenient way to implement \BV
quantization. In our case, we consider the enlarged mapping space
$$
   \bracket{\Sigma_g}_{dR}\to T^*X_{\dbar}[1].
$$
The dg-space $\bracket{\Sigma_g}_{dR}$ is equipped with a volume form of degree $-2$, and $T^*X_{\dbar}[1]$ has a natural
symplectic form of degree $1$. This fits into the AKSZ-construction \cite{AKSZ} and leads to an odd symplectic structure of degree
$-1$ on the mapping space as desired for \BV formalism.

We are interested in the locus around constant maps. As explained in section \ref{mapping-space}, such locus is represented by the
$L_\infty$-space
$$
  \bracket{X, \A_{\Sigma_g}\otimes_\mathbb{C} \g_{T^*X_{\dbar}[1]} },
$$
where $\g_{T^*X_{\dbar}[1]}=\g_X\oplus \g_X[1]^\vee$ is the curved $L_\infty$-algebra representing $T^*X_{\dbar}[1]$.

\begin{defn} The space of fields of the  B-twisted $\sigma$-model is the $\A_X^\sharp$-module
$$
   \E:=\A_{\Sigma_g}^\sharp\otimes_{\C} \bracket{\g_X[1]\oplus \g_X^\vee}.
$$
\end{defn}

\begin{lem-defn}
There exists a natural graded sympletic pairing $\abracket{-,-}$ on $\E$ of degree $-1$.
\end{lem-defn}

%\begin{proof}
%Let $\alpha_i\otimes g_i\in\A_{\Sigma_g}\otimes_\C\mathfrak{g}_X[1]$, and  $\beta_i\otimes
%g^i\in\A_{\Sigma_g}\otimes_\C\mathfrak{g}_X^\vee$, for $i=1,2$. Then
%the symplectic pairing is defined by
%\begin{align*}
%&\langle \alpha_1\otimes g_1+\beta_1\otimes g^1, \alpha_2\otimes g_2+\beta_2\otimes g^2\rangle \\
%:=&(-1)^{|g_1||\beta_2|}\langle g_1,g^2\rangle\cdot\int_{\Sigma_g}\alpha_1\wedge\beta_2+(-1)^{|g^1||\alpha_2|}\langle
%g^1,g_2\rangle\cdot\int_{\Sigma_g}\beta_1\wedge\alpha_2.
%\end{align*}
\iffalse
It is straightforward to check that this pairing is graded skew-symmetric.
\begin{equation*}
 \begin{aligned}
 &\langle\beta_1\otimes g^1,\alpha_2\otimes g_2\rangle\\
=&(-1)^{|\alpha_2||g^1|}\int_\Sigma\beta_1\wedge\alpha_2\cdot\langle g^1,g_2\rangle\\
=&(-1)^{|\alpha_2||g^1|+|\alpha_2||\beta_1|+|g^1||g_2|+1}\int_\Sigma\alpha_2\wedge\beta_1\cdot\langle g_2,g^1\rangle\\
=&(-1)^{|\alpha_2||g^1|+|\alpha_2||\beta_1|+|g^1||g_2|+|\beta_1||g_2|+1}\langle\alpha_2\otimes g_2,\beta_1\otimes g^1\rangle\\
=&(-1)^{|\alpha_2\otimes g_2||\beta_1\otimes g^1|}(-1)\langle\alpha_2\otimes g_2,\beta_1\otimes g^1\rangle.
\end{aligned}
\end{equation*}
\fi
%\end{proof}

The proof is standard and we omit here. The classical action functional is constructed in a similar way as in \cite{Kevin-CS}.

\begin{defn}\label{defn-classical-functional} The classical action functional is defined as the $\A_X^{\sharp}$-valued formal function on $\E$
$$
S(\alpha+\beta):=\int_{\Sigma_g}\left( \langle d_{\Sigma_g}\alpha,\beta\rangle+
\sum_{k\geqslant
0}\dfrac{1}{(k+1)!}\langle l_k(\alpha^{\otimes k}),\beta\rangle\right),
$$
where $\alpha\in \A_{\Sigma_g}^{\sharp}\otimes \g_X[1], \beta\in \A_{\Sigma_g}^{\sharp}\otimes \g_X^\vee$, $d_{\Sigma_g}$ is the de Rham differential on $\Sigma_g$, and $l_k$ is the $L_\infty$-product for $\g_X$.
\end{defn}

 We will let $$Q=d_{\Sigma_g}+l_1:\mathcal{E}\rightarrow\mathcal{E}$$ and split the classical action
$S$ into its free and interaction parts
\begin{align*}
S&=S_{free}+I_{cl},
\end{align*}
 where $$I_{cl}(\alpha+\beta)=\int_{\Sigma_g}\left(\langle
l_0,\beta\rangle+\sum_{k\geqslant 2}\dfrac{1}{(k+1)!}\langle l_k(\alpha^{\otimes k}),\beta\rangle\right)$$  and
$$S_{free}(\alpha+\beta)=\int_{\Sigma_g}\langle Q(\alpha),\beta\rangle.$$
For later discussion, we denote the following functionals by
\begin{equation}\label{def: tilde-l_k}
\tilde{l}_k(\alpha+\beta):=\dfrac{1}{(k+1)!}\int_{\Sigma_g}\langle l_k(\alpha^{\otimes k}),\beta\rangle, \hspace{5mm} \text{for\ } k\geq 0.
\end{equation}

\subsection{Classical master equation} The classical action functional $S$ satisfies the classical master equation, which is equivalent to the gauge invariance in the \BV formalism. We will explain the classical master equation in this section and set up some notations to be used for quantization later.

\subsubsection{Functionals on fields}
The space of fields $\E$ is an $\A_{\Sigma_g}^\sharp$-module. Let $\E^{\otimes k}$ denote the $\A_{X}^{\sharp}$-linear completed tensor product of $k$ copies of
$\E$, where the completion is over the products of Riemann surfaces. Explicitly,
$$
\E^{\otimes k}:=\A_{\Sigma_g\times \cdots \times \Sigma_g} \otimes_{\C} \bracket{\bracket{\g_X[1]\oplus \g_X^\vee}\otimes_{\A_{X}^{\sharp}}\cdots
\otimes_{\A_{X}^{\sharp}} \bracket{\g_X[1]\oplus \g_X^\vee}}.
$$
The permutation group $S_k$ acts naturally on $\E^{\otimes k}$ and we will let
$$
   \Sym^k\bracket{\E}:=\bracket{\E^{\otimes k}}_{S_k}
$$
denote the $S_k$-coinvariants.

We will use ${\overline \A}_{\Sigma_g}$ to denote the distribution valued de Rham complex on $\Sigma_g$. $\overline{\E}$ will be
distributional sections
of $\E$:
$$
   \overline{\E}={\overline \A}_{\Sigma_g}\otimes_{\C} \bracket{\g_X[1]\oplus \g_X^\vee}.
$$
We will also use
$$
    \E^\vee:= \Hom_{\A_X^{\sharp}}\bracket{\E, \A_X^{\sharp}}
$$
to denote  functionals on $\E$ which are linear in $\A_X^{\sharp}$.  The symplectic pairing $\abracket{-,-}$ gives a natural
embedding
$$
   \E\into \E^\vee[-1],
$$
which induces an isomorphism
$$
  \overline{\E}\iso \E^\vee[-1].
$$

 \begin{defn}  We define the space of $k$-homogenous functionals on $\E$ by  the linear functional (distribution) on $\Sigma_g\times\cdots\times
\Sigma_g$ ($k$-copies)
 $$
    \OO^{(k)}(\E):=\Hom_{\A_X^{\sharp}}\bracket{\Sym^k(\E), \A_X^{\sharp}},
 $$
where our convention is that $\OO^{(0)}(\E)=\A_X^{\sharp}$. We introduce the following notations:
 $$
    \OO(\E):=\prod_{k\geq 0}\OO^{(k)}(\E), \quad    \OO^+(\E):=\prod_{k\geq 1}\OO^{(k)}(\E).
 $$
 \end{defn}

 Therefore $\OO(\E)$ can be viewed as formal power series on $\E$.  The isomorphism $\overline{\E}\iso \E^\vee[-1]$ leads to natural isomorphisms
$$
     \OO^{(k)}(\E)=\bracket{\E^\vee}^{\otimes k}_{S_k}\iso \bracket{\overline{\E}[1]}^{\otimes k}_{S_k},
$$
where the tensor products are the $\A_X^{\sharp}$-linear completed tensor products over $k$ copies of $\Sigma_g$.

\begin{defn}\label{contraction operator} Let $P \in \Sym^k(\E)$. We define the operator of contraction with $P$
$$
    {\pa\over \pa P}: \OO^{(m+k)}(\E)\to \OO^{(m)}(\E)
$$
by
$$
   \bracket{{\pa\over \pa P}\Phi}(\mu_1, \cdots, \mu_m):=\Phi(P, \mu_1, \cdots, \mu_m),
$$
where $\Phi\in \OO^{(m+k)}(\E), \mu_i\in \E$.
\end{defn}

\begin{defn}
We will denote by $\Ol(\E)\subset \OO(\E)$ the subspace of local functionals, i.e. those of the form given by the integration of a Lagrangian density on $\Sigma_g$
$$
     \int_{\Sigma_g} \mathcal L(\mu), \quad \mu \in \E.
$$
$\Ol^+(\E)$ is defined similarly as local functionals modulo constants.
\end{defn}

\begin{eg}
The classical action functional $S$ in Definition \ref{defn-classical-functional} is a local functional.
\end{eg}

\subsubsection{Classical master equation} As a general fact in symplectic geometry, the Poisson kernel of a symplectic form induces a Poisson bracket on the
space of functions. In our case we are dealing with the infinite dimensional symplectic space $\bracket{\E, \abracket{-,-}}$. The Poisson bracket is of the
form of $\delta$-function distribution, therefore the Poisson bracket is well-defined on local functionals.

\begin{lem-defn}\label{Poisson-bracket} The symplectic pairing $\abracket{-,-}$ induces an odd Poisson bracket of degree $1$ on the space of local functionals, denoted by
$$
  \fbracket{-,-}: \Ol(\E)\otimes_{\A_{X}^{\sharp}} \Ol(\E) \to \Ol(\E),
$$
which is bilinear in $\A_X^{\sharp}$.
\end{lem-defn}

\begin{lem}\label{lem:classical-master-equation}
Let $F_{l_1}$ be the functional on $\E$ defined as follows:
$$
F_{l_1}(\alpha+\beta):=\langle l_1^2(\alpha),\beta \rangle,\hspace{5mm} \alpha\in\A^\sharp_{\Sigma_g}\otimes\g_X[1],
\beta\in\A^\sharp_{\Sigma_g}\otimes\g_X^\vee.
$$
The classical interaction functional $I_{cl}$ satisfies the following classical master equation:
\begin{equation} \label{eqn:classical-master-equation-S}
   QI_{cl}+\frac{1}{2}\{I_{cl},I_{cl}\}+F_{l_1}=0.
\end{equation}
\end{lem}
\begin{proof} This follows from the fact that the maps $\{l_k\}_{ k\geq 0}$ of $\g_X$ defines a curved $L_\infty$-structure. See \cite{Kevin-CS}. The extra $F_{l_1}$ describes the
curving: $\{F_{l_1},-\}=Q^2=l_1^2$.
\end{proof}

In particular, Lemma \ref{lem:classical-master-equation} implies that the operator $Q+\{I_{cl},-\}$ defines a differential on $\Ol(\E)$.

\begin{defn} The complex $Ob:=\bracket{\Ol^+(\E), Q+\{I_{cl},-\}}$ is called the deformation-obstruction complex associated to the  classical field theory defined by $(\E, S)$.
\end{defn}

As established in \cite{Kevin-book}, the complex $Ob$ controls the deformation theory of the perburbative quantization of $S$, hence the name.

%For the purpose of performing the quantization procedure later,  it is convenient to separate the interaction part of $S$. It is easy to see that
%$\{S_{free},-\}=Q$, therefore the classical master equation (\ref{eqn:classical-master-equation-S}) is equivalent to
%$$
%   QI_{cl}+{1\over 2}\fbracket{I_{cl},I_{cl}}=0,
%$$
%where $Q$ defines a derivation on $\OO(\E)$ via duality.

\section{Quantization}\label{section-quantization}
In this section we establish the quantization of our B-twisted $\sigma$-model via Costello's perturbative renormalization method \cite{Kevin-book}.
We show that the obstruction to the quantization is given by $(2-2g)c_1(X)$. When $c_1(X)=0$, i.e. $X$ being Calabi-Yau, every choice of holomorphic volume form on
$X$ leads to an associated canonical quantization of the B-twisted $\sigma$-model.

\subsection{Regularization}\label{section-regularization}
Perturbative quantization of the classical action functional $S$ is to model the asymptotic $\hbar$-expansion of the infinite dimensional path integral
$$
      \int_{L\subset \E} e^{S/\hbar},
$$
where $L$ is an appropriate subspace related to some gauge fixing (a BV-Lagrangian in the \BV formalism). A natural formalism based on finite dimensional models is
$$
      \int_{L\subset \E} e^{S/\hbar}\mapsto \exp\bracket{\hbar^{-1}W(G, I_{cl})},
$$
where $W(G, I_{cl})$ is the weighted sum of Feynman integrals over all connected graphs, with $G$ ($=\mathbb{P}_0^\infty$ below) labeling the internal edges, and $I_{cl}$ labeling
the vertices.  One essential difficulty is the infinite dimensionality of the space of fields which introduces singularities in the propagator $G$ and breaks the naive
interpretation of Feynman diagrams. Certain regularization is required to make sense of the theory, which is the celebrated idea of
renormalization in quantum field theory. We will use the heat kernel regularization to fit into Costello's renormalization technique \cite{Kevin-book}.

\subsubsection{Gauge fixing}
We need to choose a gauge fixing operator for regularization. For any Riemann surface $\Sigma_g$, we pick the  metric on $\Sigma_g$ of constant curvature $0,1$ or $-1$, depending on the genus $g$. In particular, we choose the hyperbolic
metric on $\Sigma_g$ when $g>1$. The gauge fixing operator is
$$
Q^{GF}:=d_{\Sigma_g}^*,
$$
where $d_{\Sigma_g}^*$ is the adjoint of the de Rham differential $d_{\Sigma_g}$ on $\Sigma_g$ with respect to the chosen metric.
It is clear that the Laplacian $H=[Q,Q^{GF}]=d_{\Sigma_g}d_{\Sigma_g}^*+d_{\Sigma_g}^*d_{\Sigma_g}$ is the usual Laplacian on $\A_{\Sigma_g}$. We will let $e^{-tH}$ denote the heat operator acting  on
$\A_{\Sigma_g}$ for $t>0$.

\begin{rmk}The operators $Q^{GF}, H$ and  $e^{-tH}$ extend trivially over $\g_X[1]\oplus \g^\vee_X$ to define operators on $\E$, and we will use the same notations without confusion.
\end{rmk}
\subsubsection{Effective propagator}
To analyze B-twisted $\sigma$-model, we first describe the propagator of the theory.

\begin{defn}
The heat kernel $\mathbb{K}_t$ for $t>0$ is the element in $\Sym^2\bracket{\E}$ defined by the equation
$$
\langle \mathbb{K}_t(z_1,z_2),\phi(z_2)\rangle=e^{-tH}(\phi)(z_1), \quad \forall \phi\in \E, z_1\in \Sigma_g.
$$
\end{defn}

\begin{notn} The fact that  the symplectic pairing on $\E$ is (up to sign) the tensor product of the natural pairings on $\A_{\Sigma_g}^\sharp$ and
$\g_X[1]\oplus\g_X^\vee$ implies that the heat kernel $\mathbb{K}_t(z_1,z_2)$ is of the following form:
$$
\mathbb{K}_t(z_1,z_2)=K_t(z_1,z_2)\otimes(\text{Id}_{\g_X}+\text{Id}_{\g_X^\vee})
,$$
where $K_t$ is simply the usual heat kernel of $e^{-tH}$ on $\Sigma_g$, and $\text{Id}_{\g_X}+\text{Id}_{\g_X^\vee}$ is the Poisson kernel  corresponding to the natural symplectic
pairing on $\g_X[1]\oplus\g_X^\vee$. We will call $K_t$ and $\text{Id}_{\g_X}+\text{Id}_{\g_X^\vee}$ the analytic and combinatorial part of $\mathbb{K}_t$ respectively.
\end{notn}
The combinatorial part of $\mathbb{K}_t$ can be described locally as follows: pick a local basis $\{X_i\}$  of $\mathfrak{g}_X[1]$ as an
$\mathcal{A}_X^{\sharp}$-module, and let $\{X^i\}$ be the corresponding dual basis of $\mathfrak{g}_X^\vee$. Then we have
$$
\text{Id}_{\g_X}+\text{Id}_{\g_X^\vee}=\sum_{i}(X_i\otimes X^i+X^i\otimes X_i).
$$

\begin{defn}\label{defn:propagator}
For $0<\epsilon<L<\infty$, we define the effective propagator $\mathbb{P}_\epsilon^L$ as the element in $\Sym^2\bracket{\E}$  by
$$
\mathbb{P}_\epsilon^L(z_1,z_2)=P_\epsilon^L(z_1,z_2)\otimes (\text{Id}_{\g_X}+\text{Id}_{\g_X^\vee}),
$$ where the analytic part of the propagator $P_\epsilon^L$ is given by
$$
P_\epsilon^L:=\int_\epsilon^L (Q^{GF}\otimes 1) K_tdt.
$$
\end{defn}
\begin{rmk}
In the notations $P_\epsilon^L(z_1,z_2)$ and $K_t(z_1,z_2)$, we have omitted their anti-holomorphic dependence for simplicity.
\end{rmk}
In other words, $\mathbb{P}_\epsilon^L$ is the kernel representing the operator $\int_\epsilon^L Q^{GF}e^{-tH}dt$ on $\E$.  The full propagator
$\mathbb{P}_0^\infty$ represents the operator ${Q^{GF}\over H}$, which is formally the inverse of the quadratic pairing  $S_{free}$ after gauge
fixing. The standard trick of Feynman diagram expansions picks $\mathbb{P}_0^\infty$ as the propagator. However $\mathbb{P}_0^\infty$ exhibits
singularity along the diagonal in $\Sigma_g\times\Sigma_g$, and the above effective propagator with cut-off parameters $\epsilon, L$ is viewed as a
regularization.

It is known that the heat kernel $K_t$ on a Riemann surface $\Sigma_g$ has an asymptotic expansion:
\begin{equation}\label{eqn:asymp-heat-kernel}
K_t(z_1,z_2)\sim \dfrac{1}{4\pi t}e^{-\frac{\rho^2(z_1,z_2)}{4t}}\left(\sum_{i=0}^\infty t^i\cdot a_i(z_1,z_2)\right) \quad \text{as}\  t\rightarrow 0,
\end{equation} where each $a_i(z_1,z_2)$ is a smooth $2$-form on $\Sigma_g\times\Sigma_g$ and $\rho(z_1,z_2)$ denotes the geodesic distance between $z_1$ and
$z_2$. Similarly, for the propagator $P_\epsilon^L$, we have
\begin{lem}[Appendix \ref{appendix:asymp-propagaotr}]\label{lem:asymp-propagator}
The propagator on the hyperbolic upper half plane $\mathbb{H}$ is given explicitly by
\begin{equation}\label{eqn:propagator}
P_{\epsilon}^L=\int_{\epsilon}^Lf(\rho,t)dt\cdot\left(\dfrac{2(x_1-x_2)}{y_1y_2}(dy_1-dy_2)-\dfrac{(y_1-y_2)(y_1+y_2)}{y_1y_2
}
\left(\dfrac{dx_1}{y_1}-\dfrac{dx_2}{y_2}\right)\right),
\end{equation}
where  $x_i=\Re z_i$ and $y_i=\Im
z_i$, for $i=1,2$. The function $f(\rho,t)$ is smooth on
$\mathbb{R}_{\geqslant 0}\times\mathbb{R}_{>0}$, and has an asymptotic expansion as $t\rightarrow 0$:
\begin{equation}\label{eqn:asymp-propagator}
 f(\rho,t)\sim
\sum_{k=0}^\infty t^{-2+k}e^{-\frac{\rho^2}{4t}}b_k(\rho).
\end{equation}
\end{lem}

\subsubsection{Effective \BV formalism} The heat kernel cut-off also allows us to regularize the Poisson bracket $\fbracket{-,-}$
and extend its definition from
local functionals to all distributions.
\begin{defn} We define the effective BV Laplacian $\Delta_L$ at scale $L>0$
$$
  \Delta_L:={\pa\over \pa \mathbb{K}_L}: \OO\bracket{\E}\to \OO\bracket{\E}
$$
by contracting with $\mathbb{K}_L$ (see Definition \ref{contraction operator}).
\end{defn}

Since the regularized Poisson kernel $\mathbb{K}_L$ is smooth, $\Delta_L$ is well-defined on $\OO\bracket{\E}$
and can be viewed as a second order differential operator in our infinite dimensional setting.

\begin{defn} We define the effective BV bracket at scale $L$
$$
  \fbracket{-,-}_L: \OO(\E)\times \OO(\E)\to \OO(\E)
$$
by
$$
   \fbracket{\Phi_1, \Phi_2}_L:=\Delta_L\bracket{\Phi_1 \Phi_2}-\bracket{\Delta_L \Phi_1}\Phi_2-(-1)^{|\Phi_1|}\Phi_1\bracket{\Delta_L \Phi_2}, \quad \forall
\Phi_1, \Phi_2 \in \OO(\E).
$$
\end{defn}
As we will see, \BV structures at different scales will be related to each other via the renormalization group flow.

For two distributions $\Phi_1, \Phi_2\in \OO(\E)$, the bracket $\fbracket{\Phi_1, \Phi_2}_L$ will in general diverge as $L\to 0$. However for $\Phi_1, \Phi_2\in \Ol(\E)$,
   $$
      \lim_{L\to 0}\fbracket{\Phi_1, \Phi_2}_L=\fbracket{\Phi_1, \Phi_2},
   $$
   where on the right hand side $\fbracket{-,-}$ is the Poisson bracket as in Lemma/Definition \ref{Poisson-bracket}. Therefore
$\fbracket{-,-}_L$ is a regularization of the classical Poisson bracket.

\subsection{Effective renormalization} We discuss Costello's quantization framework \cite{Kevin-book} in our current set-up.
\subsubsection{Renormalization group flow}
We start from the definition
of graphs:
\begin{defn}
A graph $\gamma$ consists of the following data:
\begin{enumerate}
 \item A finite set of vertices $V(\gamma)$;
 \item A finite set of half-edges $H(\gamma)$;
 \item An involution $\sigma: H(\gamma)\rightarrow H(\gamma)$. The set of fixed points of this map is denoted by $T(\gamma)$ and is
called the set of tails of $\gamma$. The set of two-element orbits is denoted by $E(\gamma)$ and is called the set of internal edges of
$\gamma$;
 \item A map $\pi:H(\gamma)\rightarrow V(\gamma)$ sending a half-edge to the vertex to which it is attached;
 \item A map $g:V(\gamma)\rightarrow \mathbb{Z}_{\geqslant 0}$ assigning a genus to each vertex.
\end{enumerate}
It is clear how to construct a topological space $|\gamma|$ from the above abstract data. A graph $\gamma$ is called $connected$ if
$|\gamma|$ is connected. A graph is called stable if every vertex of genus $0$ is at least trivalent, and every genus $1$ vertex is at
least univalent. The genus of the graph $\gamma$ is defined to be $$g(\gamma):=b_1(|\gamma|)+\sum_{v\in V(\gamma)}g(v),$$ where
$b_1(|\gamma|)$ denotes the first Betti number of $|\gamma|$.
\end{defn}
Let $$\left(\mathcal{O}(\mathcal{E})[[\hbar]]\right)^+\subset\mathcal{O}(\mathcal{E})[[\hbar]]$$ be the subspace consisting of those
functionals  which are at least cubic modulo $\hbar$ and modulo the nilpotent ideal $\mathcal{I}$ in the base ring
$\mathcal{A}_X^\sharp$.  Let $I\in (\mathcal{O}(\mathcal{E})[[\hbar]])^+$ be a functional which can be expanded as
$$I=\sum_{k,i\geq 0}\hbar^k
I_{i}^{(k)}, \quad I_{i}^{(k)}\in \OO^{(i)}(\E).
$$
We view $I_{i}^{(k)}$ as an
$S_i$-invariant linear map
$$
I_{i}^{(k)}: \mathcal{E}^{\otimes i}\rightarrow\mathcal{A}_X^\sharp.
$$
With the propagator $\mathbb{P}_\epsilon^L$, we will describe the $Feynman\ weights$
$$
W_\gamma(\mathbb{P}_\epsilon^L,I)\in (\mathcal{O}(\mathcal{E})[[\hbar]])^+
$$ for any connected stable graph $\gamma$: we label every vertex $v$ in $\gamma$ of genus $g(v)$ and valency $i$ by
$I^{(g(v))}_i$, which we denote by:
$$I_{v}:\mathcal{E}^{\otimes H(v)}\rightarrow\mathcal{A}_X^\sharp,$$
where $H(v)$ is the set of half-edges of $\gamma$ which are incident to $v$.
We label every internal edge $e$ by the propagator $$\mathbb{P}_e=\mathbb{P}_\epsilon^L\in\mathcal{E}^{\otimes H(e)},$$
where $H(e)\subset H(\gamma)$ is the two-element set consisting of the half-edges forming $e$. Now we can contract
$$\otimes_{v\in V(\gamma)}I_v: \mathcal{E}^{H(\gamma)}\rightarrow\mathcal{A}_X^\sharp$$
with $$\otimes_{e\in E(\gamma)}\mathbb{P}_e\in\mathcal{E}^{H(\gamma)\setminus T(\gamma)}\rightarrow\mathcal{A}_X^\sharp$$ to yield a linear map
$$W_\gamma(\mathbb{P}_\epsilon^L, I) : \mathcal{E}^{\otimes T(\gamma)}\rightarrow\mathcal{A}_X^\sharp.$$
We can now define the $renormalization\ group\ flow$ operator:
\begin{defn}
The renormalization group flow operator from scale $\epsilon$ to scale $L$ is the map
$$
W(\mathbb{P}_\epsilon^L,-):\left(\mathcal{O}(\mathcal{E})[[\hbar]]\right)^+\rightarrow\left(\mathcal{O}(\mathcal{E})[[\hbar]]\right)^+
$$
defined by taking the sum of Feynman weights over all stable connected graphs:
$$I\mapsto\sum_{\gamma}\dfrac{1}{|\text{Aut}(\gamma)|}\hbar^{g(\gamma)}W_\gamma(\mathbb{P}_\epsilon^L,I).$$  A collection of functionals
$$\{I[L]\in\left(\mathcal{O}(\mathcal{E})[[\hbar]]\right)^+|L\in\mathbb{R}_+\}$$ is said to satisfy the renormalization group equation (RGE) if for
any $0<\epsilon<L<\infty$, we have
$$I[L]=W(\mathbb{P}_\epsilon^L,I[\epsilon]). $$
\end{defn}

\begin{rmk} Formally,  the RGE can be equivalently described as
$
  e^{I[L]/\hbar}=e^{\hbar {\pa\over \pa \mathbb{P}_\epsilon^L}}e^{I[\epsilon]/\hbar}.
$
\end{rmk}
\subsubsection{Quantum master equation}
Now we explain the quantum master equation as the quantization of the classical master equation. Usually the quantum master
equation is associated with the following operator \cite{Kevin-book} in the \BV formalism
$$
   Q+\hbar \Delta_L,
$$
which can be viewed as a quantization of the differential $Q$.

However, in our case, the above operator does not define a differential due to the curving
$$
  \bracket{ Q+\hbar \Delta_L}^2=l_1^2.
$$
We will modify the construction in \cite{Kevin-book} to incorporate with the curving.

\begin{defn}
We define the effective curved differential $Q_L: \E \to \E$ by
$$
  Q_L:= Q+ l_1^2\int_0^L Q^{GF}e^{-tH}dt,
$$
where $l_1^2\int_0^L Q^{GF}e^{-tH}dt$ is the composition of the operator $\int_0^L Q^{GF}e^{-tH}dt$ with $l_1^2$.
\end{defn}

It is straightforward to prove the following lemma:
\begin{lem}\label{QME-RG}
The quantized operator ${Q_L+\hbar \Delta_L+{F_{l_1}\over \hbar}}$ is compatible with the renormalization group flow in the following sense (recall Lemma \ref{lem:classical-master-equation} for the definition $F_{l_1}$)
$$
   e^{\hbar {\pa\over \pa \mathbb{P}_\epsilon^L}} \bracket{Q_\epsilon+\hbar \Delta_\epsilon+{F_{l_1}\over \hbar}}  = \bracket{Q_L+\hbar
\Delta_L+{F_{l_1}\over \hbar}} e^{\hbar {\pa\over \pa \mathbb{P}_\epsilon^L}}.
$$
Moreover, it squares zero modulo $\A_X^\sharp$:
$$
\bracket{Q_L+\hbar \Delta_L+{F_{l_1}\over \hbar}}^2=C,
$$
equals the multiplication by some $C\in \A_X^\sharp$.
\end{lem}

Therefore  we will use ${Q_L+\hbar \Delta_L+{F_{l_1}\over \hbar}}$ instead of $Q+\hbar\Delta_L$ in order to define the quantum
master equation. The
constant $C$ does not bother us since the perturbative quantization in \cite{Kevin-book} is defined modulo the constant terms.
Precisely,

\begin{defn}\label{defn-quantization}
Let $\{I[L]\in\left(\mathcal{O}(\mathcal{E})[[\hbar]]\right)^+|L\in\mathbb{R}_+\}$ be a collection of effective interactions which satisfies the renormalization group
equation. We say that they satisfy the quantum master equation if for all $L>0$ the following scale $L$ quantum mater equation (QME) is satisfied:
\begin{equation}\label{eqn:qunatum-master-equation}
\bracket{Q_L+\hbar \Delta_L+{F_{l_1}\over \hbar}}e^{I[L]/\hbar}=R e^{I[L]/\hbar},
\end{equation}
where $R\in \A_X^\sharp[[\hbar]]$ does not depend on $L$.
\end{defn}

In other words, if we view ${Q_L+\hbar \Delta_L+{F_{l_1}\over \hbar}}$ as defining a projective flat connection, then a solution of quantum master equation defines a projectively flat section.

\begin{lem} The quantum master equation is compatible with the renormalization group flow in the following sense: if the collection $\{I[L]|L\in\R_+\}$ satisfies QME at some
scale $L_0>0$, then QME holds for any scale.
\end{lem}
\begin{proof} This follows from Lemma \ref{QME-RG}.

\end{proof}

\begin{lem}\label{quantum-BRST}
Suppose $I[L]$ satisfies the quantum master equation at scale $L>0$, then
$
    Q_L+\hbar \Delta_L + \{I[L],-\}_L
$
defines a square-zero operator on $\OO(\E)[[\hbar]]$.
\end{lem}
\begin{proof} Let $U_L=Q_L+\hbar \Delta_L + \{I[L],-\}_L$ and $\Phi\in \OO(\E)[[\hbar]]$. Then
$$
   \bracket{Q_L+\hbar \Delta_L+{F_{l_1}\over \hbar}}\bracket{\Phi e^{I[L]/\hbar}}=\bracket{U_L(\Phi)+R\Phi } e^{I[L]/\hbar}.
$$
Applying $Q_L+\hbar \Delta_L+{F_{l_1}\over \hbar}$ again to both sides, we find
$$
   C\Phi  e^{I[L]/\hbar}=\bracket{U_L^2(\Phi)+U_L(R\Phi )+R (U_L(\Phi)+R \Phi)}e^{I[L]/\hbar}.
$$
Set $\Phi=1$, we find $C=d_XR+R^2$, while $R^2=0$ since $R$ is a 1-form. Here $d_X$ is the de Rham differential on $X$. On the other hand,
$$
U_L(R \Phi)= (d_XR) \Phi-RU_L(\Phi).
$$
Comparing the two sides of the above equation, we get $U_L^2(\Phi)=0$ as desired.
\end{proof}

\begin{rmk}  From the above proof, we find the following compatibility equation: $C=d_XR$. It is not hard to see that the two form
$C$ is given by the contraction between $F_{l_1}$ and $\Delta_L$, describing the curvature $l_1^2$. In fact $C$ represents
$(2-2g)c_1(X)$. The compatibility equation says that $C$ is an exact form, implying that the Calabi-Yau condition is necessary for
the quantum consistency (if $g\neq 1$). In section \ref{subsection:genuine-quantization}, we will show that Calabi-Yau condition
is also sufficient for anomaly cancellation. Such phenomenon arising from the curved $L_\infty$-algebra  does not play a role in
\cite{Kevin-CS}, but we expect that it would appear in general.

\end{rmk}

It is easy to see that the leading $\hbar$-order of the quantum master equation reduces to  the classical master equation when $L\to 0$. Therefore the square-zero operator
$Q_L+\hbar \Delta_L + \{I[L],-\}_L$ defines a quantization of the classical $Q+\{I_{cl},-\}$.

\subsubsection{$\C^\times$-symmetry}\label{section:classical-symmetry}
Later, when we study quantum theory of B-twisted $\sigma$-model, we will be interested in quantizations which preserve certain
symmetries we describe now: we consider an action of $\mathbb{C}^\times$ on $\mathcal{E}$ as follows:
$$
\lambda \cdot(\alpha_1\otimes\mathfrak{g}_1+\alpha_2\otimes\mathfrak{g}_2^\vee):=\alpha_1\otimes\mathfrak{g}_1+\lambda^{-1}
\alpha_2\otimes\mathfrak{g}_2^\vee,\ \lambda\in\mathbb{C}^\times.
$$
\begin{defn}
We define an action of $\mathbb{C}^\times$ on $\mathcal{O}(\mathcal{E})((\hbar))$ by
$$(\lambda \cdot
(\hbar^k F))(v):=\lambda^{k}\hbar^kF(\lambda^{-1}\cdot v), F\in\mathcal{O}(\mathcal{E}),v\in\mathcal{E}.
$$
\end{defn}
It is obvious that the classical interaction ${I_{cl}/\hbar}$ is invariant under this action. The following lemma can be proved by straightforward calculation, which we omit:
\begin{lem}\label{lem:C-action}
 The following operations are equivariant under the action of $\mathbb{C}^\times$:
\begin{enumerate}
 \item The renormalization group flow operator:
$W(\mathbb{P}_\epsilon^L,-):(\mathcal{O}(\mathcal{E})[[\hbar]])^+\rightarrow(\mathcal{O}(\mathcal{E})[[\hbar]])^+$,
 \item The differential $Q:\mathcal{O}(\mathcal{E})[[\hbar]]\rightarrow\mathcal{O}(\mathcal{E})[[\hbar]]$,
 \item The quantized differential $Q_L+\hbar\Delta_L+{\hbar^{-1}F_{l_1}}$,
 \item The BV bracket
$\{-,-\}_L:\mathcal{O}(\mathcal{E})[[\hbar]]\otimes_{\A_X^\sharp[[\hbar]]}\mathcal{O}(\mathcal{E})[[\hbar]]\rightarrow\mathcal{O}
(\mathcal{E} )[[\hbar]]$, for all $L>0$.
\end{enumerate}
\end{lem}
The following proposition describes those functionals in $\mathcal{O}(\mathcal{E})[[\hbar]]$ that are $\mathbb{C}^\times$-invariant.
\begin{prop}\label{prop:invariant-actions}
 Let $I=\sum_{i\geq 0}I^{(i)}\cdot \hbar^i\in\mathcal{O}(\mathcal{E})[[\hbar]]$. If $I$ is invariant under the
$\mathbb{C}^\times$
action, then $I^{(i)}=0$ for $i>1$, and furthermore $I^{(1)}$ lies in the subspace
$$\mathcal{O}(\mathcal{A}(\Sigma_g)\otimes\mathfrak{g}_X[1])\subset\mathcal{O}(\mathcal{E}).$$
\end{prop}
\begin{proof}
The hypothesis that $I$ is invariant implies that $I^{(i)}$ has weight $-i$ under the $\mathbb{C}^\times$ action. Notice that the
weight of the $\mathbb{C}^\times$ action on $\mathcal{O}(\mathcal{E})$ can be only $-1$ or $0$, which implies the first statement. The
second
statement of the proposition is obvious.
\end{proof}

\subsection{Quantization}
We study the quantization of B-twisted $\sigma$-model in this section. Firstly, let us recall the definition  of
perturbative quantization of classical field theories in the \BV formalism in \cite{Kevin-book}:
\begin{defn}
 Let $I\in\mathcal{O}_{loc}(\mathcal{E})$ be a classical interaction functional satisfying the classical master equation.  A quantization of the classical field theory defined by
$I$ consists of a collection
$\{I[L]\in(\mathcal{O}(\mathcal{E})[[\hbar]])^+|L\in\mathbb{R}_+\}$ of effective functionals such that
\begin{enumerate}
 \item The renormalization group equation is satisfied,
 \item  The functional $I[L]$ must satisfy a locality axiom, saying that as $L\rightarrow 0$ the functional $I[L]$ becomes more and
more local,
 \item The functional $I[L]$ satisfies the scale $L$ $quantum\ master\ equation$ (\ref{eqn:qunatum-master-equation}),
 \item Modulo $\hbar$, the $L\rightarrow 0$ limit of $I[L]$ agrees with the classical interaction functional $I$.
\end{enumerate}

\end{defn}
The strategy for constructing a quantization of a given classical action functional is to run the renormalization group flow from
scale $0$ to scale $L$. In other words, we try to define the effective interaction $I[L]$ as the following
limit $$\lim_{\epsilon\rightarrow 0}W(\mathbb{P}_\epsilon^L,I).$$ However, the
above limit in general does not exist. Then the technique of counter terms solves the problem: after the choice of a
$Renormalization\ Scheme$, there is a unique set of counter terms $I^{CT}(\epsilon)\in(\mathcal{O}_{loc}(\mathcal{E})[[\hbar]])^+$ such that the limit
\begin{equation}\label{eqn:naive-quantization}
\lim_{\epsilon\rightarrow
0}W(\mathbb{P}_\epsilon^L,I-I^{CT}(\epsilon))\in(\mathcal{O}(\mathcal{E})[[\hbar]])^+
\end{equation}
 exists. For more details on Renormalization Scheme and counter terms, we refer
the readers to \cite{Kevin-book}. It is then natural to define the $naive\ quantization$ $I_{naive}[L]$ of $I$ to be the limit in
equation (\ref{eqn:naive-quantization}). For B-twisted $\sigma$-model, we calculate the naive quantization in section
\ref{subsection:naive-quantization}. In particular, we show that the counter terms for our theory actually vanish.

The naive quantization $\{I_{naive}[L]| L>0\}$ automatically satisfies the renormalization group equation and the locality axiom by
construction. However, it does not satisfy the quantum master equation in general. In order to find the genuine quantization, we
need to analyze the possible cohomological obstruction to solving the QME, and correct the naive
quantization $\{I_{naive}[L]|L>0\}$ term by term in $\hbar$ accordingly if the obstruction vanishes. The
$\mathbb{C}^\times$ symmetry of the classical interaction $I_{cl}$ simplifies this computation to one-loop anomaly, and in Appendix
\ref{appendix:one-loop-anomaly} we give
a formula for one-loop anomaly for general field theories. This formula, when specialized to B-twisted $\sigma$-model, shows that
the condition for anomaly cancellation is exactly the Calabi-Yau condition of the target $X$. This is done in section
\ref{subsection:genuine-quantization}.  In section \ref{subsection:one-loop-correction}, we give an explicit formula for the
one-loop correction of the naive quantization when $X$ is Calabi-Yau.
\begin{rmk}
 In later sections, we will give the details of the analysis for  Riemann surfaces of genus $g>1$, the study of $\mathbb{P}^1$ and
elliptic curves are similar and actually easier which we omit.
\end{rmk}

\subsubsection{The naive quantization}\label{subsection:naive-quantization}
Let $I_{cl}$ denote the classical interaction of B-twisted $\sigma$-model. We will show that the following limit exists for all $L>0$:
\begin{equation}\label{eqn:RG flow}
\lim_{\epsilon\rightarrow 0}\ W(\mathbb{P}_\epsilon^L,I_{cl}).
\end{equation}
 The following simple observation simplifies the analysis greatly: for any $L>\epsilon>0$ and any graph $\gamma$, the associated Feynman weight
$\frac{\hbar^{|g(\gamma)|}}{|\text{Aut}(\gamma)|}W_\gamma(\mathbb{P}_\epsilon^L,I_{cl})$ is invariant
under the $\mathbb{C}^\times$ action defined in section
\ref{section:classical-symmetry}, by Lemma \ref{lem:C-action} and the fact that $I_{cl}/\hbar$ is $\mathbb{C}^\times$-invariant. By
Proposition
\ref{prop:invariant-actions}, we have $$W_\gamma(\mathbb{P}_\epsilon^L,I_{cl})=0$$ for
those stable graphs $\gamma$ with genus greater than $1$. Thus we only need to consider Feynman weights for trees and one-loop
graphs. For any classical interaction $I$, the limit (\ref{eqn:RG flow}) always exists  for trees, but not necessarily for
one-loop graphs. Fortunately, for the classical interaction $I_{cl}$ of the B-twisted $\sigma$-model, the limit (\ref{eqn:RG
flow}) exists.
\begin{lem-defn}\label{lem:naive quantization}
Let $\gamma$ be a graph of genus $1$, then the following limit exists: $$ \lim_{\epsilon\rightarrow
0}W_\gamma(\mathbb{P}_\epsilon^L,I_{cl}).$$
We define the naive quantization at scale $L$ to be
$$
   I_{naive}[L]:=\lim_{\epsilon\to 0}W(\mathbb{P}_\epsilon^L,I_{cl}).
$$
\end{lem-defn}

\begin{rmk}
As discussed in Definition \ref{defn:propagator}, the propagator $\mathbb{P}_\epsilon^L$ consists of an analytic part and a combinatorial part. It is clear that
only the analytic part is relevant concerning the convergence issue. In the following, we will use the notation $W(P_\epsilon^L,I_{cl})$ to denote the
analytic part of the RG flow $W(\mathbb{P}_\epsilon^L,I_{cl})$, whose inputs are only differential forms on $\Sigma_g$.
We will also use similar notations replacing $\mathbb{K}_\epsilon$ by $K_\epsilon$ later.
\end{rmk}
\begin{proof}
Since $I_{cl}\in(\mathcal{O}(\E))^+\subset(\mathcal{O}(\E)[[\hbar]])^+$, we only need to consider those genus $1$ graphs $\gamma$
whose vertices are all of genus $0$, i.e $b_1(\gamma)=1$.  Such a graph is called a $wheel$ if it can not be disconnected by
removing a single edge. Every graph with first Betti number $1$ is
a wheel with trees attached on it. Since trees do not contribute any divergence, we only need to prove
the lemma for wheels. Let $\gamma$ be a wheel with $n$ tails, and
let $\alpha_1\otimes\mathfrak{g}_1,\cdots,\alpha_n\otimes\mathfrak{g}_n\in\mathcal{E}$ be inputs on the tails. If the valency of
some vertices of $\gamma$ is greater than 3, we can combine the analytic part of the inputs on the tails that are attached to the same
vertex. More precisely, the convergence property of the following two Feynman weights are the same:
\begin{equation*}
\figbox{0.2}{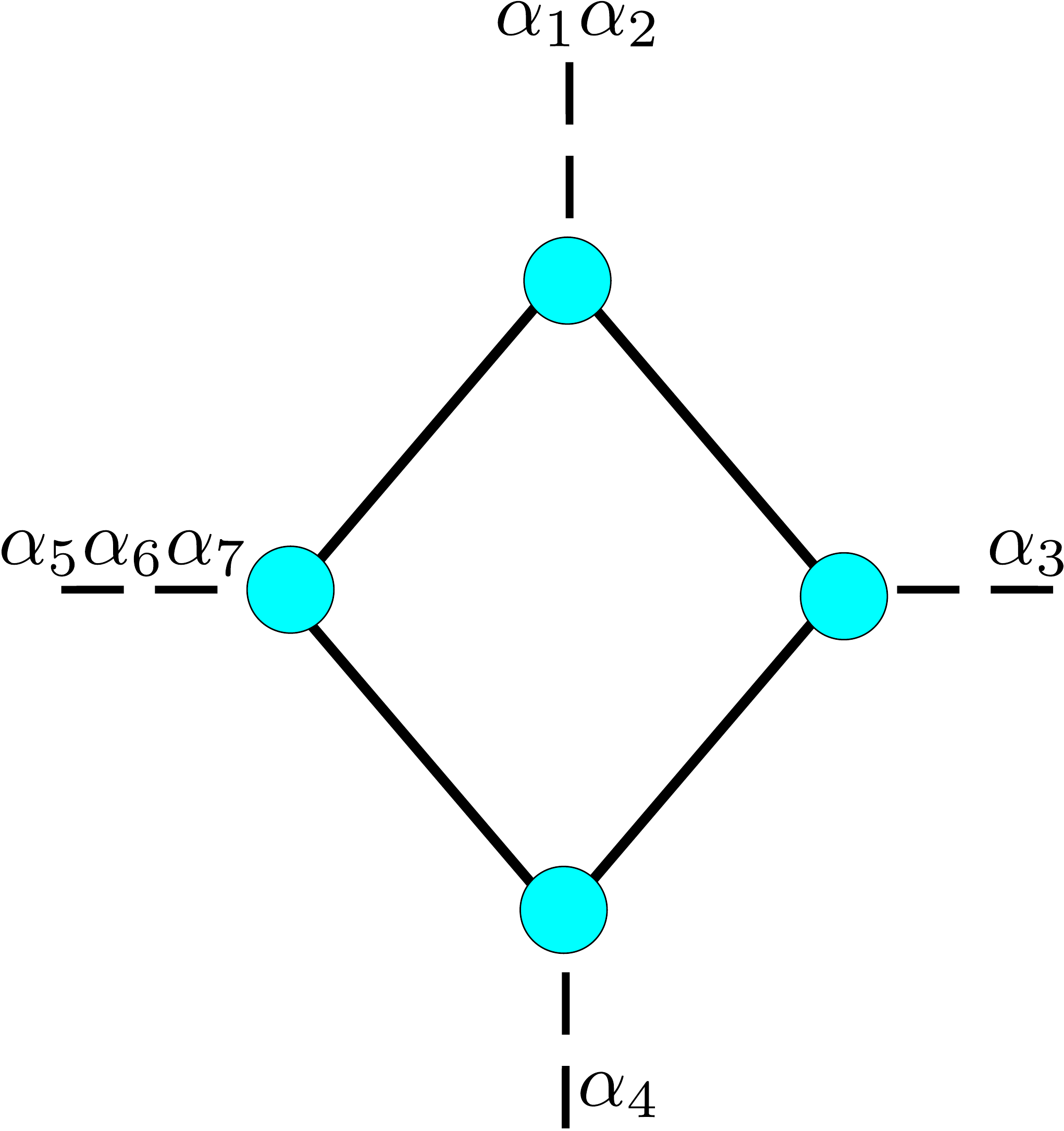}\hspace{30mm}\figbox{0.2}{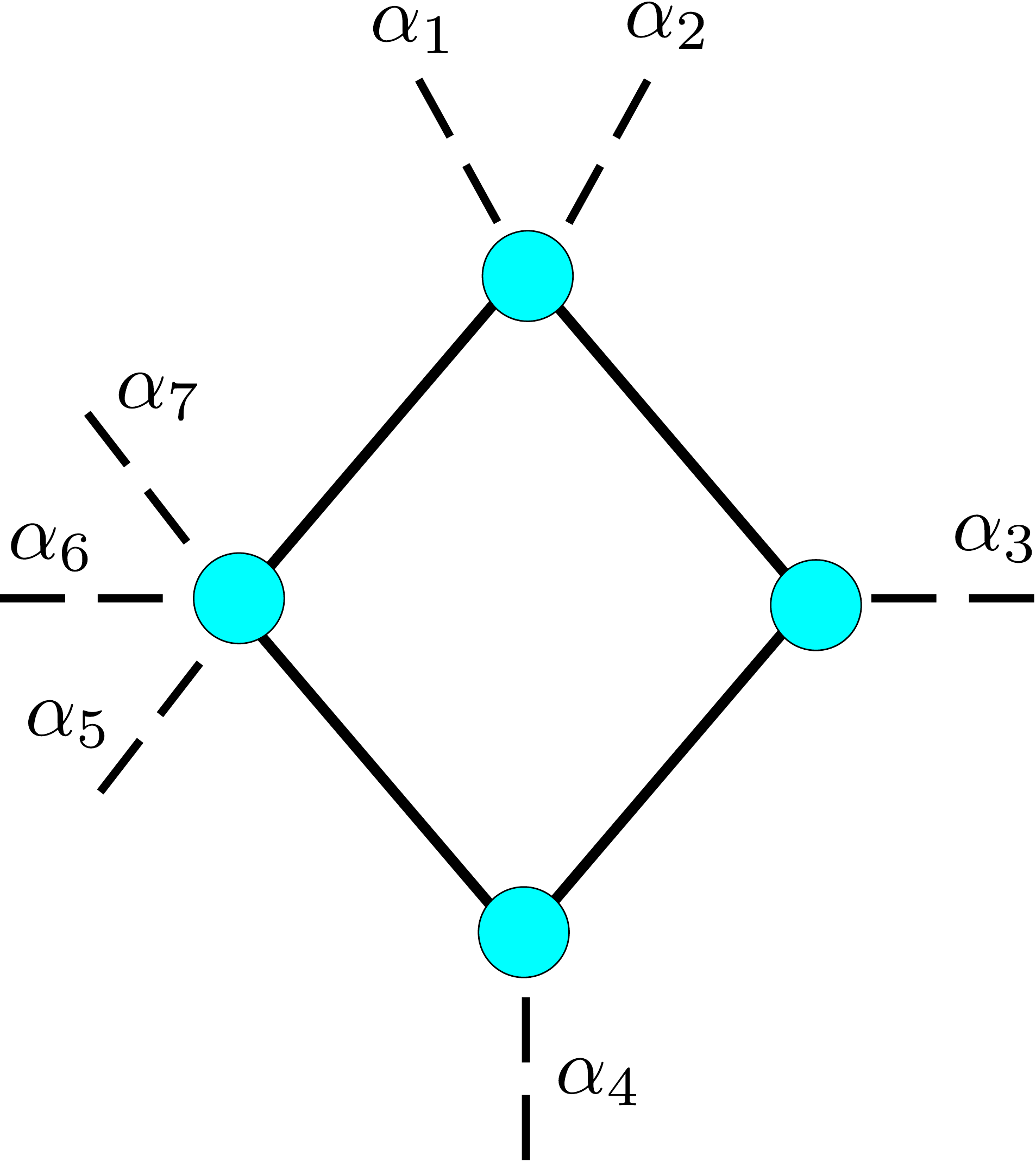}
\end{equation*}
Thus the proof of the lemma can be further reduced to trivalent wheels. Let $\gamma$ be a trivalent wheel with $n$ vertices, we
prove the lemma for the three
possibilities:

(1) $n=1$: in this case, the graph $\gamma$ contains  a self-loop, and the Feynman weight is given by
$$
W_\gamma(P_\epsilon^L,I_{cl})(\alpha)=\int_{t=\epsilon}^Ldt\int_{z_1\in\Sigma_g}d^*K_t(z_1,z_1)\alpha.
$$
Let the Riemann surface $\Sigma_g$ be of the form $\Sigma_g=\H/\Gamma$, where $\Gamma$ is a subgroup of isometry acting discretely on $\H$. Let $k_t$ denote
the heat kernel on $\H$, and let $\pi$ denote the natural projection $\H\rightarrow \Sigma_g$. The heat kernel on $\Sigma_g$ can be written as:
$$
K_t(\pi(x_1),\pi(x_2))=\sum_{g\in\Gamma}k_t(x_1,g\cdot x_2),
$$
from which it is clear that $K_t$ is regular along the diagonal in $\Sigma_g\times\Sigma_g$: we can pick $x_1=x_2$ in the above identity. If $g=\text{id}$, then $k_t(x_1,x_1)$
vanishes by Lemma \ref{lem:asymp-propagator}, otherwise $k_t(x_1,g\cdot x_1)$ is automatically regular since heat kernel is singular only along the diagonal but $x_1\not=g\cdot
x_1$.

(2) $n\geq 3$: the Feynman weight is given explicitly by:
\begin{equation}\label{eqn:Feynman-weight}
\begin{aligned}
&W_\gamma(P_\epsilon^L,I_{cl})(\alpha_1,\cdots,\alpha_n)\\
=&\int_{z_1,\cdots,z_n\in\Sigma_g}P_\epsilon^L(z_1,z_2)
P_\epsilon^L(z_2,z_3)\cdots
P_\epsilon^L(z_n,z_1) \alpha_1(z_1,\bar{z}_1)\cdots\alpha_n(z_n,\bar{z}_n)\\
=&\int_{t_1,\cdots,t_n=\epsilon}^Ldt_1\cdots
dt_n\int_{z_1,\cdots,z_n\in\Sigma_g}\left(d^*K_{t_1}(z_1,z_2)\right)\cdots\left(d^*K_{ t_n } (z_n ,
z_1)\right)\alpha_1(z_1,\bar{z}_1)\cdots\alpha_n(z_n,\bar{z}_n).
\end{aligned}
\end{equation}
Using the same argument as in case (1), there is no difference if we  replace $\Sigma_g$ in equation (\ref{eqn:Feynman-weight}) by
$\mathbb{H}$ concerning the convergence property:
\begin{equation}\label{eqn:Feynman-weight-H}
 \int_{t_1,\cdots,t_n=\epsilon}^Ldt_1\cdots
dt_n\int_{z_1,\cdots,z_n\in\mathbb{H}}\left(d^*k_{t_1}(z_1,z_2)\right)\cdots\left(d^*k_{ t_n }
(z_n,z_1)\right)\alpha_1(z_1,\bar{z}_1)\cdots\alpha_n(z_n,\bar{z}_n).
\end{equation}
For simplicity, we keep the notation for the inputs $\alpha_1,\cdots,\alpha_n$ which are now differential forms on $\mathbb{H}$ with compact support.
We can write the integral (\ref{eqn:Feynman-weight-H}) as the sum of the following integrals where $\sigma$ runs over the symmetric group $S_n$:
\begin{equation}\label{eqn:Feynman-weight-sum} \int_{\epsilon\leqslant t_{\sigma(1)}\leqslant\cdots\leqslant t_{\sigma(n)}\leqslant
L}dt_1\cdots
dt_n\int_{z_1,\cdots,z_n\in\mathbb{H}}\left(d^*k_{t_1}(z_1,z_2)\right)\cdots\left(d^*k_{ t_n }
(z_n,z_1)\right)\alpha_1(z_1,\bar{z}_1)\cdots\alpha_n(z_n,\bar{z}_n).
\end{equation}

We will show that the integral (\ref{eqn:Feynman-weight-sum})  converges as $\epsilon\rightarrow 0$ for $\sigma=\text{id}\in S_n$, the proof for other
permutations $\sigma$ are the same. Let $(z_1,\cdots,z_n)=(x_1+iy_1,\cdots,x_n+iy_n)$ denote the standard coordinates on
$\mathbb{H}\times\cdots\times\mathbb{H}$. By Lemma \ref{lem:asymp-propagator}, the leading term of $d^*k_{t_k}(z_{k},z_{k+1})$ is
of the form
\begin{equation}\label{eqn:propagator-leading-term}
\begin{aligned}
\dfrac{1}{t_k^2}e^{-\frac{\rho^2(z_{k},z_{k+1})}{4t_k}}\Big(&Q_1(z_{k},\bar{z}_k;z_{k+1},\bar{z}_{k+1})(x_{k+1}-x_{k})(dy_{k+1}-dy_{k}
)\\
-&Q_2(z_{k},\bar{z}_k;z_{k+1},\bar{z}_{k+1})(y_{k+1}-y_{k})(\frac{dx_{k+1}}{y_{k+1}}-\frac{dx_{k}}{y_k})\Big),
\end{aligned}
\end{equation} where $Q_1$ and $Q_2$ are smooth functions on $\mathbb{H}\times\mathbb{H}$. To show the convergence of the integral
(\ref{eqn:Feynman-weight-sum}) as $\epsilon\rightarrow 0$, we will apply Wick's lemma and show that the integral of the leading term  in
 (\ref{eqn:Feynman-weight-sum}) converges. The higher order terms have better convergence property.

 Without loss of generality, we can assume that the supports of $\alpha_i's$ lie in a small neighborhood of the small diagonal
$\Delta=\{(z,\cdots,z): z\in \mathbb{H}\}$ of $\mathbb{H}\times \cdots \times \mathbb{H}$. Otherwise we can choose a cut-off
function supported around $\Delta$ and split the integral into parts of the desired form. We consider the following change of
coordinates:  let
$w_0=(u_0,v_0)=(x_1,y_1)\in\mathbb{H}$ and let $w_k=(u_k,v_k)\in\mathbb{R}^2$ be the Riemann normal coordinate of the point
$(x_{k+1},y_{k+1})$ with center $(x_k,y_k)$ for $1\leqslant k\leqslant n-1$ such that on
$\Delta_k:=\{(z_1,\cdots,z_n)\in\mathbb{H}\times\cdots\times\mathbb{H}: z_k=z_{k+1}\}$, we have
\begin{equation}\label{eqn:change-coordinates-derivative}
 \dfrac{\partial(x_{k+1}-x_k)}{\partial u_k}\bigg|_{\Delta_k}=\dfrac{\partial(y_{k+1}-y_k)}{\partial v_k}\bigg|_{\Delta_k}=\dfrac{1}{y_k},
\hspace{5mm} \dfrac{\partial(x_{k+1}-x_k)}{\partial v_k}\bigg|_{\Delta_k}=\dfrac{\partial(y_{k+1}-y_k)}{\partial u_k}\bigg|_{\Delta_k}=0.
\end{equation}

By the definition of Riemann normal coordinates, the geodesic distance between $z_{k}$ and $z_{k+1}$ is
$\rho(z_{k},z_{k+1})=(u_k^2+v_k^2)^{\frac{1}{2}}=||w_k||$ when they are close.
It is obvious that there are smooth positive functions $\{\phi_k,\psi_k,1\leqslant k\leqslant n\}$ on
$\mathbb{H}\times\mathbb{R}^{2n-2}$ such that
\begin{equation}\label{eqn:change-coordinates-estimate}
\begin{aligned}
&|x_{k+1}-x_{k}|\leqslant \phi_k\cdot||w_k||,\hspace{9mm}|y_{k+1}-y_{k}|\leqslant \psi_k\cdot||w_k||,\hspace{5mm}\text{for}\ \
1\leqslant k\leqslant n-1\\
&|x_n-x_1|\leqslant\phi_n\cdot(\sum_{k=1}^{n-1}||w_k||),\hspace{3mm}|y_n-y_1|\leqslant
\psi_n\cdot(\sum_{k=1}^{n-1}||w_k||).\hspace{9mm}
\end{aligned}
\end{equation}

With the above preparation, we are now ready to show  the convergence of the integral of the leading term in
(\ref{eqn:Feynman-weight-sum}). After plugging (\ref{eqn:propagator-leading-term}) into (\ref{eqn:Feynman-weight-sum}) and
using the estimate (\ref{eqn:change-coordinates-estimate}), it is not difficult to see that there is a smooth positive function
$\Phi$ on $\mathbb{H}\times\mathbb{R}^{2n-2}$, such that the integral (\ref{eqn:Feynman-weight-sum}) with $\sigma=\text{id}$ is bounded above in absolute value
by:
\begin{align*}
&\int_{\epsilon\leqslant t_1\leqslant\cdots\leqslant
t_n\leqslant L}\prod_{i=1}^{n}dt_i\int_{w_0\in\mathbb{H}}\int_{w_1\cdots,w_{n-1}\in\mathbb{R} ^2 }
\Phi(w_0,\cdots,w_n)\cdot\left(\prod_{i=1}^{n-1}\dfrac{||w_i||}{t_i^2}e^{-\frac{||w_i||^2}{4t_i}}\right)\\
&\hspace{65mm}\cdot\dfrac{||w_1||+\cdots+||w_{n-1}||}{
t_n^2}\cdot e^{-\frac{\rho^2(z_n,z_1)}{4t_n} }\prod_{i=0}^{n-1}d^2w_i\\
\leqslant&\int_{\epsilon\leqslant
t_1\leqslant\cdots\leqslant
t_n\leqslant L}\prod_{i=1}^{n}dt_i\int_{w_0\in\mathbb{H}}\int_{w_1\cdots,w_{n-1}\in\mathbb{R}^2}\Phi(w_0,\cdots,w_n)\cdot\left(\prod_{
i=1}^{n-1}\dfrac{||w_i||}{t_i^2}e^{-\frac{||w_i||^2}{4t_i}}\right)\\
&\hspace{65mm}\cdot\dfrac{||w_1||+\cdots+||w_{n-1}||}{t_n^2}\prod_{i=0}
^{n-1}d^2w_i.
\end{align*} The  inequality follows simply by dropping the term $ e^{-\frac{\rho^2(z_n,z_1)}{4t_n}}$, and the function $\Phi$
arises as the product of absolute value of the following functions or differential forms:
\begin{enumerate}
 \item the functions $\phi_k$ and $\psi_k$ in  (\ref{eqn:change-coordinates-estimate});
 \item The Jacobian of the change from the standard coordinates to the Riemann normal coordinates;
 \item The functions $Q_1,Q_2$ in (\ref{eqn:propagator-leading-term});
 \item The inputs on the tails of the wheel.
\end{enumerate}

From (4) it is clear that  we can choose $\Phi$ with compact support. Thus we only need to show that the
following integral is convergent:
\begin{equation}\label{eqn:estimate-integral-final-form}
 \int_{\epsilon\leqslant
t_1\leqslant\cdots\leqslant
t_n\leqslant L}\prod_{i=1}^{n}dt_i\int_{w_1\cdots,w_{n-1}\in\mathbb{R}^2}\left(\prod_{
i=1}^{n-1}\dfrac{||w_i||}{t_i^2}e^{-\frac{||w_i||^2}{4t_i}}\right)
\cdot\dfrac{||w_1||+\cdots+||w_{n-1}||}{t_n^2}\prod_{i=1}
^{n-1}d^2w_i.
\end{equation}

We can further change the coordinates: let $$\xi_k=w_k\cdot t_k^{-\frac{1}{2}}, \hspace{5mm}1\leqslant k\leqslant n-1.$$ Then
 (\ref{eqn:estimate-integral-final-form}) becomes
\begin{align*}
\int_{\epsilon\leqslant t_1\leqslant\cdots\leqslant
t_n\leqslant L}\prod_{i=1}^{n-1}dt_i\int_{\xi_1\cdots,\xi_{n-1}\in\mathbb{R}^2}\big(\prod_{
i=1}^{n-1}\dfrac{||\xi_i||}{t_i^{1/2}}e^{-\frac{||\xi_i||^2}{4}}\big)\cdot\dfrac{||\xi_1||t_1^{1/2}+\cdots+||\xi_{n-1}||t_{n-1}^{
1/2 }
} { t_n^2 } \prod_ { i=1 }
^{n-1}d^2\xi_i,
\end{align*}
which is bounded above by
$$\left(\int_{\epsilon\leqslant
t_1\leqslant\cdots\leqslant t_n\leqslant L}\big(\prod_{i=1}^{n-1}t_i^{-\frac{1}{2}}\big)t_n^{-\frac{3}{2}}\prod_{i=1}^{n}
dt_i\right)\cdot\left(\int_{\xi_1,\cdots,\xi_{n-1}\in\mathbb{R}^2}P(||\xi_i||)e^{-\sum_{i=1}^{n-1}||\xi_i||^2/4}\right),$$ where
$P(||\xi_i||)$ is a polynomial of $||\xi_i||$'s. It is not difficult to see that the first integral converges as
$\epsilon\rightarrow0$ when
$n\geqslant 3$, and that the second integral is finite.

(3) $n=2$:  Plugging the leading term of
(\ref{eqn:propagator-leading-term}) into the integral (\ref{eqn:Feynman-weight-H}) for $n=2$, we can see that the integral of
the leading term is of the following form:
\begin{equation}\label{eqn:integral}
 \begin{aligned}
\dfrac{1}{t_1^2t_2^2}\int_{t_1,t_2=\epsilon}^Ldt_1dt_2\int_{(u_0,v_0)\in\mathbb{H}}\int_{(u_1,v_1)\in\mathbb{R}^2}&\Phi(u_0,v_0,u_1,
v_1)(x_1-x_2)(y_1-y_2)\\
&\exp\left(-(u_1^2+v_1^2)(\frac{1}{t_1}+\frac{1}{t_2})\right)du_0dv_0du_1dv_1,
 \end{aligned}
\end{equation}
where $\Phi$ is similar to that in the case where $n\geq 3$. The fact that the functions $(x_1-x_2)^2$ and $(y_1-y_2)^2$ do not show up in
equation (\ref{eqn:integral}) follows from the trivial observation that $(dy_1-dy_2)^2=(\frac{dx_1}{y_1}-\frac{dx_2}{y_2})^2=0$. This simple fact, together
with the derivatives of $x_1-x_2$ and $y_1-y_2$ in (\ref{eqn:change-coordinates-derivative}) implies that the leading term in  Wick's expansion of
the integral
$$\dfrac{1}{t_1^2t_2^2}
\int_{(u_1,v_1)\in\mathbb{R}^2}\Phi(u_0,v_0,u_1,v_1)(x_1-x_2)(y_1-y_2)\exp\left(-(u_1^2+v_1^2)(\frac{1}{t_1}+\frac{1}{t_2}
)\right)du_1dv_1$$
 is given by a multiple of
$$\dfrac{1}{t_1^2t_2^2}
\int_{(u_1,v_1)\in\mathbb{R}^2}u_1^2 v_1^2 \exp\left(-(u_1^2+v_1^2)(\frac{1}{t_1}+\frac{1}{t_2}
)\right)du_1dv_1\propto \dfrac{t_1t_2}{(t_1+t_2)^3}.
$$
%$$\left(\dfrac{\partial\Phi}{\partial u_1}\dfrac{\partial\Phi}{\partial
%v_1}\dfrac{1}{y_1^2}\right)\Bigg|_{u_1=v_1=0}\cdot\dfrac{1}{t_1^2t_2^2}\cdot\left(\dfrac{t_1t_2}{t_1+t_2}\right)^3=\left(\dfrac{
%\partial\Phi}{\partial u_1}\dfrac{\partial\Phi}{\partial
%v_1}\dfrac{1}{y_1^2}\right)\Bigg|_{u_1=v_1=0}\cdot\dfrac{
%t_1t_2 } { (t_1+t_2)^3 }.
%$$

The integral of $\dfrac{t_1t_2}{(t_1+t_2)^3}$
on $[\epsilon,L]\times[\epsilon,L]$ clearly converges as $\epsilon\rightarrow 0$. Furthermore, since $\Phi$ has compact support on
$\mathbb{H}\times\mathbb{R}^2$, it is clear that  (\ref{eqn:integral}) converges as $\epsilon\rightarrow 0$.
\end{proof}

\subsubsection{Obstruction analysis} By construction, the naive quantization $\{I_{naive}[L]|L\in\R_+\}$ satisfies all requirements of a quantization except for
the quantum master equation. In general, there exist potential obstructions to solving quantum master equation known as the
anomaly. The analysis of such obstructions is usually very difficult. In \cite{Kevin-book}, Costello has developed a convenient
deformation theory to deal with this problem, which we will follow to compute the obstruction space of the B-twisted
$\sigma$-model.

Recall that $Ob=\bracket{\Ol^+(\E), Q+\fbracket{I_{cl},-}}$ is the deformation-obstruction complex of our theory. Costello's deformation method says that
$$
   H^1(Ob)
$$
is the obstruction space for solving quantum master equation, and
$$
  H^0(Ob)
$$
parametrizes the deformation space. Both cohomology groups can be computed via $D$-module techniques. In our case, we can restrict to a subcomplex of $Ob$,
thanks to the $\C^\times$-symmetry.

\begin{defn} We define $\tilde{\E}\subset \E$ to be the subspace
$$
  \tilde{\E}:=\A_{\Sigma_g}\otimes \g_X[1]
$$
and $\widetilde{Ob}$ to be the reduced deformation-obstruction complex
$$
   \widetilde{Ob}:=\bracket{\Ol^+\bracket{\tilde{\E}}, Q+\fbracket{I_{cl},-}}.
$$
\end{defn}

\begin{prop} The obstruction space for solving quantum master equation with prescribed $\C^\times$-symmetry is $H^1\bracket{\widetilde{Ob}}$.
\end{prop}
\begin{proof} This is the same as the holomorphic Chern-Simons theory in \cite{Kevin-CS}.
\end{proof}

To describe the complex $\widetilde{Ob}$, we first introduce some notations. Let
$$
  \Jet_{\Sigma_g}\bracket{\tilde{\E}}:=\Jet_{\Sigma_g}\bracket{\A_{\Sigma_g}}\otimes \g_X[1]
$$
be the sheaf of smooth jets of differential forms on $\Sigma_g$ valued in $\g_X[1]$, and let $D_{\Sigma_g}$ be the sheaf of smooth differential operators on
$\Sigma_g$. $\Jet_{\Sigma_g}\bracket{\tilde{\E}}$ is naturally a $D_{\Sigma_g}$-module, and we define its dual
$$
  \Jet_{\Sigma_g}\bracket{\tilde{\E}}^\vee:=\Hom_{\cinfty(\Sigma_g)\otimes \A_{X}^{\sharp}}\bracket{\Jet_{\Sigma_g}\bracket{\tilde{\E}},
\cinfty(\Sigma_g)\otimes \A_X^{\sharp}}.
$$
Equivalently,
$$
\Jet_{\Sigma_g}\bracket{\tilde{\E}}^\vee= \Jet_{\Sigma_g}\bracket{\A_{\Sigma_g}}^\vee\otimes \g_X[1]^\vee,
$$
where $\Jet_{\Sigma_g}\bracket{\A_{\Sigma_g}}^\vee$ is the complex of dual $D_{\Sigma_g}$-module of $\Jet_{\Sigma_g}\bracket{\A_{\Sigma_g}}$, with induced differential which we still denote by $d_{\Sigma_g}$. There is a natural identification between complexes of $D_{\Sigma_g}$-modules
$$
\Jet_{\Sigma_g}\bracket{\A_{\Sigma_g}}^\vee \iso D_{\Sigma_g}\otimes \wedge^* T_{\Sigma_g},
$$
where $T_{\Sigma_g}$ is the smooth tangent bundle, and the right hand side is the usual complex of Spencer's resolution. In particular, we have the quasi-isomorphism
\begin{align}\label{spencer resolution}
  (\Jet_{\Sigma_g}\bracket{\A_{\Sigma_g}}^\vee, d_{\Sigma_g})\simeq \cinfty(\Sigma_g)
\end{align}

$\Jet_{\Sigma_g}\bracket{\tilde{\E}}^\vee$ is a locally free $D_{\Sigma_g}$-module. We will let $\A_{\Sigma_g}^{top}$ denote the right $D_{\Sigma_g}$-module of
top differential forms on $\Sigma_g$. According to the definition of local functionals, $\Ol^+(\tilde{\E})$ is isomorphic to
the global sections of the following complex of sheaves on $\Sigma_g$:
\begin{equation}\label{eqn:deformation-obstruction-jet}
\begin{aligned}
\A_{\Sigma_g}^{top}\otimes_{D_{\Sigma_g}} \prod_{k\geq 1}\Sym^k_{D_{\Sigma_g}\otimes \A_X^{\sharp}}\bracket{\Jet_{\Sigma_g}\bracket{\tilde{\E}}^\vee}
\end{aligned}
\end{equation}
with the differential induced from $Q+\fbracket{I_{cl},-}$. All the sheaves here, including the sheaf of jets, are sheaves over
smooth functions on $\Sigma_g$. Thus these sheaves are all fine sheaves, a fact which implies that the cohomology we want to compute
is nothing but the hypercohomology of the complex (\ref{eqn:deformation-obstruction-jet}) with respect to the global section functor.

\begin{prop}
The cohomology of the deformation-obstruction complex of B-twisted $\sigma$-model is
$$
H^k(\widetilde{Ob})=\sum_{p+q=k+2}H^p_{dR}(\Sigma_g)\otimes H^q(X,\Omega_{cl}^1),
$$
where $\Omega_{cl}^1$ is the sheaf of closed holomorphic 1-forms on $X$. In particular, the obstruction space for the quantization at one-loop is given by
\begin{align*}
H^1(\widetilde{Ob})= \bracket{H^0_{dR}(\Sigma_g)\otimes H^3(X,\Omega_{cl}^1)} \oplus \bracket{H^1_{dR}(\Sigma_g)\otimes
H^2(X,\Omega_{cl}^1)}
\oplus \bracket{H^2_{dR}(\Sigma_g) \otimes H^1(X,\Omega_{cl}^1)}.
\end{align*}

\end{prop}
\begin{proof} We follow the strategy developed in \cite{Kevin-CS}. The Koszul resolution gives a resolution of
the $D_{\Sigma_g}$-module
$$
      \A_{\Sigma_g}(D_{\Sigma_g})[2]\to \A_{\Sigma_g}^{top},
$$
where $\A_{\Sigma_g}\bracket{D_{\Sigma_g}}[2]$ is the de Rham complex of $D_{\Sigma_g}$. Together with the quasi-isomorphism \eqref{spencer resolution} and the fact that the
$D_{\Sigma_g}$-module $\prod\limits_{k\geq
1}\Sym^k_{D_{\Sigma_g}\otimes \A_X^{\sharp}}\bracket{\Jet_{\Sigma_g}\bracket{\tilde{\E}}^\vee}$ is  flat,   we find
quasi-isomorphisms
\begin{align*}
  \widetilde{Ob}&\iso \A_{\Sigma_g}^{top}\otimes^{L}_{D_{\Sigma_g}} \bracket{\prod_{k\geq 1}\Sym^k_{D_{\Sigma_g}\otimes
\A_X^{\sharp}}\bracket{\Jet_{\Sigma_g}\bracket{\tilde{\E}}^\vee}}\\
  &\iso \A_{\Sigma_g}(D_{\Sigma_g})\otimes_{D_{\Sigma_g}}^L\bracket{ \prod_{k\geq 1}\Sym^k_{D_{\Sigma_g}\otimes
\A_X^{\sharp}}\bracket{\Jet_{\Sigma_g}\bracket{\tilde{\E}}^\vee}}[2]\\
&\iso \A_{\Sigma_g}\otimes_{\C}\bracket{\prod_{k\geq 1}
\Sym^k_{\A_{X}^{\sharp}}\bracket{\g_X[1]^\vee}}[2]=\A_{\Sigma_g}\otimes_{\C} C^*_{red}\bracket{\g_X}[2].
\end{align*}
The differential on the last complex is $d_{\Sigma_g}+l_1+\{I_{cl},-\}=d_{\Sigma_g}+d_{CE}$, where $d_{\Sigma_g}$ is the de Rham differential on
$\A_{\Sigma_g}$ and $d_{CE}$ is the \CE
differential on  the reduced \CE complex of $\mathfrak{g}_X$. Therefore
$$
   H^k\bracket{\widetilde{Ob}}=\sum_{p+q=k+2} H^p_{dR}\bracket{\Sigma_g}\otimes H^q\bracket{C^*_{red}(\g_X), d_{CE}}.
$$
Finally, from the following short exact sequence
$$
  0\to \A_X\to C^*\bracket{\g_X}\to C^*_{red}\bracket{\g_X}\to 0,
$$
we have the quasi-isomorphism of complexes of sheaves
$$
C^*_{red}\bracket{\g_X} \simeq \bbracket{\A_X\to C^*\bracket{\g_X}}\simeq \bbracket{\C\to \OO_X}\simeq \Omega^1_{X, cl},
$$
which implies
$$
   H^q\bracket{C^*_{red}(\g_X), d_{CE}}\cong H^q\bracket{X, \Omega_{X,cl}^1}.
$$
\end{proof}

\subsubsection{Computation of the obstruction}\label{subsection:genuine-quantization}
We now compute the obstruction to the quantization of B-twisted  $\sigma$-model.
In section \ref{subsection:naive-quantization}, we have seen that the $\mathbb{C}^\times$-invariance of $I_{cl}/\hbar$ and the RG flow
operator guarantees that the the naive quantization $\{I_{naive}[L]|L>0\}$ only contains the constant term and linear term in the power
expansion of $\hbar$. The naive quantization automatically satisfies the quantum master equation modulo $\hbar$ since $I_{cl}$
satisfies the classical master equation. Thus, we only need to take care of the one-loop anomaly. We have the following explicit
graphical expression  of one-loop anomaly for general perturbative QFT's:
 \begin{thm}\label{theorem:one-loop-anomaly} The  one-loop obstruction $O_1$ to quantizing a classical field theory with classical interaction $I_{cl}$ is
given graphically by
\begin{equation}\label{graph:one-loop anomaly}
O_1=\lim_{\epsilon\rightarrow0}\left(\figbox{0.18}{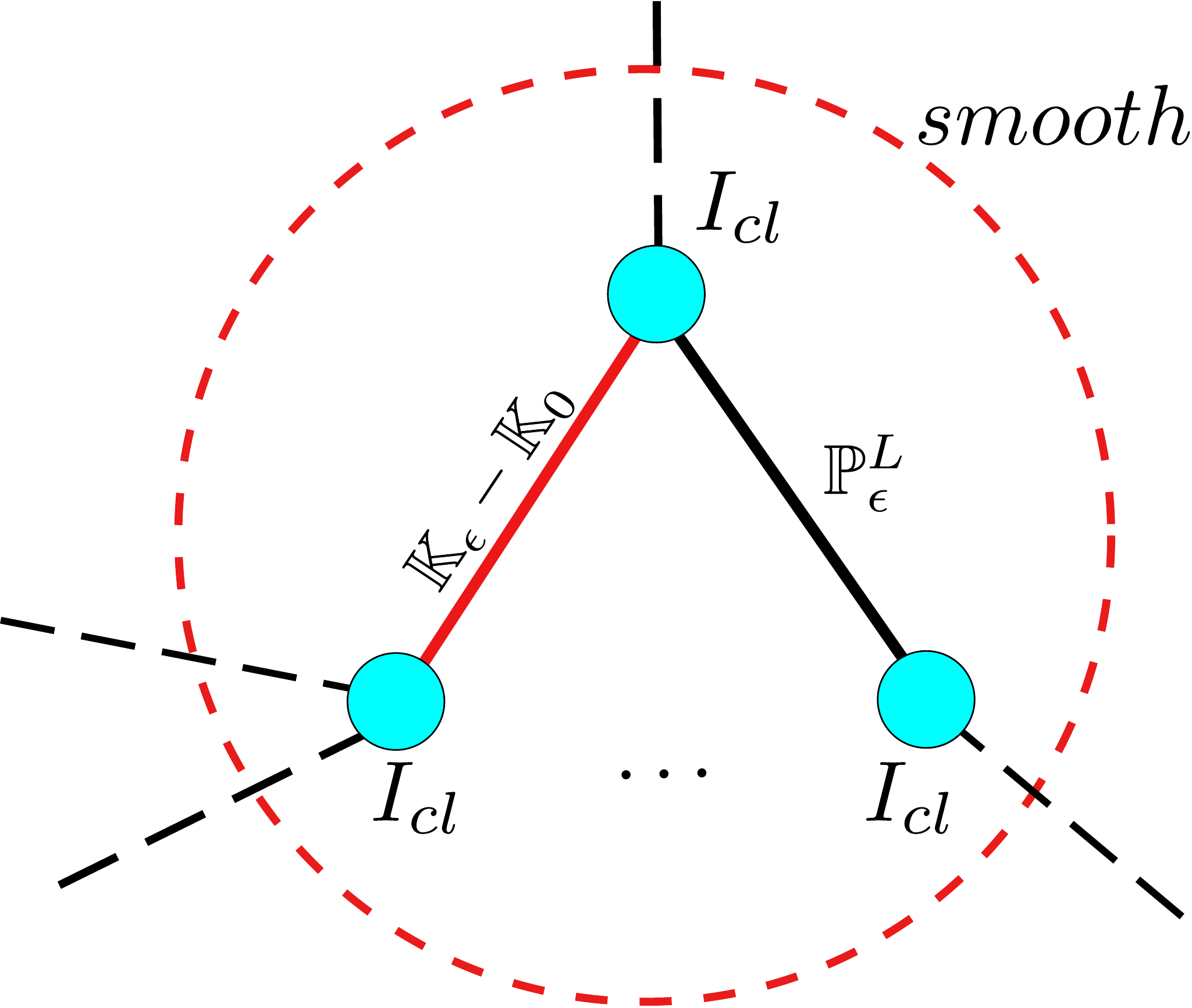}\right)+\lim_{\epsilon\rightarrow
0}\Big(\figbox{0.18}{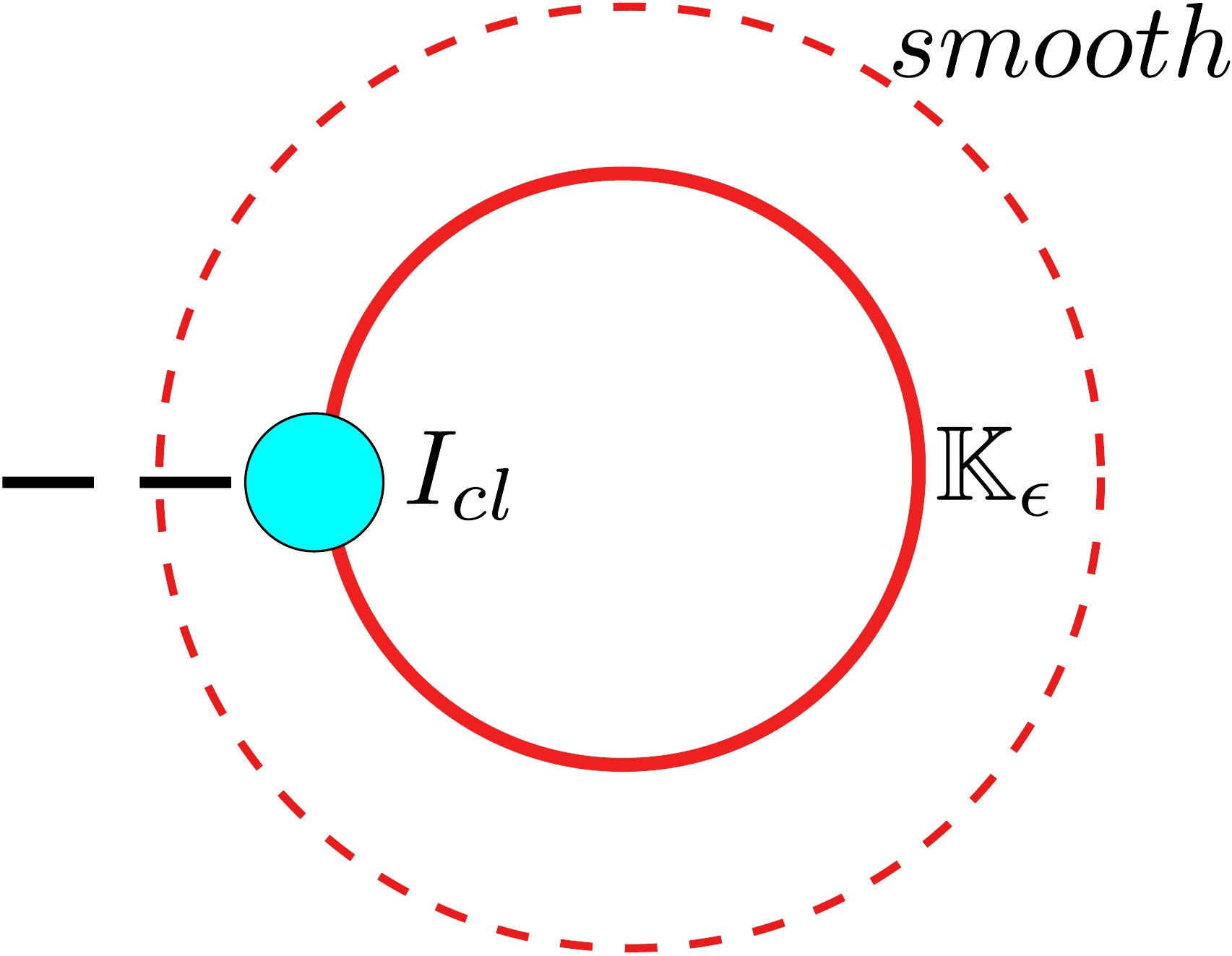}\Big)
\end{equation}
\end{thm}

\begin{rmk}
After fixing a renormalization scheme, we can define the smooth part of a Feynman weight $W_\gamma(P_\epsilon^L, I_{cl})$ for any graph
$\gamma$.  We take the smooth part of the term in the dashed red circle.
%whose tails can be either a tail of the whole graph which stretches out the circle, or attached to a tail of a tree
%which become solid lines labeled by $P_\epsilon^L$ outside the circle. These $P_\epsilon^L$ are not part of the function of
%$\epsilon$ of which we take smooth part. This is the reason we put such $P_\epsilon^L$ outside the dotted red circle in the above
%picture.
\end{rmk}
The proof of this theorem is given in Appendix \ref{appendix:one-loop-anomaly}. For B-twisted $\sigma$-model, the following two
lemmas imply that the first term in (\ref{graph:one-loop anomaly}) vanishes as $\epsilon\rightarrow 0$. We defer the proof of these two
lemmas to Appendix \ref{appendix:Feynman-graph-computation}.
\begin{lem}\label{lem:two-vertex-wheels}
Let $\gamma$ be a genus $1$ graph containing a  wheel with $2$ vertices. Then the following Feynman weight vanishes:
\begin{equation}\label{eqn:two-vertex-vanish-susy}
\figbox{0.18}{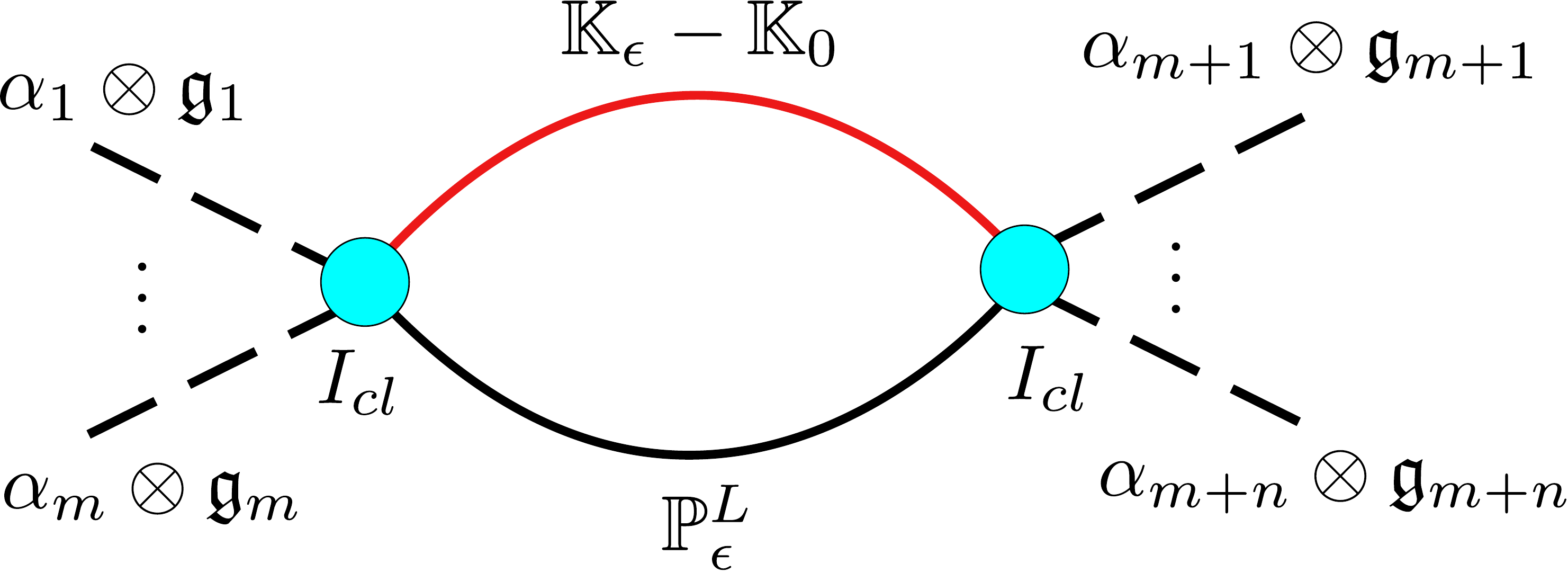}
\end{equation}
\end{lem}

\begin{lem}\label{lem:more-vertex-wheels}
Let $\gamma$ be a genus $1$ graph containing a wheel with $n$ vertices, and let $e$ be an edge of $\gamma$ which is
part of the wheel.
Assume
that $n\geq 3$, then we have $$\lim_{\epsilon\rightarrow 0}
W_{\gamma,e}(\mathbb{P}_{\epsilon}^L,\mathbb{K}_\epsilon-\mathbb{K}_0,I_{cl})=\lim_{\epsilon\rightarrow
0}\left(\figbox{0.18}{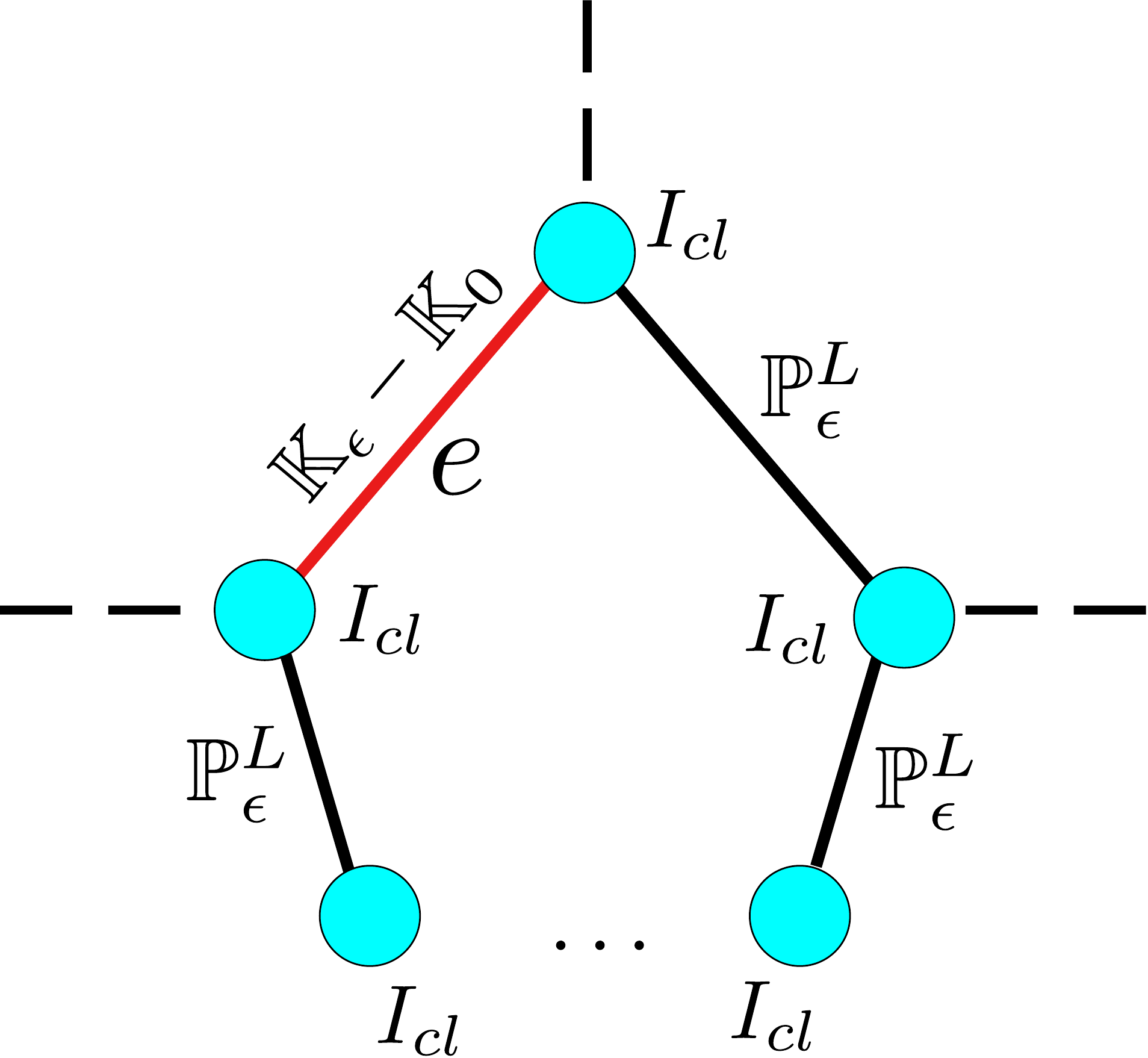}\right)=0.$$
\end{lem}
Hence the scale $L$ one-loop obstruction is given by
\begin{equation}\label{eqn:obstruction}
O_1[L]=\sum_{\gamma:tree}\lim_{\epsilon\rightarrow
0}\dfrac{1}{|\text{Aut}(\gamma)|}W_\gamma\Big(\mathbb{P}_\epsilon^L,\figbox{0.18}{one-loop-self}\Big).
\end{equation}
By the fact that $\lim_{L\rightarrow 0}(I_{naive}^{(0)}[L])=I_{cl}$, we have:
\begin{equation}\label{eqn:scale-0-obstruction}
O_1=\lim_{L\rightarrow 0}O_1[L]=\hspace{1mm}\lim_{\epsilon\rightarrow 0}\left(\figbox{0.18}{one-loop-self} \right).
\end{equation}

The obstruction $O_1$ contains an analytic part and a combinatorial part. It is clear that the analytic part is given by the limit of the super trace
of the heat kernel along the diagonal in $\Sigma_g\times \Sigma_g$:
$$
\lim_{\epsilon\rightarrow 0}Str(K_\epsilon(z,z))=(2-2g)\text{dvol}_{\Sigma_g},
$$ where
$\text{dvol}_{\Sigma_g}$ is the normalized volume form on $\Sigma_g$ with respect to the constant curvature metric, and the identity follows from local
index theorem. Similar to the holomorphic Chern-Simons theory \cite{Kevin-CS}, the combinatorial factor
of (\ref{eqn:scale-0-obstruction}) gives the first Chern class of the target manifold $X$. Thus, we can conclude this section by:
\begin{thm}\label{thm:obstruction-quantization}
The obstruction to quantizing B-twisted $\sigma$-model is given by $$[(2-2g)\text{dvol}_{\Sigma_g}]\otimes c_1(X)=c_1(\Sigma_g)\otimes
c_1(X)\in H^2_{dR}(\Sigma_g)\otimes
H^1(X,\Omega_{cl}^1)\subset H^1(\widetilde{Ob}),$$ and the topological B-twisted $\sigma$-model can be quantized (on any Riemann
surface $\Sigma_g$) if and only if the target $X$ is Calabi-Yau.
\end{thm}

\subsubsection{One-loop quantum correction}\label{subsection:one-loop-correction} Now let us assume that $X$ is a Calabi-Yau
manifold  with a holomorphic volume form $\Omega_X$. By Theorem \ref{thm:obstruction-quantization}, the quantization of our
topological B-twisted $\sigma$-model is unobstructed. This
means that there exists some quantum correction $I_{qc}[L]$ to the naive quantization $I_{naive}[L]$ such that $I_{naive}[L]+\hbar I_{qc}[L]$ solves the quantum
master equation. In this section we give an explicit description of the one-loop quantum correction which will be used in the next section to compute the
quantum correlation functions.

We first have the following lemma:
\begin{lem}\label{lem:obstruction-local-effective}
Let $I_{qc}\in\mathcal{O}_{loc}(\mathcal{E})$ be a local functional on $\mathcal{E}$ satisfying the equation
\begin{equation}\label{eqn:local-version-eqn-one-loop-correction}
QI_{qc}+\{I_{cl},I_{qc}\}=O_1,
\end{equation}
where $O_1$ is the one-loop anomaly described in section \ref{subsection:genuine-quantization}.
 Then the effective functionals $$I_{qc}[L]:=\lim_{\epsilon\rightarrow
0}\sum_{\gamma\in\text{trees}, v\in V(\gamma)}W_{\gamma,v}(P_\epsilon^L, I_{cl},I_{qc})$$ satisfy the equation
$$
QI_{qc}[L]+\{I^{(0)}_{naive}[L],I_{qc}[L]\}_L=O_1[L],
$$
where $W_{\gamma,v}(P_\epsilon^L, I_{cl},I_{qc})$ is the Feynman weight associated to the graph $\gamma$ with the vertex $v$
labeled by $I_{qc}$ and all other vertices labeled by $I_{cl}$. In particular, $I_{naive}[L]+\hbar I_{qc}[L]$ solves the quantum
master equation.
\end{lem}
\begin{proof}
The proof of the lemma is a simple Feynman graph calculation. See \cite{Kevin-book}.
\end{proof}

The objective is to find a local functional $I_{qc}$ satisfying equation (\ref{eqn:local-version-eqn-one-loop-correction}). Let
$\Delta$ be the operator on $\text{Sym}^*(\mathfrak{g}_X)\otimes\text{Sym}^*(\mathfrak{g}_X[1]^\vee)$ given by  contraction
with the identity in
$\text{End}_{\mathcal{A}_X}(\g_X\oplus\g_X^\vee)$, and let $L$ denote the functional on
$\mathfrak{g}_X[1]\oplus\mathfrak{g}_X^\vee$ given by $$L(\alpha+\beta):=\dfrac{1}{(n+1)!}\sum_{n\geqslant 0}\langle
l_n(\alpha^{\otimes n}),\beta\rangle,\hspace{3mm} \alpha\in\mathfrak{g}_X[1],\beta\in\mathfrak{g}_X^\vee.$$ From the graphical expression of $O_1$
in equation (\ref{eqn:scale-0-obstruction}), it is not difficult to see that $O_1$ is only a functional on
$C^\infty(\Sigma_g)\otimes\mathfrak{g}_X[1]$ of the following form: $$(O_1)_k((f_1\otimes
g_1)\otimes\cdots\otimes (f_k\otimes g_k))=(2-2g)(\Delta L)_k(g_1\otimes\cdots\otimes g_k) \int_{\Sigma_g}f_1\cdots f_k\ \text{dvol}_{\Sigma_g},
$$
where $(O_1)_k$ denotes the k-component of $O_1$ in $\OO^{(k)}(\E)$, and similarly for $(\Delta L)_k$.  We are looking for an $I_{qc}$ which is only a functional on $C^\infty(\Sigma_g)\otimes\g_X[1]$ of the form
\begin{align}\label{correction-form}(I_{qc})_k((f_1\otimes
g_1)\otimes\cdots\otimes (f_k\otimes g_k))=B_k(g_1\otimes\cdots\otimes g_k)\int_{\Sigma_g}f_1\cdots f_k\ \text{dvol}_{\Sigma_g},
\end{align}
where  $B_k\in \text{Sym}^k(\mathfrak{g}_X[1]^\vee)$. With this ansatz, we have $QI_{qc}=l_1 I_{qc}$ by type reason and
equation (\ref{eqn:local-version-eqn-one-loop-correction}) is reduced to
\begin{equation}\label{eqn:quantum-correction}
l_1 I_{qc}+\{I_{cl}, I_{qc}\}=O_1.
\end{equation}
Let $B=\sum_{k\geq 0} B_k$, it is clear that
$$
\left(l_1I_{qc}+\{I_{cl},I_{qc}\}\right)((f_1\otimes g_1)\otimes\cdots\otimes (f_k\otimes g_k))=(d_{CE}B)(g_1\otimes\cdots\otimes
g_k) \int_{\Sigma_g}f_1\cdots f_k\ \text{dvol}_{\Sigma_g}.
$$
Equation (\ref{eqn:quantum-correction}) is then reduced to
\begin{equation*}
d_{CE}B=(2-2g)\Delta L
\end{equation*}
which, since the Chevalley-Eilenberg differential $d_{CE}$ is the same as the bracket
$\{L,-\}$, can be further reduced to
\begin{equation}\label{eqn:one-loop-correction-equation}
(2-2g)\Delta L-\{L, B\}=0.
\end{equation}
 Since we only need to solve the equation modulo constant functionals, equation (\ref{eqn:one-loop-correction-equation}) is
equivalent to the vanishing of the operator $\{(2-2g)\Delta L-\{L,B\},-\}$.

\begin{lem}\label{lem:one-loop-correction-brackets}
We have the following two identities for any $B$:
\begin{equation*}
 \begin{aligned}
 \{\{L,B\},-\}&=[\{L,-\},\{B,-\}],\\
\{\Delta L,-\}&=[\Delta,\{L,-\}].
 \end{aligned}
\end{equation*}

\end{lem}
\begin{proof}
 The first identity follows directly from the Jacobi identity. The second identity follows from the identity
 $$
   \bbracket{\Delta, \bbracket{\Delta, L}}=0.
 $$
\end{proof}
By Lemma \ref{lem:one-loop-correction-brackets}, to solve equation (\ref{eqn:one-loop-correction-equation}), we only need to find
$B\in C^*(\mathfrak{g}_X)$ such that the operator $\Delta+\{B,-\}$ commutes with the Chevalley-Eilenberg differential
$d_{CE}=\fbracket{L,-}$. The following technical
proposition transfers the problem to a  geometric context:
\begin{prop}\label{prop:iso-CE-jet}\cite{Kevin-CS}
There is a natural isomorphism of cochain complexes of $\mathcal{A}_X$-modules
\begin{equation}\label{eqn:iso-CE-Jet}
\tilde{K}:\left(C^*(\mathfrak{g}_X,\text{Sym}^*\mathfrak{g}_X),d_{CE}\right)\overset{\sim}{\rightarrow}\left(\mathcal{A}_X\otimes_{
\mathcal{O}_X}\text{Jet}^{hol}_X(\wedge^*T_X),d_{D_X}\right),
\end{equation}
where $d_{CE}$ on the left hand side is the Chevalley-Eilenberg differential of the $\mathfrak{g}_X$-module
$\text{Sym}^*\mathfrak{g}_X$, and $d_{D_X}$ is the differential of the de Rham complex of the holomorphic jet bundle.
\end{prop}

The explicit formula of the above isomorphism is given in Appendix \ref{appendix:L_infty}.

There is a natural second order differential operator on the right hand side of equation (\ref{eqn:iso-CE-Jet}) which commutes
with the differential $d_{D_X}$: let $\Omega_X$ be a holomorphic volume form on $X$ which induces an isomorphism between
holomorphic polyvector fields and holomorphic differential forms via the contraction map:
\begin{equation*}
 \begin{aligned}
 \wedge^* T_X& \overset{\sim}{\rightarrow} \Omega_X^*\\
  \alpha&\mapsto\alpha\lrcorner\ \Omega_X.
 \end{aligned}
\end{equation*}

This isomorphism transfers the holomorphic de Rham differential $\pa$ on $\Omega_X^*$ to an operator on polyvector fields:
 $$\pa_{\Omega_X}:\Gamma(\wedge^*T_X)\rightarrow
\Gamma(\wedge^{*-1}T_X),$$
which naturally induces a second order operator (denoted by the same symbol)
$$
\partial_{\Omega_X}: \mathcal{A}_X\otimes_{\mathcal{O}_X}\text{Jet}_X^{hol}(\wedge^{*}T_X)\to \mathcal{A}_X\otimes_{\mathcal{O}_X}\text{Jet}_X^{hol}(\wedge^{*-1}T_X),
$$
that  commutes with
$d_{D_X}$.

To solve equation (\ref{eqn:one-loop-correction-equation}), we need to transfer the operator $\Delta$ to the de Rham complex of
jet bundle in (\ref{eqn:iso-CE-Jet}). For simplicity, we still denote this operator by $\Delta$.
\begin{claim}
The two second order differential operators $\Delta$ and $\partial_{\Omega_X}$ on
$\mathcal{A}_X\otimes_{\mathcal{O}_X}\text{Jet}_X^{hol}(\wedge^*T_X)$ have the same symbol.
\end{claim}
\begin{proof}
We prove the claim by some local calculation from which we can also find an explicit expression of the functional $B\in
C^*(\mathfrak{g}_X)$.

Let $\{z^1,\cdots, z^n\}$ be local holomorphic coordinates on $U\subset X$ where $n=\dim_{\mathbb{C}}X$, such that the holomorphic volume can be expressed as
$\Omega_X|_U=dz^1\wedge\cdots\wedge dz^n$, and let $\delta z^1,\cdots,\delta z^n$ be the corresponding jet coordinates.
% It is clear that $\{dz^1,\cdots,dz^n,\partial_{z^1},\cdots,\partial_{z^n}\}$ generates the
%algebra $C^*(\mathfrak{g}_X,\text{Sym}^*(\mathfrak{g}_X))$ over $\mathcal{A}_X^{*,*}$.
The isomorphism $\tilde{K}$ in equation (\ref{eqn:iso-CE-Jet}) gives rise to (recall Notation \ref{notation-basis}):
\begin{align}\label{local-identification}
 \A_X(U)[[\delta z^i, \pi_2^*(\pa_{z^i})]] =\mathcal{A}_X\otimes_{\mathcal{O}_X}\text{Jet}^{hol}_X(\wedge^*T_X)(U)\cong
\mathcal{A}_X(U)[[\tilde{K}(\widetilde{dz^i}),\tilde{K}(\widetilde{\partial_{z^j}})]].
\end{align}
 Let $T$ denote the restriction of  ${\rho^*}^{-1}$ in (\ref{eqn:identification-CE-jet}) to
$\Omega_X^1$:
$$
   T: \Omega_X^1 \to \cinfty(X)\otimes_{\OO_X}\Jet^{hol}_X(\OO_X).
$$
Let $\pa_{dR}$ be the internal de Rham differential
$$
  \pa_{dR}: \Jet^{hol}_X(\Omega^*_X)\to \Jet^{hol}_X(\Omega^{*+1}_X),
$$
and let $\pa_{dR}\circ T$ be the composition
\begin{equation}\label{eqn:splitting-T}
  \pa_{dR}\circ T: \Omega_X^1\to  \cinfty(X)\otimes_{\OO_X}\Jet^{hol}_X(\Omega_X^1).
\end{equation}
Let
$$
 \abracket{-,-}:\bracket{\cinfty(X)\otimes_{\OO_X}\Jet^{hol}_X(T_X)}\otimes_{\cinfty(X)} \bracket{\cinfty(X)\otimes_{\OO_X}\Jet^{hol}_X(\Omega_X^1)}\to
\cinfty(X)\otimes_{\OO_X}\Jet^{hol}_X(\OO_X)
$$
be the natural pairing induced from that between $T_X$ and $\Omega_X^1$. By our convention, $T(dz^i)=\tilde T(\widetilde{dz^i})$, and
\begin{align}\label{compare-coordinate}
   \abracket{\pa_{dR}\circ T(dz^i), \tilde K(\widetilde{\pa_{z^j}})}=\delta_{j}^i, \quad \abracket{\pa_{dR}(\delta z^i), \pi_2^*(\pa_{z^j})}=\delta^i_j.
\end{align}

By construction, there exists an invertible $P\in \cinfty(X)\otimes_{\OO_X} \text{Jet}_X^{hol}(\mathcal{O}_X)(U)$ such that
\begin{equation}\label{eqn:quantum-correction-exp}
\pi_2^*(dz^1\wedge\cdots \wedge dz^n)=P\cdot\left((\partial_{dR}\circ T)(dz^1)\wedge\cdots\wedge (\partial_{dR}\circ T)(dz^n)\right)\in\text{Jet}_X^{hol}(\Omega_X^*)(U).
\end{equation}

Under the identification \eqref{local-identification},
$$
  \Delta=\sum_i\dfrac{\partial}{\partial(\tilde{T}(\widetilde{dz^i})}\dfrac{\partial}{\partial(\tilde{K}(\widetilde{\partial_{z^i}}))}, \quad
\pa_{\Omega_X}=\sum_i\dfrac{\partial}{\partial(\delta z^i)}\dfrac{\partial}{\partial(\pi_2^*(\partial_{z^i}))}.
$$
By \eqref{compare-coordinate}, \eqref{eqn:quantum-correction-exp},
it is not difficult to see that
\begin{equation}\label{eqn:comparison-operator}
\begin{aligned}
\pa=\Delta+\sum_i \abracket{\pa_{dR}\circ T(dz^i), \log P}\dfrac{\partial}{\partial(\tilde{K}(\widetilde{\partial_{z^i}}))}=\partial+\{\log P,-\}.
\end{aligned}
\end{equation}
This proves the claim.
\end{proof}

We conclude this section with the following theorems:
\begin{thm}\label{thm:one-loop-correction} Any pair $(X, \Omega_X)$ leads to a canonical quantization of topological B-twisted $\sigma$-model, whose one-loop quantum correction,
which will be denoted by $I_{qc}$, is of the form (\ref{correction-form}).
\end{thm}

The theorem follows from the following explicit description of $B$ in \eqref{correction-form}.  By taking the top wedge product of $\pa_{dR}\circ T$, we define
$$
   \wedge^n\bracket{\pa_{dR}\circ T}: \Omega_X^n\to \cinfty(X)\otimes_{\OO_X}\Jet^{hol}_X(\Omega_X^n).
$$

\begin{prop}\label{correction-formula}  The quantum correction associated to the canonical quantization of the pair $(X, \Omega_X)$ has the combinatorial part
$$
   B=(2-2g)\log \bracket{{\pi_2^*(\Omega_X) \over \wedge^n\bracket{\pa_{dR}\circ T}(\Omega_X)}}\in
\cinfty(X)\otimes_{\OO_X}\Jet_X^{hol}(\OO_X)\subset C^*(\g_X),
$$
where $\pi_2$ is the same as in Definition \ref{def:jet-bundle}.
\end{prop}
\begin{proof}
This follows from \eqref{eqn:one-loop-correction-equation} and the local calculation \eqref{eqn:comparison-operator}.
\end{proof}

\begin{rmk}The existence of the quantum correction is due to the fact that the curved $L_\infty$ structure $\g_X$ requires the choice of a splitting \cite{Kevin-CS}, although
different choices lead to homotopic equivalent theory. The quantum correction $I_{qc}$ precisely compensates such choice and link the effective \BV geometry to the canonical \BV
structure of polyvector fields associated to the Calabi-Yau structure.

\end{rmk}

With the one-loop quantum correction term $I_{qc}$, we can give an explicit formula of the constant term $R\in\A_X$ in quantum
master equation (\ref{eqn:qunatum-master-equation}), which will be used later in observable theory:
\begin{lem}\label{lem:S_1-contant term}
Let $(I_{qc})_1$ denote the linear term in the one-loop correction $I_{qc}$, and let $\tilde{l}_0$ denote the functional on $\E$
given by
$$
\tilde{l}_0(\alpha+\beta)=\langle l_0,\beta\rangle.
$$
Then the constant term $R$ is given by:
$$
R=\{(I_{qc})_1,\tilde{l}_0\}.
$$
\end{lem}
\begin{proof}
Let $I[L]=I^{(0)}[L]+\hbar I^{(1)}[L]$ be the scale
$L$ effective interaction. Then the quantum master equation (\ref{eqn:qunatum-master-equation}) can be expanded as
\begin{equation}\label{eqn:QME-constant-term}
Q_{L}I[L]+\frac{1}{2}\{I^{(0)}[L]+\hbar I^{(1)}[L],I^{(0)}[L]+\hbar I^{(1)}[L]\}_L+\hbar\Delta_L
I[L]+\hbar R+F_{l_1}=0.
\end{equation}
It is clear by type reason that the  constant term in (\ref{eqn:QME-constant-term}) other than $\hbar R$ can only live in the
bracket $\{I^{(0)}[L],\hbar I^{(1)}[L]\}_L$.
Thus we only need to find the linear terms in both $I^{(0)}[L]$ and $I^{(1)}[L]$. On one hand, it is obvious that the
only linear term in $I^{(0)}[L]$ is $\tilde{l}_0$ since $\tilde{l}_0$ does not propagate by the type reason. Therefore the only linear term in
$I^{(1)}[L]$ that contributes $\{I^{(1)}[L], \tilde{l}_0\}_L$ is $(I_{qc})_1$.  It follows that
$$
 R=\{(I_{qc})_1, \tilde{l}_0\}_L=\{(I_{qc})_1, \tilde{l}_0\}
$$
since $R$ does not depend on $L$.
\end{proof}

\section{Observable theory}
The objective of this section is to study the quantum observables of B-twisted topological $\sigma$-model following the general
theory developed by Costello and Gwilliam \cite{Kevin-Owen}. In section
\ref{section:local-observable}, we show that classical and quantum local observables are  given by the cohomology of polyvector
fields. In section \ref{section:global-observable}, we study global topological quantum observables on Riemann
surfaces of any genus $g$. Using the local to global factorization map, we define the topological correlation functions of quantum observables. In section \ref{section:correlation-function}, we show that the correlation
functions on $\mathbb{P}^1$ are given by the trace map on Calabi-Yau manifold, and the partition function on the elliptic curve reproduces the Euler
characteristic of the target manifold. This is in complete agreement with the physics prediction.

\subsection{Classical observables}\label{section:local-observable}

We first recall that classical observables are given by the derived critical locus of the classical action functional \cite{Kevin-Owen}.

\begin{defn} The classical observables of the B-twisted $\sigma$-model is the graded commutative factorization algebra on
$\Sigma_g$ whose value on an open subset $U\subset
\Sigma_g$ is the cochain complex
\begin{equation}\label{eqn:classical-observable}
 \text{Obs}^{cl}(U):=\left(\mathcal{O}(\mathcal{E}_U), Q+\{I_{cl},-\}\right).
\end{equation}
Here $I_{cl}$ is the classical interaction functional and
$\mathcal{E}_U=\mathcal{A}_{\Sigma_g}(U)\otimes(\mathfrak{g}_X[1]\oplus\mathfrak{g}_X^\vee)$.
\end{defn}

By definition,
$$
\mathcal{O}(\mathcal{E}_U)=\widehat{\Sym}\bracket{\mathcal{E}_U^\vee}=\prod_{k\geq 0}\Sym^k \bracket{\mathcal{E}_U^\vee}.
$$
With the help of the symplectic pairing, we have the following identification:
$$
\quad \mathcal{E}_U^\vee\iso \overline{\mathcal{A}}_c(U)[2]\otimes(\mathfrak{g}_X^\vee[-1]\oplus\mathfrak{g}_X),
$$
where $\overline{\A}_c(U)$ is the space of compactly supported distribution-valued differential forms on $U$. Thus we have
\begin{equation*}
\begin{aligned}
\text{Sym}^n(\mathcal{E}_U^\vee)&=\text{Sym}^n\left(\bracket{\mathcal{A}(U)\otimes\bracket{\g_X[1]\oplus \g_X^\vee}}^\vee\right)\\
&\cong\text{Sym}^n\bracket{\overline{\mathcal{A}}_c(U)[2]\otimes \bracket{\g_X^\vee[-1]\oplus \g_X}}.\\
\end{aligned}
\end{equation*}

We would like to consider local observables in a small disk on $\Sigma_g$ and define their correlation functions. This can be viewed as the mirror
consideration of observables associated to marked points in Gromov-Witten theory. At the classical level, we have

\begin{prop}\label{prop-classical-ob}
Let $U\subset \Sigma_g$ be a  disk. The cohomology of classical local observables of B-twisted topological $\sigma$-model on
$U$ is given by the cohomology
of polyvector fields:
$$
H^k(\text{Obs}^{cl}(U))\iso \bigoplus_{p+q=k}H^p(X,\wedge^qT_X).
$$
\end{prop}
\begin{proof} Recall that $\text{Obs}^{cl}(U)$ is a dg-algebra over $\A_X$. Let $\mathcal{A}_X^k$ denote the smooth $k$-forms on $X$. We filter $\text{Obs}^{cl}(U)$  by defining
$$
F^k\text{Obs}^{cl}(U):= \mathcal{A}^k_X\text{Obs}^{cl}(U).
$$
Since the operator $l_1+\{I_{cl},-\}$ increases the
degree of differential forms on $X$ by one while $d_{\Sigma_g}$ preserves it, it is clear that the $E_1$-page of the spectral sequence is obtained by taking the cohomology
with respect to $d_{\Sigma_g}$. By Atiyah-Bott's lemma, the chain complex of currents on $U$ is quasi-isomorphic to the chain complex of compactly supported differential forms.
Thus we have:
$$
E_1=\bracket{\widehat{\Sym}\bracket{H_c^2(U)\otimes\bracket{\g_X[1]^\vee\oplus \g_X}}, l_1+\fbracket{I_{cl},-}}.
$$
The next lemma identifies the $E_1$-page of the spectral sequence with the de Rham complex of certain jet bundle on $X$. It is clear that the spectral sequence
degenerates at the $E_2$-page. Thus we have the quasi-isomorphism
\begin{equation*}
\begin{aligned}
\text{Obs}^{cl}(U)&\cong\left(\mathcal{A}_X\otimes_{\mathcal{O}_X}\text{Jet}^{hol}_X(\wedge^*T_X),d_{D_X}\right)\cong(\mathcal{A}_X^{0,*}\otimes_{
\mathcal{O}
_X}\wedge^*T_X,\bar{\partial}).
\end{aligned}
\end{equation*}
The proposition follows by taking the cohomology of the rightmost cochain complex.
\end{proof}

\begin{lem}\label{lem:local observable-iso-jet}
We have the following isomorphism of cochain complexes over the dga $\A_X$:
\begin{equation*}
\begin{aligned}
&\bracket{\widehat{\Sym}\bracket{H_c^2(U)\otimes\bracket{\g_X[1]^\vee\oplus \g_X}}, l_1+\fbracket{I_{cl},-}}
\cong&\left(\mathcal{A}_{X}\otimes_{\mathcal{O}_X}\text{Jet}^{hol}_X(\wedge^*T_X),d_{D_X}\right),
\end{aligned}
\end{equation*} where $d_{D_X}$ denotes the differential of the de Rham complex of
$\text{Jet}^{hol}_X(\wedge^*T_X)$.
\end{lem}
\begin{proof}
Since $U$ is a disk in $\Sigma_g$, we have the canonical isomorphism $H_c^2(U)\cong\C$ induced by the integration of $2$-forms. And the following isomorphism is clear:
$$
\bracket{\widehat{\Sym}\bracket{H_c^2(U)\otimes\bracket{\g_X[1]^\vee\oplus \g_X}}, l_1+\fbracket{I_{cl},-}}\cong
(C^*(\g_X,\text{Sym}^*\g_X),d_{CE}),
$$
thus the Lemma follows from Proposition \ref{prop:iso-CE-jet}.
\end{proof}

\subsection{Quantum observables} \label{section:global-observable}
Quantum observables are the quantization of  classical observables. Let $I[L]$ be a quantization of the classical interaction
$I_{cl}$. The operator $Q_L+\{I[L],-\}_L+\hbar
\Delta_L$ sqaures zero (Lemma \ref{quantum-BRST}) and defines a quantization of the classical  operator $Q+\{I_{cl},-\}$.

\begin{defn} The quantum observables on $\Sigma_g$ at scale $L$ is defined as the cochain complex
$$\text{Obs}^q(\Sigma_g)[L]:=\bracket{\mathcal{O}(\mathcal{E})[[\hbar]],Q_L+\{I[L],-\}_L+\hbar\Delta_L}.$$
\end{defn}
The definition is independent of the scale $L$ since quantum observables at different scales are homotopic equivalent via renormalization group flow
(See \cite[Chatper 5, Section 9]{Kevin-book}). Therefore
we will also use $\text{Obs}^q(\Sigma_g)$ to denote quantum observables when the scale is not specified.

The quantum observables form a factorization algebra on $\Sigma_g$ \cite{Kevin-Owen}. To define the quantum observables on an arbitrary open subset $U\subset \Sigma_g$, we need
the concept of parametrices.

\begin{defn} A parametrix $\Phi$ is a distributional section
$$
   \Phi \in \Sym^2\bracket{\overline{\E}}
$$
with the following properties:
\begin{enumerate}
\item $\Phi$ is of cohomological degree $1$ and $(Q\otimes 1+1\otimes Q)\Phi=0$,
\item ${1\over 2}\bracket{H\otimes 1+1\otimes H}\Phi- K_{0} \in \Sym^2\bracket{\E}$ is smooth, where $H=[Q,Q^{GF}]$ is the Laplacian and $K_{0}=\lim\limits_{L\to 0}K_L$ is the kernel of the
identity operator.
\end{enumerate}

\end{defn}

\begin{rmk} We have dropped the "proper" condition as in \cite{Kevin-Owen}. This is automatic here since we are working with compact Riemann surface
$\Sigma_g$. We have also symmetrized $(H\otimes 1)\Phi$ used in \cite{Kevin-Owen}.
\end{rmk}

\begin{defn} We define the propagator $P(\Phi)$ and BV kernel $K_{\Phi}$ associated to a parametrix $\Phi$ by
$$
  P(\Phi):={1\over 2}\bracket{Q^{GF}\otimes 1+1\otimes Q^{GF}}\Phi \in \Sym^2\bracket{\overline{\E}}, \quad K_{\Phi}:=K_0-{1\over 2}\bracket{H\otimes 1+1\otimes H}\Phi.
$$
The effective BV operator $\Delta_\Phi:={\pa \over \pa K_\Phi}$ induces a BV bracket $\{-,-\}_\Phi$ on $\OO\bracket{\E}$ in a similar way as the scale $L$ BV bracket $\{-,-\}_L$.
\end{defn}

The following identity describes the relation between the propagator $P(\Phi)$ and BV kernel $K_{\Phi}$:
$$
(Q\otimes 1+1\otimes Q)P(\Phi)=K_0-K_{\Phi},
$$
i.e., $P(\Phi)$ gives a homotopy between the singular kernel $K_0$ and the regularized kernel $K_{\Phi}$.

\begin{eg} $\Phi=\int_0^L \mathbb K_t dt$ is the parametrix we have used to define quantization. There
$$
   P(\Phi)={1\over 2}\int_0^L \bracket{Q^{GF}\otimes 1+1\otimes Q^{GF}}\mathbb K_t dt =\int_0^L \bracket{Q^{GF}\otimes 1}\mathbb K_t dt = \mathbb P_0^L, \quad K_\Phi=\mathbb K_L, \quad \Delta_\Phi=\Delta_L.
$$
\end{eg}

The basic reason we use arbitrary parametrix here is that the usual renormalization group flow $W\bracket{\mathbb P_\epsilon^L,-}$ of observables using  length scales does not
preserve the property of being supported in an open subset $U$. Instead, there exist parametrices whose supports are arbitrarily close to the diagonal
$\Delta\subset\Sigma_g\times\Sigma_g$ that we can use to achieve this.

\begin{defn} Let $I[L]$ be a given quantization of $I_{cl}$, and let $\Phi$ be a parametrix. We define the effective quantization $I[\Phi]$ at the parametrix $\Phi$ by
$$
  I[\Phi]:=W\bracket{P(\Phi)-\mathbb P_0^L, I[L]}.
$$
\end{defn}
Note that $P(\Phi)-\mathbb P_0^L\in \Sym^2(\E)$ is a smooth kernel since
$$
  (H\otimes 1+1\otimes H)(P(\Phi)-\mathbb P_0^L)=(Q^{GF}\otimes 1+1\otimes Q^{GF})({1\over 2}(H\otimes 1+1\otimes
H)\Phi-\mathbb{K}_0+\mathbb{K}_L)
$$
is smooth and $H$ is an elliptic operator.

$I[\Phi]$ satisfies a version of quantum master equation described by the parametrix $\Phi$ as in \cite{Kevin-Owen} (with a slight modification to include $F_{l_1}$),  and defines
the corresponding cochain complex of quantum observables. We leave the details to the readers since we will not use its form for later discussions. Furthermore, different
parametrices $\Phi, \Phi^\prime$ lead to homotopic equivalent cochain complexes which are linked by the renormalization group flow $W(P(\Phi)-P(\Phi^\prime),-)$.

\begin{defn}[\cite{Kevin-Owen}] Given a quantum observable $O[L]$ at scale $L$, we define its value $O[\Phi]$ at the parametrix $\Phi$ by requiring that
$$
  I[\Phi]+\delta O[\Phi]:=W\bracket{P(\Phi)-\mathbb P_0^L, I[L]+\delta O[L]},
$$
where $\delta$ is a square-zero parameter. The map $O[L]\mapsto O[\Phi]$ defines a homotopy between the corresponding cochain complexes of observables.
\end{defn}

\iffalse
It is straightforward to prove the following lemma:
\begin{lem}
$I[\Phi]$ satisfies the following $\Phi$-quantum master equation
$$
  \bracket{Q_{\Phi}+\hbar \Delta_{\Phi}+{F_{l_1}\over \hbar}}e^{I[\Phi]/\hbar}=R e^{I[\Phi]/\hbar}.
$$
Here $Q_{\Phi}:=Q+l_1^2 \widehat{P(\Phi)}$, where $\widehat{P(\Phi)}$ denotes the linear operator on $\E$ whose kernel is $P(\Phi)$.  Moreover, the evaluation map
$$
   O[L] \to O[\Phi]
$$
gives an isomorphism of complexes
$$
      \bracket{\mathcal{O}(\mathcal{E})[[\hbar]],Q_L+\{I[L],-\}_L+\hbar\Delta_L}\to
\bracket{\mathcal{O}(\mathcal{E})[[\hbar]],Q_\Phi+\{I[\Phi],-\}_\Phi+\hbar\Delta_\Phi}.
$$
\end{lem}
\fi

\begin{defn} Given $O\in \OO(\E)=\prod\limits_{k,i\geq 0}\Sym^i\bracket{\E^\vee} \hbar^k$, we will let $O_{i}^{(k)}$ denote the
corresponding component, i.e.
$$
  O=\sum_{k,i\geq 0}O_{i}^{(k)}\hbar^k.
$$
\end{defn}

\begin{defn}[\cite{Kevin-Owen}] We say that a quantum observable $O[L]$ has support in $U$, if for any $k,i\geq 0$, there exists a parametrix $\Phi$ such that
$$
   \text{Supp}\bracket{O[\Phi]_{i}^{(k)}}\subset U.
$$
\end{defn}

As shown in \cite{Kevin-Owen}, the subspace of quantum observables supported in $U$ forms a sub-cochain complex of $\text{Obs}^q(\Sigma_g)$, which will be denoted by $\text{Obs}^q(U)$.

\subsubsection{Local quantum observable} Let $U$ be a disk on $\Sigma_g$. As shown in \cite{Kevin-Owen} with great generality, the cohomology of the local quantum observables
$$
  H^*\bracket{\text{Obs}^q(U)}
$$
defines a deformation of $H^*\bracket{\text{Obs}^{cl}(U)}$:
\begin{align}\label{local-quantum-classical}
H^*\bracket{\text{Obs}^q(U)}\otimes_{\C[[\hbar]]}\C \iso H^*\bracket{\text{Obs}^{cl}(U)}.
\end{align}
We will construct a splitting map in this subsection, reflecting the vanishing of quantum corrections for observables in our B-model.

 Let $\eta\in H^2_c(U)$ be a fixed generator with $\int_U\eta=1$. By the proof of Propostion \ref{prop-classical-ob}, it induces a quasi-isomorphic embedding
 $$
 \left(\mathcal{A}_X\otimes_{\mathcal{O}_X}\text{Jet}^{hol}_X(\wedge^*T_X),d_{D_X}\right)\into \text{Obs}^{cl}(U),
 $$
and different choices of $\eta$ are homotopic equivalent. Let $\mu\in \mathcal{A}_X\otimes_{\mathcal{O}_X}\text{Jet}^{hol}_X(\wedge^*T_X)$, and we will denote by $O_\mu$
the corresponding local classical observable. Let $\{O_\mu[L]|L>0\}$ denote the RG flow of the classical observable $O_\mu$.  More
explicitly, we define $O_\mu[L]$ by requiring that
$$
I[L]+\delta O_\mu[L]=\lim_{\epsilon\to 0}W(\mathbb{P}_\epsilon^L,I_{cl}+\hbar I_{qc}+\delta O_\mu),
$$
where $\delta^2=0$,  and $I_{qc}$ denotes the one-loop quantum correction in
equation (\ref{eqn:local-version-eqn-one-loop-correction}). The existence of the limit follows from Lemma/Definition
\ref{lem:naive quantization} and the observation that the distribution $O_\mu$ is in fact smooth (tensor products of $\eta$'s). By
construction, $O_\mu[L]$ is a local quantum observable supported in $U$. We denote the above map by
$$
     \Psi: \mathcal{A}_X\otimes_{\mathcal{O}_X}\text{Jet}^{hol}_X(\wedge^*T_X)\to \text{Obs}^{q}(U), \quad \mu\mapsto O_\mu[L].
$$

\begin{prop}\label{prop:local-observable} $\Psi$ is a cochain map.
\end{prop}
\begin{proof} Let $U_L=Q_L+\hbar\Delta_L+\{I[L],-\}_L$ be the differential on quantum observables. By construction,
$$
   O_\mu[L]e^{I[L]/\hbar}=\lim_{\epsilon\to 0}e^{\hbar {\pa \over \pa \mathbb{P}_\epsilon^L}} \left(O_\mu e^{I_{cl}/\hbar+I_{qc}}\right).
$$
By Lemma \ref{quantum-BRST},
\begin{align*}
   &(U_L(O_\mu[L])+O_\mu[L]R)e^{I[L]/\hbar}
   =(Q_L+\hbar\Delta_L+F_{l_1}/\hbar)\bracket{O_\mu[L]e^{I[L]/\hbar}}\\
   =&\lim_{\epsilon\to 0}e^{\hbar {\pa \over \pa \mathbb{P}_\epsilon^L}} (Q_\epsilon+\hbar\Delta_\epsilon+F_{l_1}/\hbar)\bracket{O_\mu e^{I_{cl}/\hbar+I_{qc}}}\\
   =&\lim_{\epsilon\to 0}e^{\hbar {\pa \over \pa \mathbb{P}_\epsilon^L}}\bracket{\bracket{Q_\epsilon O_\mu+ \{I_{cl},O_\mu\}_\epsilon}
e^{I_{cl}/\hbar+I_{qc}}+O_\mu (Q_\epsilon+\hbar\Delta_\epsilon+F_{l_1}/\hbar)e^{I_{cl}/\hbar+I_{qc}}}
%   =&\lim_{\epsilon\to 0}e^{\hbar {\pa \over \pa \mathbb{P}_\epsilon^L}}\bracket{\bracket{Q_\epsilon O_\mu+
%\{I_{cl},O_\mu\}_\epsilon} e^{I_{cl}/\hbar+I_{qc}}+O_\mu
%(\{I_{cl},I_{cl}\}_\epsilon-\{I_{cl},I_{cl}\}_0+R)e^{I_{cl}/\hbar+I_{qc}}},
\end{align*}
where we have used the fact that both $O_\mu$ and $I_{qc}$ can only have non-trivial inputs for $0$-forms on $\Sigma_g$, hence
$$
\hbar\Delta_\epsilon O_\mu=\{O_\mu, I_{qc}\}_\epsilon=0
$$
by the type reason. Since the distribution $O_\mu$ is in fact smooth,  we are safe to take $\epsilon\to 0$ by a similar argument as Lemma \ref{lem:naive quantization}, Lemma \ref{lem:two-vertex-wheels} and Lemma \ref{lem:more-vertex-wheels}. The first term above gives
$$
\lim_{\epsilon\to 0}e^{\hbar {\pa \over \pa \mathbb{P}_\epsilon^L}}\bracket{\bracket{Q_\epsilon O_\mu+ \{I_{cl},O_\mu\}_\epsilon} e^{I_{cl}/\hbar+I_{qc}}}=e^{\hbar {\pa \over \pa \mathbb{P}_0^L}} (O_{d_{D_X}\mu} e^{I_{cl}/\hbar+I_{qc}}).
$$
The quantum master equation implies that the second term is
$$
\lim_{\epsilon\to 0}e^{\hbar {\pa \over \pa \mathbb{P}_\epsilon^L}}\bracket{O_\mu (Q_\epsilon+\hbar\Delta_\epsilon+F_{l_1}/\hbar)e^{I_{cl}/\hbar+I_{qc}}}=e^{\hbar {\pa \over \pa \mathbb{P}_0^L}} (O_\mu R e^{I_{cl}/\hbar+I_{qc}})=O_\mu[L]Re^{I[L]/\hbar}.
$$
It follows that
\begin{align*}
U_L(O_\mu[L])e^{I[L]/\hbar}=&e^{\hbar {\pa \over \pa \mathbb{P}_0^L}} (O_{d_{D_X}\mu}) e^{I_{cl}/\hbar+I_{qc}}\\
                           =&O_{d_{D_X}\mu}[L]e^{I[L]/\hbar},
\end{align*}
i.e., $
 U_L(\Psi(\mu))=\Psi(d_{D_X}\mu)
$
as desired.
\end{proof}

\begin{cor}\label{cohomology-quantum-local}
The cohomology of local quantum observables on a disk $U$ is given by
$$
  H^*\bracket{\text{Obs}^q(U)}\iso H^*\bracket{X, \wedge^*T_X}[[\hbar]].
$$
\end{cor}
\begin{proof} In fact, the map $\Psi$ defines a splitting of  \eqref{local-quantum-classical}.

\end{proof}

This says that the local observables do not receive quantum corrections via perturbative quantization, which is a very special property of B-model.

\subsubsection{Global quantum observable}
Now we consider global observables on the Riemann surface $\Sigma_g$. The cochain complex of
global quantum observables on $\Sigma_g$ at scale $L$ is defined as
\begin{equation}\label{eqn:global-quantum-observable}
\text{Obs}^q(\Sigma_g)[L]:=(\mathcal{O}(\mathcal{E})[[\hbar]],Q_L+\{I[L],-\}_L+\hbar\Delta_L).
\end{equation}
Since the complexes of quantum observables are homotopic equivalent for different length scales, we only need to
compute the cohomology of global observables at scale $L=\infty$.  By considering the $d_{\Sigma_g}$-cohomology first, the
complex (\ref{eqn:global-quantum-observable}) at $L=\infty$ is quasi-isomorphic to the following complex:
$$
 \bracket{\OO\bracket{\mathbb{H}^*(\Sigma_g)\otimes \bracket{\g_X[1]\oplus \g_X^\vee}}[[\hbar]],
l_1+\{I^{(0)}[\infty]|_{\H},-\}_{\infty}+\hbar(\{I^{(1)}[\infty]|_{\H},-\}_\infty+\Delta_\infty)},
$$
where $\H^*(\Sigma_g)$ denotes the space of harmonic forms on $\Sigma_g$. $I^{(0)}[\infty]|_{\H}$ and $I^{(1)}[\infty]|_{\H}$ are the
restrictions of the tree-level and one-loop effective interactions to the space of harmonic fields:
$$
\H:=\H^*(\Sigma_g)\otimes \bracket{\g_X[1]\oplus \g_X^\vee}.
$$

\begin{lem}\label{effective-infinity} Restricted to the harmonic fields at scale $L=\infty$, we have
$$
I^{(0)}[\infty]|_{\H}=I_{cl}|_{\H}, \quad I^{(1)}[\infty]|_{\H}=I_{qc}|_{\H}+I^{(1)}_{naive}[\infty]|_{\H}.
$$
\end{lem}
\begin{proof}
We only prove the first identity, and the second one can be  proved similarly. Let $\Gamma$ be a tree diagram with at least two vertices. We show that the Feynman weight
$W_{\Gamma}(\mathbb{P}_0^\infty, I_{cl})$ associated to $\Gamma$ vanishes when restricted to harmonic fields,
$$
 W_{\Gamma}(\mathbb{P}_0^\infty, I_{cl})|_{\H}=0.
$$
 We choose an orientation of the internal edges of $\Gamma$ such that every vertex is connected by a unique oriented path to a vertex
$v_\bullet$ in $\Gamma$, where $v_\bullet$ has only one edge which is oriented toward $v_\bullet$.  The vertex $v_\bullet$ will be
called the root. A vertex which has only one edge oriented outward will be called a leaf. By assumption, $\Gamma$ has at least
one leaf which is distinct from the root.

If a leaf has only $\H^0(\Sigma_g)$ and  $\H^2(\Sigma_g)$ inputs on its tails,  then the propagator $\mathbb{P}_0^\infty$ attached to its edge
will annihilates $ W_{\Gamma}(\mathbb{P}_0^\infty, I_{cl})|_{\H}$ since wedge products of harmonic $0$-forms and $2$-forms are still
harmonic, and
$$
{d^*}=0 \quad \mbox{on}\hspace{3mm} \H^*(\Sigma_g).
$$
Similarly, if a leaf has only one input of type $\H^1(\Sigma_g)$, then $W_{\Gamma}(\mathbb{P}_0^\infty, I_{cl})|_{\H}=0$. So we can assume that all leaves have at least two inputs of type
$\H^1(\Sigma_g)$ on their tails (possibly other inputs of type $\H^0(\Sigma_g)$). Since $\mathbb{P}_0^\infty$ is a $1$-form on $\Sigma_g\times\Sigma_g$, it is easy to see by tracing the
path that
the incoming edge of the root $v_\bullet$ has to contribute a $1$-form to the copy of $\Sigma_g$ corresponding to $v_\bullet$
which is $d^*$-exact, and by the type reason there is exactly one extra input of type $\H^1(\Sigma_g)$ on one tail of $v_\bullet$.
Since
$$
  \int_{\Sigma_g} d^*(a) \wedge b=0, \quad \forall a\in \A(\Sigma_g), b\in \H^*(\Sigma_g).
$$
This again implies that $W_{\Gamma}(\mathbb{P}_0^\infty, I_{cl})|_{\H}=0$.
\end{proof}
For later discussions on correlation functions of observables, we also need some description of the one-loop naive interaction as in the following Lemma:
\begin{lem}\label{lem:one-loop-naive-interaction-harmonic-fields}
For Riemann surfaces $\Sigma_g$ of genus $g=0$ and $g=1$, the infinity scale one-loop naive interaction vanishes when restricted to $\H$:
$$
I^{(1)}_{naive}[\infty]|_\H=0.
$$
\end{lem}
\begin{proof}
For both genus $0$ and genus $1$ Riemann surfaces with constant curvature metric, the product of harmonic forms remains harmonic. Thus by the same
argument as in the proof of Lemma \ref{effective-infinity}, if a one-loop graph $\gamma$ is a wheel with nontrivial trees attached to it, then
$$
W_\gamma(\mathbb{P}_0^\infty,I_{cl})|_\H=0.
$$
Hence we only need to deal with wheels. For the genus $0$ Riemann surface $\mathbb{P}^1$, since there are no harmonic $1$-forms, there must be at least one vertex on the wheel,
attached to which all inputs are harmonic $0$-forms by the type reason. The corresponding Feynman integral vanishes since the
composite of two propagators $\mathbb{P}_0^\infty$ on that vertex is zero by $(d^*)^2=0$.

For an elliptic curve $\Sigma_1=\mathbb{C}/(\mathbb{Z}+\mathbb{Z}\tau)$, if the number of vertices on a wheel is even, then the vanishing of the associated Feynman weight can be
proved by the same argument as in Lemma \ref{lem:two-vertex-wheels}. For a wheel with odd number of vertices, a $\mathbb{Z}_2$-symmetry of the analytic propagator
$P_0^\infty$ results in the vanishing of the Feynman weights:  Let $dw$ be a harmonic $1$-form on $\Sigma_1$, and we assume without lost of generality that all the vertices of the
wheel are trivalent, and all the inputs are
of the form $dw\otimes g_i$, where $g_i\in\g_X$.
$$
\figbox{0.2}{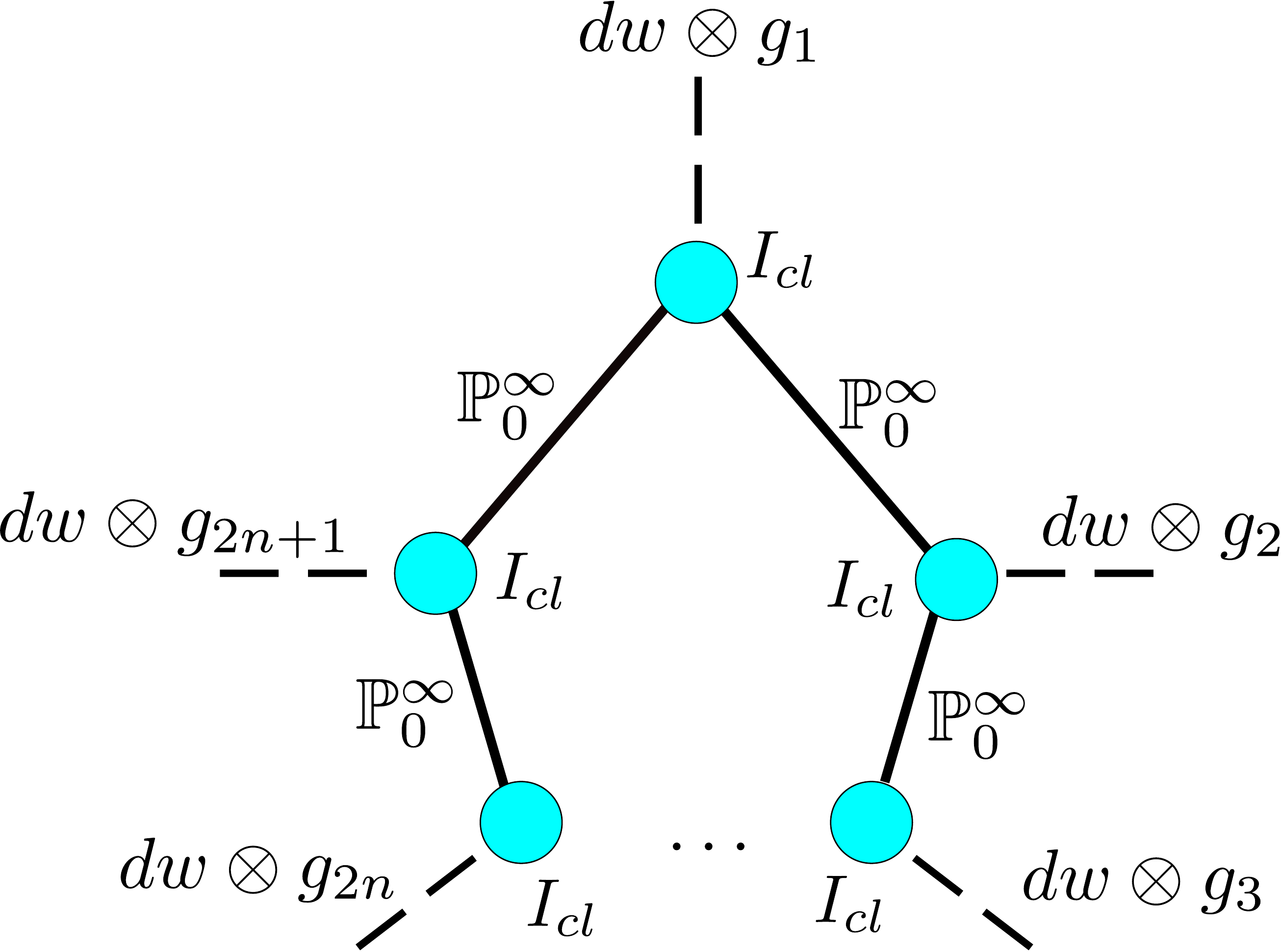}
$$
Similar to \cite[Lemma 17.4.4]{Kevin-CS}, the analytic part of the corresponding Feynman weight $ W(P_0^\infty,I_{cl})(dw)$ will be a linear combination of
$$
\sum_{(a,b)\in\mathbb{Z}^2\setminus\{0\}}\dfrac{1}{(a\tau+b)^{k}(a\bar\tau+b)^{2n+1-k}},
$$ which clearly vanishes.
\end{proof}

Now let us compute the cohomology of the global quantum observables.
\begin{rmk}
In the following discussion, the harmonic forms $\H^k(\Sigma_g)$ sit at degree $k$, and $\Omega_X^1[1]\cong T_X^\vee[1]$ sits at degree $-1$.
\end{rmk}
%Let
%$$
%  \H^*(\Sigma_g):= \bigoplus\limits_{k=0}^2 \H^k[-k],
%$$
%where $\H^k:= \H^k(\Sigma_g)$ is concentrated at degree $k$.
\begin{lem}\label{splitting-isom}
There is a natural isomorphism of $\A_X$-modules
\begin{equation}\label{eqn:structure-sheaf-ringed-space-T-g}
\begin{aligned}
 &\OO\bracket{\mathbb{H}^*(\Sigma_g)\otimes \bracket{\g_X[1]\oplus \g_X^\vee}}
\iso\\
&\A_X\otimes_{\OO_X}
 \Jet_X^{hol}\bracket{\widehat{\Sym}\bracket{T_X\otimes \H^1(\Sigma_g)}^\vee \otimes \widehat{\Sym}\bracket{T_X\otimes \H^2(\Sigma_g)}^\vee
\otimes \widehat{\Sym}(\Omega^1_X[1]\otimes\H^*(\Sigma_g))^\vee}
\end{aligned}
\end{equation}
\end{lem}
\begin{proof}
We have the following isomorphisms:
\begin{align*}
\mathcal{O}(\H^*(\Sigma_g)\otimes(\g_X[1]\oplus\g_X^\vee))&\cong\mathcal{O}(\H^0(\Sigma_g)\otimes\g_X[1])\otimes_{\A_X}\mathcal{O}
\left(\bigoplus_{k=1}^{2}\H^k(\Sigma_g)\otimes\g_X[1]\oplus\bigoplus_{k=0}^2\H^k(\Sigma_g)\otimes\g_X^\vee\right)\\
&\cong C^*(\g_X)\otimes_{\A_X}\widehat{\Sym}
\left(\bigoplus_{k=1}^{2}\H^k(\Sigma_g)\otimes\g_X[1]\oplus\bigoplus_{k=0}^2\H^k(\Sigma_g)\otimes\g_X^\vee\right)^\vee.\\
\end{align*}
It is clear that the tensor products of the isomorphisms $\tilde{T}$ and $\tilde{K}$ in Propositions
\ref{proposition:T-dual-to-jet} and
\ref{proposition:T_X-to-jet} respectively give the desired isomorphism.

\end{proof}

\begin{defn}
Let $\pi_2^*\bracket{\Omega_X^{2g-2}}$ denote the canonical flat section of the jet bundle
$$
 \Jet_X^{hol}\bracket{{\Sym}^{2g\cdot \dim_{\C}X}\bracket{T_X\otimes \H^1(\Sigma_g)}^\vee \otimes{\Sym}^{\dim_{\C}X}(\Omega^1_X[1]\otimes \H^0(\Sigma_g))^\vee
\otimes{\Sym}^{\dim_{\C}X}(\Omega^1_X[1]\otimes \H^2(\Sigma_g))^\vee}
$$
induced by the holomorphic volume form $\Omega_X$. Here we use the notation $\pi_2^*$  to be consistent with Definition \ref{def:jet-bundle}.
\end{defn}

We can view $\pi_2^*\bracket{\Omega_X^{2g-2}}$ as a quantum observable via the identification in Lemma \ref{splitting-isom}.  The general philosophy in \cite{Kevin-CS} says that
the quantization gives rise to a projective volume form, and the next proposition says that the volume form is exactly given by $\pi_2^*\bracket{\Omega_X^{2g-2}}$.

\begin{prop}\label{trivial-system}
The following embedding
$$
  \iota: \A_X((\hbar)) \into  \OO\bracket{\mathbb{H}^*(\Sigma_g)\otimes \bracket{\g_X[1]\oplus \g_X^\vee}}((\hbar))
$$
defined by
$$
   A\mapsto \iota(A):= \hbar^{-2\dim_{\C}X}A\otimes_{\mathcal{O}_X} {\pi_2^*\bracket{\Omega_X^{2g-2}}}, \quad \forall A\in \A_X
$$
is a quasi-isomorphism which is equivariant with respect to the $\C^\times$-symmetry defined in Section \ref{section:classical-symmetry}.
\end{prop}

\begin{proof} We first show that $\iota$ respects the differential. This is equivalent to showing that
$\pi_2^*\bracket{\Omega_X^{2g-2}}$ is closed under
$l_1+\fbracket{I^{(0)}[\infty]|_{\H},-}_{\infty}+\hbar(\{I^{(1)}[\infty]|_{\H},-\}_\infty+\Delta_\infty)$.  By Lemma
\ref{effective-infinity},
$$
\bracket{l_1+\fbracket{I^{(0)}[\infty]|_{\H},-}_{\infty}} \bracket{\pi_2^*\bracket{\Omega_X^{2g-2}}}=d_{D_X}
\bracket{\pi_2^*\bracket{\Omega_X^{2g-2}}}=0,
  $$
since ${\pi_2^*\bracket{\Omega_X^{2g-2}}}$ is flat. Here $d_{D_X}$ is the de Rham differential of the $D_X$-module.
\begin{claim}
$
  \fbracket{I^{(1)}_{naive}[\infty], \bracket{\pi_2^*\bracket{\Omega_X^{2g-2}}}}_\infty=0.
$
\end{claim}
\begin{proof}
It is straightforward to check (similar to the proof of Lemma \ref{effective-infinity}) that for any one-loop graph $\gamma$, either the Feynman weight $W_\gamma(\mathbb{P}_0^\infty, I_{cl})$
vanishes, or the operator $\{W_\gamma(\mathbb{P}_0^\infty, I_{cl}),-\}_\infty$ applied to $ \pi_2^*\bracket{\Omega_X^{2g-2}}$will generate
new terms in $(\H^1(\Sigma_g)\otimes\g_X[1])^\vee$. These terms force the bracket $\fbracket{W_\gamma(\mathbb{P}_0^\infty,
I_{cl}),\pi_2^*\bracket{\Omega_X^{2g-2}}}_\infty$ to vanish since $\pi_2^*\bracket{\Omega_X^{2g-2}}$ already contains the
highest wedge product of $(\H^1(\Sigma_g)\otimes\g_X[1])^\vee$.
\end{proof}

Hence we only need to consider the operator $\Delta_\infty+\{I_{qc},-\}_\infty$. Let $n=\dim_\C X$, the map
$$
   \wedge^n\bracket{\pa_{dR}\circ T}: \Omega_X^n\to \cinfty(X)\otimes_{\OO_X}\Jet^{hol}_X(\Omega_X^n)
$$
in Proposition \ref{correction-formula} induces a natural embedding
\begin{align*}
 T':&\bracket{\Omega_X^n}^{\otimes(2g-2)}\hookrightarrow\\
 &\hspace{3mm}\Jet_X^{hol}\bracket{{\Sym}^{2g\cdot n}\bracket{T_X\otimes \H^1(\Sigma_g)}^\vee \otimes{\Sym}^{n}\bracket{\Omega^1_X[1]\otimes
\H^0(\Sigma_g)}^\vee\otimes{\Sym}^{n}\bracket{\Omega^1_X[1]\otimes \H^2(\Sigma_g)}^\vee}.
\end{align*}
Let $T'\bracket{\Omega_X^{2g-2}}$ denote the image of the section $(\Omega_X)^{\otimes(2g-2)}$, where $\Omega_X$ denotes the volume form and its negative power denotes its dual.
By construction,
$$
 \Delta_\infty T'\bracket{\Omega_X^{2g-2}}=0
$$
and by Proposition \ref{correction-formula}, we have
$$
   \pi_2^*\bracket{\Omega_X^{2g-2}}=e^{-I_{qc}}T'\bracket{\Omega_X^{2g-2}}.
$$
It follows that
$$
  \Delta_\infty  \pi_2^*\bracket{\Omega_X^{2g-2}}=\Delta_\infty
\bracket{e^{-I_{qc}}T'\bracket{\Omega_X^{2g-2}}}=-e^{-I_{qc}}\fbracket{I_{qc},T'\bracket{\Omega_X^{2g-2}}}_\infty=-\fbracket{I_{qc}, \pi_2^*\bracket{\Omega_X^{2g-2}}}_\infty,
$$
as desired.

Now we show that $\iota$ is a quasi-isomorphism. We consider the filtration on $ \OO\bracket{\mathbb{H}^*(\Sigma_g)\otimes
\bracket{\g_X[1]\oplus \g_X^\vee}}((\hbar))$ by the degree of the differential forms on $X$:
$$
 F^k   \OO\bracket{\mathbb{H}^*(\Sigma_g)\otimes \bracket{\g_X[1]\oplus \g_X^\vee}}((\hbar)):=\A_X^k\cdot  \OO\bracket{\mathbb{H}^*(\Sigma_g)\otimes
\bracket{\g_X[1]\oplus
\g_X^\vee}}((\hbar)).
$$
The differential of the graded complex given by
$$
   d_1= \hbar \bracket{\Delta_\infty+\fbracket{I_{qc},-}_\infty}=\hbar e^{-I_{qc}}\Delta_\infty e^{I_{qc}}.
$$
By the Poincare lemma below, the $d_1$-cohomology is precisely given by $\Im(\iota)$.  It follows that $\iota$ is a quasi-isomorphism.
\end{proof}

Recall the following Poincare lemma:
\begin{lem}\label{lem:Poincare-lemma}Let $\{x^i\}$ be even elements and let $\{\xi_i\}$ be odd elements, then we have
$$
H^*\left(\mathbb{C}[[x^i,\xi_i]],\Delta=\dfrac{\partial}{\partial
x^i}\dfrac{\partial}{\partial\xi_i}\right)=\mathbb{C}\xi_1\wedge\cdots\wedge\xi_n.
$$
\end{lem}

\begin{cor} \label{quantum-obs-cohomology}
 The top cohomology of $\text{Obs}^q(\Sigma_g)[\hbar^{-1}]$ is at degree $(2-2g)\dim_\C X$, given by
$$
   H^{(2-2g)\dim_\C X}\bracket{\text{Obs}^q(\Sigma_g)[\hbar^{-1}]}\iso \C((\hbar)).
$$
\end{cor}

\subsection{Correlation function}\label{section:correlation-function}
Proposition \ref{trivial-system} implies that a quantization $I[L]$ defines an integrable projective volume form in the sense of
\cite{Kevin-CS}, which allows us to define correlation functions for quantum observables.
\begin{defn}
Let $O\in \text{Obs}^q(\Sigma_g)$ be a closed element, representing a cohomology class $[O]\in H^k\bracket{\text{Obs}^q(\Sigma_g)}$. We define its correlation
function (via Corollary \ref{quantum-obs-cohomology}) by
$$
  \abracket{O}_{\Sigma_g}:=\begin{cases} 0 & \text{if}\ k\neq (2-2g)\dim_\C X\\  [O]\in \C((\hbar)) & \text{if}\ k=  (2-2g)\dim_{\C} X \end{cases}
$$
\end{defn}

Recall that the local quantum observables form a factorization algebra on $\Sigma_g$. This structure enables us to define
correlation functions for local observables, for which let us first introduce some notations.
\begin{defn}\label{definition:factorization-product}
Let $U_1, \cdots, U_n$ be disjoint open subsets of $\Sigma_g$. The factorization product
$$
  \text{Obs}^q(U_1)\times \cdots \times \text{Obs}^q(U_n)\to \text{Obs}^q(\Sigma_g)
$$
of local observables $O_i\in \text{Obs}^q(U_i)$ will be denoted by $O_1\star\cdots\star O_n$. This product descends to cohomologies.
\end{defn}

\begin{defn}  Let $U_1, \cdots, U_n$ be disjoint open subsets of $\Sigma_g$, and let $O_{i} \in \text{Obs}^q(U_i)$ be closed
local quantum observables supported on $U_i$. We define their correlation function by
$$
 \abracket{O_1,\cdots, O_n}_{\Sigma_g}:=\abracket{O_1\star\cdots \star O_n }_{\Sigma_g}\in \C((\hbar)).
$$
\end{defn}

We would like to compute the correlation functions for  B-twisted topological $\sigma$-model. By the degree reason, the only
nontrivial cases are $g=0,1$.  We show that they coincide with the prediction from physics.

For later computation, we give an equivalent definition of correlation functions via the BV integration point of view. Let
$\widehat{T^*\bracket{{T^{\Sigma_g}X}}}[-1]$ denote the ringed space with underlying topological space $X$ and structure sheaf as that of
(\ref{eqn:structure-sheaf-ringed-space-T-g}):
$$
\widehat{T^*\bracket{{T^{\Sigma_g}X}}}[-1]=\left(X,\OO\bracket{\mathbb{H}^*(\Sigma_g)\otimes \bracket{\g_X[1]\oplus
\g_X^\vee}}\right).
$$
It is clear that the intersection pairing on $\H^*(\Sigma_g)$, together with the canonical pairing between $\g_X[1]$ and
$\g_X^\vee$, induces an odd symplectic structure on $\widehat{T^*\bracket{{T^{\Sigma_g}X}}}[-1]$ . Let $\mathcal L$ be
the ringed space with underlying topological space $X$ and structure sheaf being generated by the odd generators over
$\A_X$  in $\OO\bracket{\mathbb{H}^*(\Sigma_g)\otimes \bracket{\g_X[1]\oplus\g_X^\vee}}$. Thus
$\mathcal L$ can be viewed as a Lagrangian subspace of $\widehat{T^*\bracket{{T^{\Sigma_g}X}}}[-1]$.

It is clear from the form of the jet bundle in \eqref{eqn:structure-sheaf-ringed-space-T-g} that there is a canonical
projection of functions on $\widehat{T^*\bracket{{T^{\Sigma_g}X}}}[-1]$ to the subspace
$$
\A_X\otimes_{\OO_X}
 \Jet_X^{hol}\bracket{ \widehat{\Sym}\bracket{T_X\otimes \H^2(\Sigma_g)}^\vee \otimes \widehat{\Sym}\bracket{\Omega^1_X[1]\otimes
\H^1(\Sigma_g)}^\vee}\pi_2^*(\Omega_X^{2g-2})
$$
generated by $\pi_2^*(\Omega_X^{2g-2})$. We denote by $(f)^{TF}$ the projection of $f$, where $TF$ is short for ``top fermions''.
\begin{prop}
The map
\begin{align*}
i_L^*:\text{Obs}^q(\Sigma_g)&\rightarrow \A_X((\hbar))[(2g-2)\dim_{\C} X]\\
    O&\mapsto \hbar^{2\dim_{\C}X} \left((e^{I[\infty]/\hbar}\cdot O|_{\H})^{TF}\big/\pi_2^*(\Omega_X^{2g-2})\right)\Big|_{\mathcal{L}}
\end{align*} is a $\C^\times$-equivariant cochain map. Here the differential on the right hand side is the de Rham differential $d_X$, and
$$
  \Big|_{\mathcal{L}}: \A_X\otimes_{\OO_X}
 \Jet_X^{hol}\bracket{ \widehat{\Sym}^*\bracket{T_X\otimes \H^2(\Sigma_g)}^\vee \otimes \widehat{\Sym}^*\bracket{\Omega^1_X[1]\otimes
\H^1(\Sigma_g)}^\vee} \to \A_X
$$
denotes the map which sets all the jet coordinates and that of $T_X\otimes \H^2(\Sigma_g),\Omega^1_X[1]\otimes\H^1(\Sigma_g)$ to 	be zero.
\end{prop}
\begin{proof} Let $O\in \text{Obs}^q(\Sigma_g)$. Recall that from QME, we have
$$
(Q_\infty O+\hbar\Delta_\infty O+\{I[\infty],O\}_\infty)\cdot
e^{I[\infty]/\hbar}=(Q_\infty+\hbar\Delta_\infty+\frac{F_{l_1}}{\hbar}-R)(e^{I[\infty]/\hbar}\cdot O).
$$
We have the following two simple observations:
\begin{enumerate}
 \item $\left(\hbar\Delta_\infty(e^{I[\infty]/\hbar}O)\right)^{TF}=0$ since $\Delta_\infty$ annihilates one odd generator,
 \item $\dfrac{F_{l_1}}{\hbar}(e^{I[\infty]/\hbar}O)$ vanishes when restricted to the Lagrangian $\mathcal{L}$, since
$F_{l_1}$ contains non-trivial bosonic generators.
\item When restricted to $\H$, $Q_\infty=l_1$.
\end{enumerate}
Thus, we only need to show that
$$
\left(\left((Q_\infty-R)(e^{I[\infty]/\hbar}\cdot
O\big|_{\H})\right)^{TF}/\pi_2^*(\Omega_X^{2g-2})\right)\bigg|_{\mathcal{L}}=d_X\left((e^{I[\infty]/\hbar} \cdot
O\big|_{\H})^{TF}/\pi_2^*(\Omega_X^{2g-2})\right)\Big|_\mathcal{L}.
$$
Since $l_1$ commutes with $ \Big|_{\mathcal{L}}$, we can assume that $(e^{I[\infty]/\hbar}O\big|_{\H})^{TF}$ is of the form
\begin{align*}
(e^{I[\infty]/\hbar}O\big|_{\H})^{TF}=B\cdot \pi_2^*(\Omega^{2g-2})
\overset{(1)}{=}B\cdot e^{-I_{qc}} T'(\Omega_X^{2g-2}),
\end{align*}
 where $B\in\A_X$, and identity $(1)$ and the map $T'$ is explained in the proof of proposition \ref{trivial-system}. Then
\begin{align*}
Q_\infty(B\cdot \pi_2^*(\Omega_X^{2g-2}))
=&d_X(B)\cdot \pi_2^*(\Omega_X^{2g-2})+B\cdot l_1(\pi_2^*(\Omega_X^{2g-2})).	
\end{align*}
The fact that $\pi_2^*(\Omega_X^{2g-2})$ is a flat section of the jet bundle is translated to
$$
l_1(\pi_2^*(\Omega_X^{2g-2}))+\{I_{cl},\pi_2^*(\Omega_X^{2g-2})\}=0,
$$
where $I_{cl}=\tilde{l}_0+\sum_{k\geq 2}\tilde{l}_k$ is the classical interaction functional and $\tilde{l}_k$ is defined in equation \eqref{def:
tilde-l_k}.

The functionals $\{\tilde{l}_k,\pi_2^*(\Omega_X^{2g-2})\}$ for $k\geq 2$ vanish when
restricted to the Lagrangian $\mathcal{L}$ since they contain jet coordinates. Thus
\begin{align*}
\left(l_1(\pi_2^*(\Omega_X^{2g-2}))\right)\Big|_\mathcal{L} &=-\{\tilde{l}_0,\pi_2^*(\Omega_X^{2g-2})\}\Big|_\mathcal{L}\\
                                                  &=-\{\tilde{l}_0,e^{-I_{qc}}T'(\Omega_X^{2g-2})\}\Big|_\mathcal{L}\\
                                                  &=-\left(-\{\tilde{l}_0,I_{qc}\}\cdot
 e^{-I_{qc}}T'(\Omega_X^{2g-2})+e^{-I_{qc}}\cdot\{\tilde{l}_0,T'(\Omega_X^{2g-2})\}\right)\Big|_{\mathcal{L}}\\
                                                  &\overset{(1)}{=}\left(\{\tilde{l}_0,I_{qc}\}\cdot
e^{-I_{qc}}T'(\Omega_X^{2g-2})\right)\Big|_{\mathcal{L}}\\
                                                  &\overset{(2)}{=}R\cdot\pi_2^*(\Omega_X^{2g-2}).
\end{align*}
The identity $(1)$ follows from the fact that $\{\tilde{l}_0,T'(\Omega_X^{2g-2})\}=0$ by type reason, and identity $(2)$
follows from Lemma \ref{lem:S_1-contant term}. Thus we have
\begin{align*}
&\left(\left((Q-R)(e^{I[\infty]/\hbar}\cdot O)\big|_{\H}\right)^{TF}/\pi_2^*(\Omega_X^{2g-2})\right)\bigg|_{\L}\\
=&\left((Q-R)(B\cdot \pi_2^*(\Omega^{2g-2}))/\pi_2^*(\Omega_X^{2g-2})\right)\Big|_{\L}\\
=&\left((dB+B\cdot R-B\cdot R)\cdot \pi_2^*(\Omega^{2g-2}))/\pi_2^*(\Omega_X^{2g-2})\right)\Big|_{\L}\\
=& dB.
\end{align*}
\end{proof}

It is clear that the cochain map $i_L^*$ induces an isomorphism on the degree $(2-2g)\dim_{\mathbb{C}} X$ component. Thus we have the following
corollary:
\begin{cor}\label{cor:correlation-function-BV-def}
Let $O$ be a global quantum observable which is closed, then the correlation function of $O$ is the same as the integral of $i_L^*(O)$
on $X$: $$\langle O\rangle_{\Sigma_g}=\hbar^{2\dim_{\C}X}\int_X\left((e^{I[\infty]/\hbar}\cdot O\big|_{\H})^{TF}/\pi_2^*(\Omega_X^{2g-2})\right)\Big|_{\L}.$$
\end{cor}

We are ready to compute the topological correlation functions on $\mathbb{P}^1$ and elliptic curves.

\subsubsection{$g=0$}
\begin{lem}\label{lem:factorization-product-genus-0}
Let $\{U_i\}$ be disjoint union of disks contained in a larger disk $U\subset \Sigma_g$, and let $[O_{\mu_i, U_i}] \in H^*\bracket{\text{Obs}^q(U_i)}$
be the local quantum observable associated to
$\mu_i \in H^*(X, \wedge^*T_X)$ on $U_i$. Then
$$
    [O_{\mu_1, U_1}\star\cdots\star O_{\mu_m, U_m}]=[O_{\mu_1\cdots\mu_m, U}]\in H^*\bracket{\text{Obs}^q(U)}.
$$
\end{lem}
\begin{proof}
 For any parametrix $\Phi$, we have
\begin{align*}
(O_{\mu_1, U_1}\star O_{\mu_2, U_2})[\Phi]&\overset{(1)}{=}\lim_{L\rightarrow 0}W(P(\Phi)-\mathbb{P}_0^L,I[L],O_{\mu_1,U_1}[L]\star
O_{\mu_2,U_2}[L])\\
&\overset{(2)}{=}W(P(\Phi),I_{cl}+\hbar I_{qc}, O_{\mu_1,U_1}\cdot O_{\mu_2,U_2})\\
&=O_{\mu_1\mu_2,U}[\Phi].
\end{align*}
Here identity $(1)$ is the definition of factorization product of observables, and identity $(2)$ follows from Proposition
\ref{prop:local-observable}.
\end{proof}

\begin{thm}\label{correlation-P1}
 Let $\Sigma_g=\mathbb{P}^1$, and $\{U_i\}$ be disjoint union of disks on $\mathbb P^1$. Let $O_{\mu_i, U_i} \in
H^*\bracket{\text{Obs}^q(U_i)}$ be a local quantum
observable associated to $\mu_i \in H^*(X,\wedge^*T_X)$ supported in $U_i$. Then
$$
\abracket{O_{\mu_1, U_1}, \cdots, O_{\mu_m, U_m}}_{\mathbb{P}^1}=\hbar^{\dim_{\C}X}\int_X \bracket{\mu_1\cdots \mu_m\vdash \Omega_X}\wedge \Omega_X.
$$
\end{thm}
\begin{proof}
We compute the correlation function at the scale $L=\infty$. By the degree reason and the previous lemma, we can assume $m=1$, and $\mu=\mu_1\in H^{\dim_{\C}X}(X, \wedge^{\dim_{\C}X}T_X)$. Let $O_\mu$ be the classical observable represented by $\mu$. By Proposition \ref{prop:local-observable}, the corresponding quantum observable is described by
$$
O_{\mu}[\infty] e^{I[\infty]/\hbar}=\lim_{L\to\infty}\lim_{\epsilon\to 0}e^{\hbar {\pa\over \pa \mathbb P_\epsilon^L}}\bracket{O_\mu
e^{I_{cl}/\hbar+I_{qc}}}.
$$

Since $\H^1(\mathbb{P}^1)=0$, a similar argument as in Lemma \ref{effective-infinity} implies
$$
   O_{\mu}[\infty] \big|_{\mathbb H}=O_{\mu}
$$
when restricted to harmonic fields. By Corollary \ref{cor:correlation-function-BV-def},
$$
\abracket{O_\mu[\infty]}_{\mathbb{P}^1}=\hbar^{2\dim_{\C}X}\int_X\left((e^{I[\infty]/\hbar}\cdot O_\mu\big|_{\H})^{TF}/\pi_2^*(\Omega_X^{-2})\right)\Big|_{\L}.
$$
By Lemma \ref{lem:one-loop-naive-interaction-harmonic-fields} and Lemma \ref{effective-infinity},
$$
(e^{I[\infty]/\hbar}\cdot O_\mu\big|_{\H})^{TF}=(e^{I_{cl}/\hbar+I_{qc}}\cdot O_\mu\big|_{\H})^{TF}.
$$
By the type reason, the only terms in $e^{I_{cl}/\hbar+I_{qc}}$ that will contribute after $\big|_{\L}$ are products of $\tilde{l}_0$:
$$
\figbox{0.26}{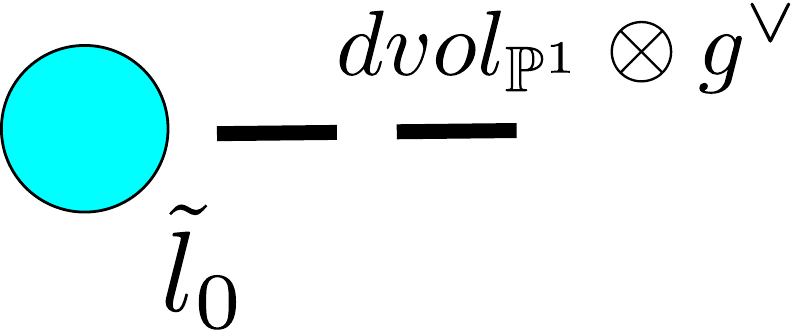}
$$

Let $\{z^i\}$ be holomorphic  coordinates on $U\subset X$ such that locally $\Omega_X|_U=dz^1\wedge\cdots\wedge dz^n$. Let
$$
\mu=A d\bar z^1\wedge\cdots \wedge d\bar z^n\otimes \pa_{z^1}\wedge\cdots \wedge \pa_{z^n},
$$
where $n=\dim_{\C} X$. We can choose the following element  in the jet bundle representing $\mu$:
$$
  O_\mu=Ad\bar z^1\wedge\cdots \wedge d\bar z^n\otimes_{\OO_X} \left((\pi_2^*(dz^{1})\otimes
1)^\vee\wedge\cdots\wedge(\pi_2^*(dz^{n})\otimes 1)^\vee\otimes\mathbf{1}\right),
$$
where $\mathbf{1}$ denotes the other component in the jet bundle. On the other hand,
$$
   e^{\tilde l_0/\hbar}= \hbar^{-n}\cdot dz^1\wedge\cdots\wedge dz^n\otimes_{\OO_X}\bracket{T(\pa_{z^1})\wedge\cdots\wedge T(\pa_{z^n})} +\text{lower
wedge products}.
$$
It follows easily that
$$
\left((e^{I[\infty]/\hbar}\cdot O_\mu\big|_{\H})^{TF}/\pi_2^*(\Omega_X^{-2})\right)\Big|_{\L}=\hbar^{-n}\cdot(\mu\vdash\Omega_X)\wedge \Omega_X.
$$
\end{proof}

\subsubsection{$g=1$} On elliptic curves, the only nontrivial input is at cohomology degree $0$.
\begin{thm}\label{correlation-elliptic} Let $E=\Sigma_1$ be an elliptic curve. Then $
   \abracket{1}_E=\chi(X)
$ is the Euler characteristic of $X$.
\end{thm}
\begin{proof}
By Corollary \ref{cor:correlation-function-BV-def}, we only need to look at the term $e^{I[\infty]/\hbar}$. For the case $g=1$, we have shown in Lemma
\ref{lem:one-loop-naive-interaction-harmonic-fields}  that $I_{naive}^{(1)}[\infty]\big|_{\H}$ vanishes, and the quantum correction $I_{qc}$ also vanishies. Hence
$I[\infty]=I^{(0)}[\infty]=I_{cl}$. Let $w$ denote the normalized holomorphic coordinate on the elliptic curve $E$ such that
$$
\sqrt{-1}\int_{E}dwd\bar{w}=1.
$$
It is not difficult to see that by type reason, the only terms in  $e^{I_{cl}/\hbar}|_{\H}$ that can survive under $\big|_{\L}$ are the following:
\begin{equation}\label{eqn:genus-1-correlation-survival-term}
 (1)\hspace{5mm}\figbox{0.26}{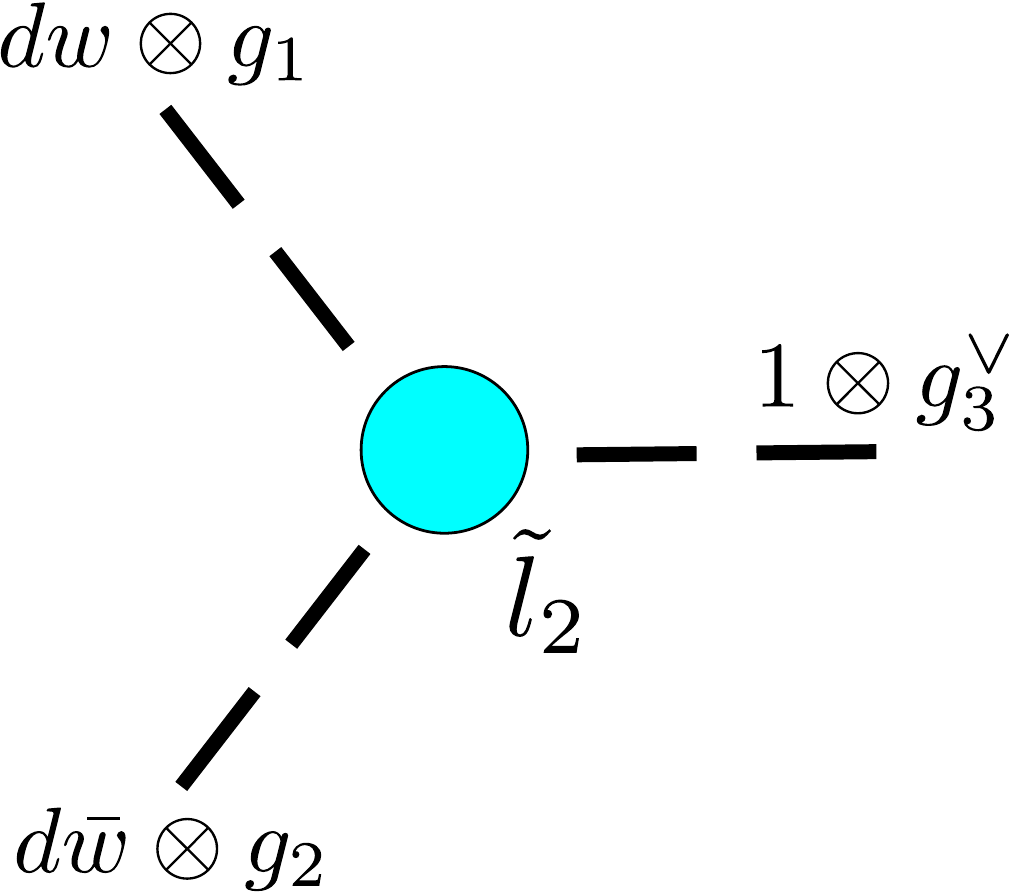},\hspace{20mm}(2)\hspace{5mm}\figbox{0.26}{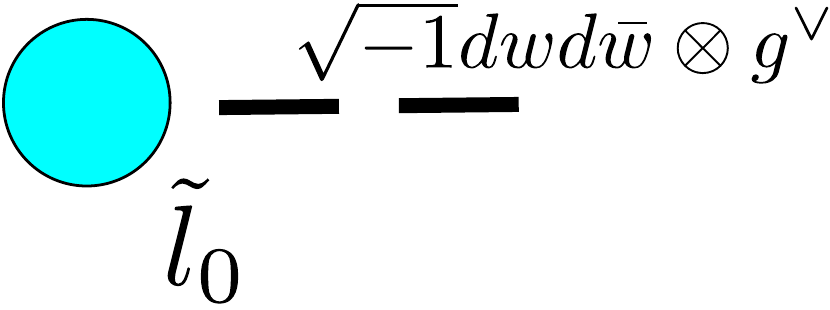}
\end{equation}
Let $\{z^i\}$ be local holomorphic coordinates on $X$ as we chose before, then term $(1)$ in (\ref{eqn:genus-1-correlation-survival-term}) can be expressed explicitly as:
\begin{equation}\label{eqn:genus-term1}
A_{ij}^k\otimes_{\OO_X}((dw)^\vee\otimes \tilde{T}(\widetilde{dz^i}))\otimes ((d\bar{w})^\vee\otimes
\tilde{T}(\widetilde{dz^j}))\otimes(1^\vee\otimes\tilde{K}(\widetilde{\partial_{z^k}})).
\end{equation}
And term (2) can be expressed as
\begin{equation}\label{eqn:genus-term2}
dz^l\otimes\left((\sqrt{-1}dwd\bar{w})^\vee\otimes  \tilde{K}(\widetilde{\partial_{z^l}})\right).
\end{equation}

By the discussion in \cite{Kevin-CS}, the following differential form valued in the bundle $End(T_X)=T_X\otimes_{\mathcal{O}_X}T_X^\vee$
$$
(A_{ij}^kdz^i)\otimes_{\OO_X}(dz^j\otimes\frac{\partial}{\partial z^k})
$$
is exactly the Dolbeault representative of the Atiyah class of the holomorphic tangent bundle $T_X$. It is straightforward to
check that
\begin{align*}
\left(\left(\exp(I[\infty]/\hbar)|_{\H}\right)^{TF}/\pi_2^*(\Omega_X^{2-2g})\right)\Big|_\mathcal{L}&=\Tr\left((A_{ij}
^kdz^i)\otimes(dz^j\otimes\frac { \partial } { \partial z^k})\right)^{n}\\
                                                        &=\Tr(At(T_X))^n\\
                                                        &=c_n(X).
\end{align*}
It then follows easily that
\begin{align*}
\langle 1\rangle_E&=\int_X c_n(X)\\
                &=\chi(X).
\end{align*}

\end{proof}

\section{Landau-Ginzburg Twisting}
In this section we discuss the Landau-Ginzburg twisting of the B-twisted $\sigma$-model and establish the topological Landau-Ginzburg
B-model via the renormalization method.
\begin{rmk} To avoid confusion,  ``twisted'' and ``untwisted'' in this section are always concerned with the twist that arises from the superpotential $W$, instead of
the $B$-twist.
\end{rmk}

\subsection{Classical theory} Let $X$ be a smooth variety with a holomorphic function
$$
  W: X\to \C.	
$$
Recall that the B-twisted $\sigma$-model describes maps
$$
  \bracket{\Sigma_g}_{dR}\to T^*X_{\dbar}[1].
$$
 Let
$$
  dW\lrcorner: \wedge^* T_X\to \wedge^* T_X
$$
be the contraction with the 1-form $dW$. It induces a differential on $\OO\bracket{T^*X_{\dbar}[1]}$ of degree $-1$, commuting
with $d_{D_X}$. By abuse of notations, we still denote this differential by $dW\lrcorner$.

\begin{defn} We define $ T^*X_{\dbar}^W[1]$ to be the $L_\infty$-space with underlying space $X$, and  sheaf of functions the
$\Z_2$-graded complex
$$
   \OO\bracket{ T^*X_{\dbar}^W[1]}:=\A_X\otimes_{\OO_X}\Jet^{hol}_X\bracket{\wedge^* T_X}
$$
with the twisted differential $d_{D_X}+{dW\lrcorner}$.
\end{defn}

\begin{rmk} When $X=\C^n$, and $W$ being a weighted homogeneous polynomial, i.e.,
$$
  W(\lambda^{q_i}z^i)=\lambda W(z^i), \quad \forall \lambda\in \C^*
$$
where $q_i$'s are positive rational numbers called the weights, then there emerges a $\Q$-grading by assigning the weights:
$
  wt(z^i)=q_i,  wt\bracket{\pa_{z^i}}=1-q_i,  wt(\bar z^i)= -q_i,  wt(d\bar z^i)=-q_i
$. However, we will not explore further on this in the current paper.
\end{rmk}

Note that there is a quasi-isomorphism of $\Z_2$-graded complexes of sheaves
$$
 \OO\bracket{ T^*X_{\dbar}^W[1]}\iso \bracket{\A_X^{0,*} \bracket{\wedge^* T_X}, \dbar_W}, \quad \dbar_W=\dbar+dW\lrcorner.
$$
Therefore $ T^*X_{\dbar}^W[1]$ can be viewed as the derived critical locus of $W$. The Landau-Ginzburg $B$-model describes
the space of maps
$$
   \bracket{\Sigma_g}_{dR}\to  T^*X_{\dbar}^W[1].
$$
As in the untwisted case, we consider those maps in the formal neighborhood of constant maps. This corresponds to the physics
statement that path integrals in B-twisted
Landau-Ginzburg model are localized around the neighborhood of constant maps valued in the critical locus of $W$.

Recall that there exists a Poisson bracket (Schouten-Nijenhuis bracket): $\wedge^*T_X\otimes_{\C} \wedge^*T_X\to \wedge^*T_X$.
Viewed as a bi-differential operator, it induces a bracket on the jet bundles, which we denote by
$$
  \{-,-\}_{T^*X_{\dbar}[1]}:   \OO\bracket{ T^*X_{\dbar}[1]} \otimes_{\A_X}   \OO\bracket{ T^*X_{\dbar}[1]} \to   \OO\bracket{
T^*X_{\dbar}[1]}.
$$

\begin{lem-defn} Let $\widetilde{W} \in \OO\bracket{ T^*X_{\dbar}[1]}$ be the image of $W$ under the natural embedding
$$
  \OO_X\into \Jet^{hol}_X(\OO_X) \into \A_X\otimes_{\OO_X}\Jet^{hol}_X\bracket{\wedge^* T_X}.
$$
Then $\{\widetilde{W},-\}_{T^*X_{\dbar}[1]}=dW\lrcorner$ as operators on $\OO\bracket{ T^*X_{\dbar}[1]}$.
\end{lem-defn}
\begin{proof} This follows from the corresponding statement on $\wedge^*T_X$.
\end{proof}
\iffalse
\begin{proof}
Let $\{\widetilde{\partial_{z^i}}\}$ and $\{\widetilde{dz^i}\}$ denote the basis of the $L_\infty$ algebra $\g_X$ and the dual
$\g_X^\vee$ over $\A_X$, and let $\{\widehat{dz^i}\}$ denote the basis of the $\g_X$-module $C^*(\g_X,\g_X^\vee) $. Let
$\tilde{K}$ denote the smooth splitting
$$\tilde{K}: \Omega_X^1\rightarrow \cinfty(X)\otimes_{\OO_X}\Jet_X^{hol}(\Omega_X^1) $$
as before. Let $f\in\OO\bracket{ T^*X_{\dbar}[1]}$, then
we have
\begin{align*}
\{\widetilde{W},f\}=&\sum_i\dfrac{\partial\widetilde{W}}{\partial(\widetilde{dz^i})}\dfrac{\partial f}{
\partial(\widetilde{\partial_{z^i}})}
 \overset{(1)}{=}\sum_i\dfrac{\partial(d_R(\widetilde{W}))}{\partial(\widehat{dz^i})}\dfrac{\partial
f}{\partial(\widetilde{\partial_{z^i}})}\\
=&\sum_{i}\sum_{j,k}\dfrac{\partial(d_R(\widetilde{W}))}{\partial(\tilde{K}(dz^j)}\dfrac{\partial(\tilde{K}(dz^j)}{
\partial(\widehat{dz^i})}\dfrac{ \partial f}{\partial(K(\partial_{z^k}))}\dfrac{\partial(K(\partial_{z^k}))}{
\partial(\widetilde{\partial_{z^i}})}\\
=&\sum_{j,k}\dfrac{\partial(d_R(\widetilde{W}))}{\partial(\tilde{K}(dz^j)}\dfrac{ \partial
f}{\partial(K(\partial_{z^k}))}\cdot\langle\tilde{K}(dz^j),K(\partial_{z^k}))\rangle\\
                   \overset{(2)}{=}&\sum_i\dfrac{\partial(d_R(\widetilde{W}))}{\partial(\tilde{K}(dz^i)}\dfrac{\partial
f}{\partial(K(\partial_{z^i}))}
=dW\lrcorner f.
\end{align*}
Here identity $(1)$ follows from the commutative diagram (\ref{eqn:right-de rham-CE}) and identity $(2)$ follows from Equation
(\ref{eqn:definition-T}).
\end{proof}
\fi

\begin{defn}\label{LG-definition}
The space of fields of  topological Landau-Ginzburg $B$-model is
$$
   \E:= \A_{\Sigma_g}\otimes_{\C}\bracket{\g_{X}[1]\oplus \g_{X}^\vee},
$$
and the classical action functional $S^W$ is defined by
$$
   S^W=S+ I_{W},
$$
where $S$ is the classical action functional of the untwisted case, and $I_W$ is the local functional on $\A_{\Sigma_g}\otimes
\g_X[1]$ defined by
$$
   I_W\bracket{\alpha}:=\int_{\Sigma_g} \widetilde{W}(\alpha), \quad \alpha \in \A_{\Sigma_g}\otimes \g_X[1].
$$
Here $\widetilde W$ is extended linearly in $\A_{\Sigma_g}$ to $\A_{\Sigma_g}\otimes \g_X[1]$. The
LG-twisted interaction is
$$
  I^{W}_{cl}=I_{cl}+I_W.
$$
\end{defn}

\begin{rmk} The $\C^\times$-symmetry of the untwisted B-model is broken by the term $I_W$. In particular, the LG-twisted theory is
no longer a cotangent theory in the sense of \cite{Kevin-CS}.

\end{rmk}

\begin{lem}\label{lem:Classical-ME-LG} The classical interaction $I_{cl}^W$ of Landau-Ginzburg $B$-model satisfies the classical master equation $QI^W_{cl}+{1\over
2}\fbracket{I^W_{cl},
I^W_{cl}}+{F_{l_1}}=0$.
\end{lem}
\begin{proof}
\begin{align*}
 QI^W_{cl}+{1\over 2}\fbracket{I^W_{cl},I^W_{cl}}+F_{l_1}=&QI_{cl}+{1\over 2}\fbracket{I_{cl},I_{cl}}+F_{l_1}+QI_W+{1\over
2}\fbracket{I_W,I_W}+\fbracket{I_{cl},I_W}\\
                        \overset{(1)}{=}&QI_W+\fbracket{I_{cl},I_W},
\end{align*}
where $(1)$ follows from the classical master equation of $I_{cl}$ in the untwisted case and the vanishing of $\{I_W,I_W\}$ by type reason. It is not
difficult to see that for $\alpha\in\A_{\Sigma_g}$,
\begin{align*}
(QI_W+\{I_{cl},I_W\})(\alpha)=\int_{\Sigma_g}d_{D_X}(\widetilde{W})(\alpha)
=0,
\end{align*}
since $\widetilde{W}$ is flat.
\end{proof}

\subsection{Quantization} We assume that $X$ is Calabi-Yau with holomorphic volume form $\Omega_X$. Since $X$ is non-compact, the
choice of $\Omega_X$ will not be unique, and we will always fix one such choice.

Let $I_{qc}$ be the one-loop quantum correction of B-twisted $\sigma$-model associated to $(X, \Omega_X)$, with which  the
quantization of the untwisted theory
 is defined as
$$
  I[L]=W(\mathbb P_0^L, I_{cl}+\hbar I_{qc}):=\lim_{\epsilon\to 0}W(\mathbb P_\epsilon^L, I_{cl}+\hbar I_{qc}).
$$

\begin{defn} We define the Landau-Ginzburg twisting of $I[L]$ by
$$
   I^W[L]=W(\mathbb P_0^L, I_{cl}+I_{W}+\hbar I_{qc}):=\lim_{\epsilon\to 0}W(\mathbb P_\epsilon^L, I_{cl}+I_{W}+\hbar I_{qc}).
$$
\end{defn}
Since $I_W$ is only a functional on $\A_{\Sigma_g}^*\otimes\g_X[1]$, it is easy to see by the type reason that
$$
I^W[L]=I[L]+W_{tree}(\mathbb P_0^L,I_{cl},I_W).
$$

\begin{prop}\label{prop:QME-LG}
 $I^W[L]$  defines a quantization of B-twisted topological Landau-Ginzburg model $S^W$ in the sense of Definition
\ref{defn-quantization}.
\end{prop}

\begin{proof} Let $\delta_W[L]=W_{tree}(\mathbb P_0^L,I_{cl},I_W)$. By the type reason, $\Delta_L \delta_W[L]=\{\delta_W[L],\delta_W[L]\}_L=0$.
Since
$I[L]$ satisfies the QME, we have
\begin{align*}
  \bracket{Q_L+\hbar \Delta_L+{F_{l_1}\over \hbar}-R}e^{I^W[L]/\hbar}
  =\hbar^{-1}\bracket{Q_L \delta_W[L]+\{I[L], \delta_W[L]\}_L}e^{I^W[L]/\hbar}.
\end{align*}
Since $\delta_W[L]$ is given by sum over trees, it satisfies the 		 classical RG flow equation. The vanishing of $Q_L
\delta_W[L]+\{I[L], \delta_W[L]\}_L$ then follows from its vanishing in the classical limit
$$
  Q I_W+\{I_{cl}, I_W\}=0.
$$
\end{proof}

\subsection{Observable theory and correlation functions}
We would like to explore the correlation functions of local quantum observables of our Landau-Ginzburg theory. One
essential difference with the untwisted case is that it is no longer a cotangent theory due to the term
$I_W$, hence the interpretation of quantization as projective volume forms \cite{Kevin-CS} does not work directly in this case.
However, the BV integration interpretation in Corollary \ref{cor:correlation-function-BV-def} still makes sense in the LG-twisted
case, which we will use to define the correlation functions.

For simplicity, we will assume $X$ to be a Stein domain in $\C^n$, and that $Crit(W)$ is finite. We choose the
holomorphic volume form $\Omega_X=dz^1\wedge\cdots\wedge dz^n$, where $\{z^i\}$ are the linear coordinates on $\C^n$.

\begin{defn} The quantum observables on $\Sigma_g$ at scale $L$ is defined as the cochain complex
$$
\text{Obs}^q(\Sigma_g)[L]:=\bracket{\mathcal{O}(\mathcal{E})[[\hbar]],Q_L+\{I^W[L],-\}_L+\hbar\Delta_L}.
$$
Observables $\text{Obs}^q(U)$ with support in $U\subset \Sigma_g$ are defined in a similar fashion as in Section \ref{section:global-observable}.
\end{defn}

Correlation functions are defined for a proper subspace of quantum observables which are "integrable" in some good sense, since we
are working with non-compact space $X$. We consider the following simplest class:

\begin{defn}
We define the sub-complex $\text{Obs}_c^q(\Sigma_g)[L]\subset \text{Obs}^q(\Sigma_g)[L]$ by
$$
\text{Obs}^q_c(\Sigma_g)[L]:=\bracket{\mathcal{O}_c(\mathcal{E})[[\hbar]],Q_L+\{I^W[L],-\}_L+\hbar\Delta_L},
$$
where $\OO_c(\E):=\OO(\E)\otimes_{\A(X)}\A_c(X)$ and $\A_c(X)$ is the space of compactly supported differential forms on $X$. The
corresponding local observables supported in $U\subset \Sigma_g$ is denoted by $\text{Obs}^q_c(U)$.
\end{defn}

\begin{prop}\label{LG-local-observable}
 Let $U\subset \Sigma_g$ be a disk. Then
$$
  H^*(\text{Obs}^q(U))\iso H^*(\text{Obs}^q_c(U))\iso \Jac(W)[[\hbar]],
$$
where $\Jac(W)$ is the Jacobian ring of $W$.
\end{prop}
\begin{proof} The strategy is completely parallel to Corollatry \ref{cohomology-quantum-local}. We just need to observe that the
cohomology of classical observables in the twisted case is given by
$$
  H^*(\A^{0,*}(X, \wedge^*T_X), \dbar+dW\lrcorner ) \quad \text{or}\hspace{5mm} H^*(\A_c^{0,*}(X, \wedge^*T_X), \dbar+dW\lrcorner
),
$$
both of which are canonically isomorphic to $\Jac(W)$  (See \cite{LLS}).
\end{proof}

Now we define the correlation function of quantum observables. For Landau-Ginzburg model, notice that the ``top fermion''
$\pi_2^*(\Omega_X^{2g-2})$ can be defined in a similar way as in the untwisted case. Thus we can define the correlation function
of quantum observables via the  BV integration in the spirit of Corollary
\ref{cor:correlation-function-BV-def} (see also there for the notations):
\begin{defn}\label{def:correlation-function-LG}
Let $O$ be a quantum observable of Landau-Ginzburg $B$-model which is closed, then the  correlation function of $O$ is
defined as
$$
\abracket{O}^W_{\Sigma_g}:=\int_X\left((e^{I^W[\infty]/\hbar} O\big|_{\H})^{TF}\Big/\pi_2^*(\Omega_{X}^{2g-2})\right)\Big|_\mathcal{L}.
$$
\end{defn}

As a parallel to Theorem \ref{correlation-P1}, we have
\begin{prop}\label{thm-LG-correlation}
Let $\{U_i\}$ be disjoint  disks on $\Sigma_g$. Let $O_{f_i,U_i}\in H^*(Obs^q_c(U_i))$ be local observables associated to
$f_i\in\Jac(W)$ by Proposition \ref{LG-local-observable}. Then
$$
\langle O_{f_1,U_1}\star\cdots\star O_{f_m,U_m}\rangle_{\Sigma_g}^W=\sum_{p\in Crit(W)}\Res_p \bracket{{f_1\cdots f_m
\det(\pa_i\pa_j W)^{g}dz^1\wedge \cdots \wedge dz^n\over \prod_i \pa_i W }},
$$
where $\star$ is the local to global factorization product, and $\Res_p$ is the residue at the critical point $p$ \cite{GH}.
\end{prop}
\begin{proof}
 As in the non-twisted case, we can assume that $m=1$ and let $f=f_1\in \Jac(W)$. Let $O_f[L]$ denote the corresponding quantum
observable and $O_f=\lim\limits_{L\to 0} O_f[L]$. By the definition of correlation function, we have
\begin{align*}
 \langle O_f[\infty]\rangle_{\Sigma_g}^W
  =\int_X\left((e^{I^W[\infty]/\hbar} O_{f}[\infty]\big|_{\H})^{TF}/\pi_2^*(\Omega_{X}^{2g-2})\right)\Big|_\mathcal{L}.
\end{align*}
Since $X\subset \C^n$, we can choose the $L_\infty$ structure on $\mathfrak{g}_X$ such that $l_i=0$ for all $i\geq 2$. It is then
clear that the RG flow of the classical interaction of B-model $I_{cl}$ has only tree parts. Furthermore, when restricted to the
subspace of harmonic fields, there is
\begin{align*}
I[\infty]|_{\H}&=I_{cl}|_\H,\\
W_{tree}(\mathbb{P}_0^L,I_{cl},I_W)|_\H&=I_W|_\H.
\end{align*}
It is then not difficult to see that the only terms in $e^{I^W[\infty]/\hbar}$ that will contribute non-trivially to
$\left((e^{I^W[\infty]/\hbar} O_{f}[\infty]\big|_{\H})^{TF}/\pi_2^*(\Omega_{\C^n}^{2g-2})\right)\Big|_\mathcal{L}$ are the following:
\begin{equation}\label{pic:correlation-LG}
\figbox{0.26}{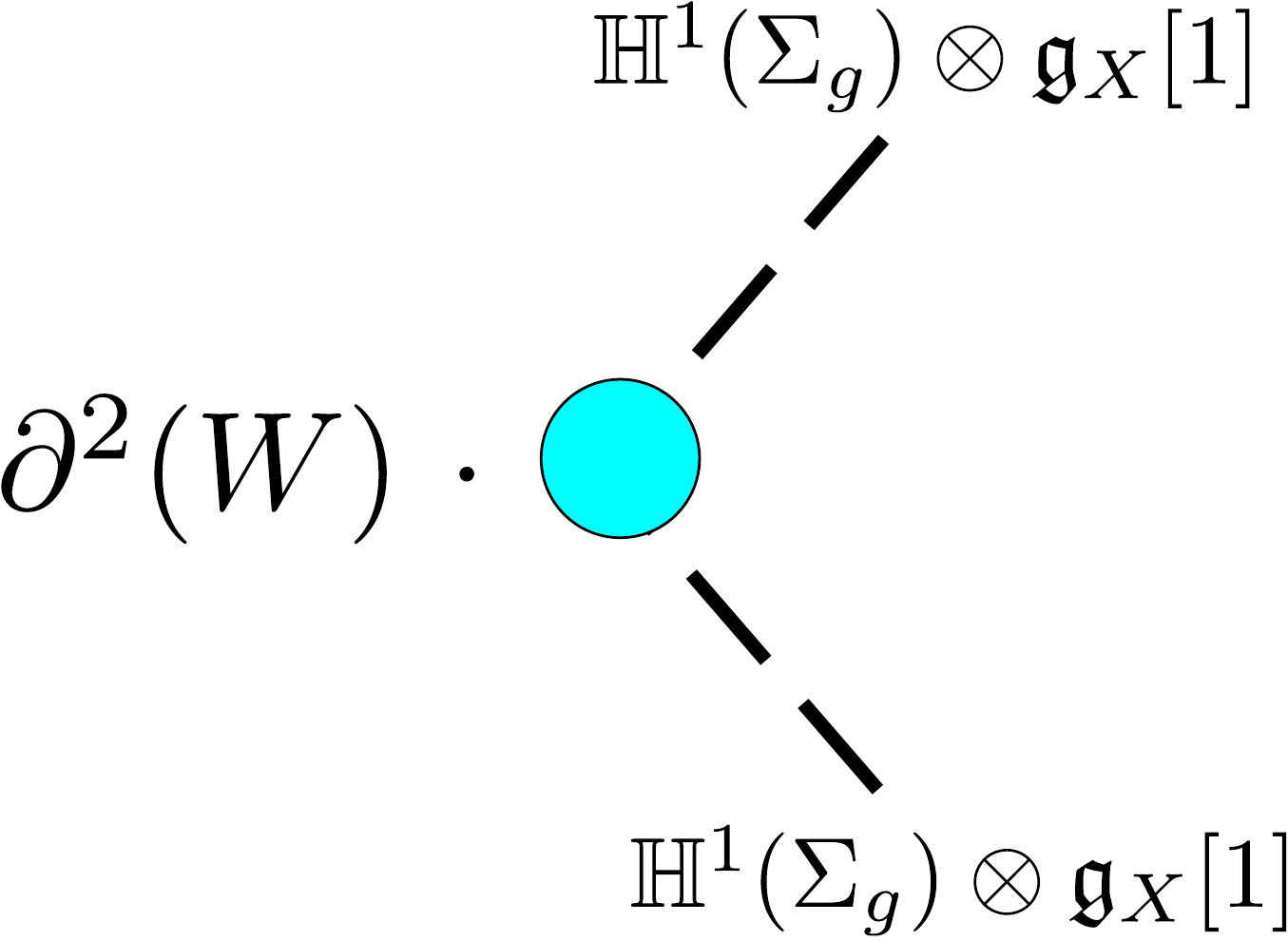}\hspace{15mm}\figbox{0.26}{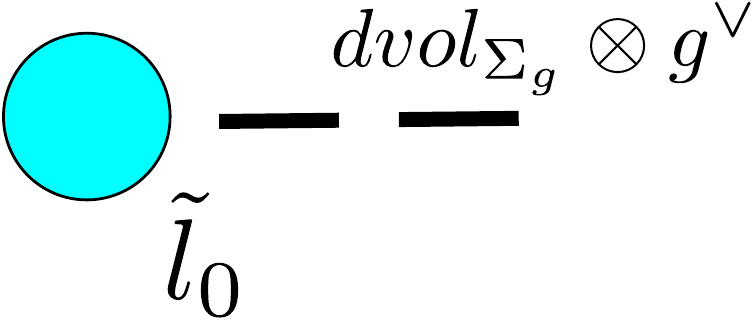}
\end{equation}
In the first picture, the two harmonic $1$-forms on $\Sigma_g$ attached to the tails must be dual to each
other. Since $\dim_\C(\H^1(\Sigma_g))=2g$, the total contribution of the first terms is
$$
\hbar^{-g\cdot n}\left(\det(\pa_i\pa_j W)\right)^g\otimes\pi_2^*(\Omega_X^{2g}).
$$
And the contribution of the second terms in (\ref{pic:correlation-LG}) together with the observable $O_{f}$
is, as in the computation of correlation functions in non-twisted B-model on $\mathbb{P}^1$, given by
$$
\hbar^{-n}((O_{f}\vdash \Omega_X)\wedge\Omega_X)\otimes \pi_2^*(\Omega_X^{-2}).
$$
All together, we have
\begin{align*}
\langle  O_{f}[\infty]\rangle_{\Sigma_g}^W=&\hbar^{-(g+1)n}\int_X\left(\det(\pa_i\pa_j W)\right)^g(O_{f}\vdash \Omega_X)\wedge\Omega_X\\
=&\hbar^{-(g+1)n}\cdot\sum_{p\in Crit(W)}\Res_p \bracket{{f \det(\pa_i\pa_j W)^{g}dz^1\wedge \cdots \wedge dz^n \over \prod_i \pa_i W}},
\end{align*}
where the last equality follows from \cite[Proposition 2.5]{LLS}.
\end{proof}

\begin{rmk}
This coincides with Vafa's residue formula for topological Landau-Ginzburg models \cite{vafa}.
\end{rmk}

\appendix
\section{Proof of Lemma \ref{lem:asymp-propagator}}\label{appendix:asymp-propagaotr}
The propagator $P_\epsilon^L$ on the upper half plane $\mathbb{H}$ with respect to the hyperbolic metric is a $1$-form
on $\mathbb{H}\times\mathbb{H}$, thus having a decomposition  under the isomorphism
$$
\A^1(\mathbb{H}\times\mathbb{H})\cong\bracket{\A^1(\mathbb{H})\otimes\A^0(\mathbb{H})}\oplus\bracket{\A^0(\mathbb{H}
)\otimes\A^1(\mathbb{H})},
$$
where $\otimes$ denotes  the completed tensor product. Let us call the projection into these two components by the
$(1,0)$ and $(0,1)$ part respectively. There is a similar decomposition of the heat kernel $k_t$ into its $(2,0),
(0,2)$ and $(1,1)$ parts. We will use $z_i=x_i+\sqrt{-1}y_i, i=1,2$ to denote the coordinates on the two copies of $\mathbb{H}$
respectively. The propagator
will be denoted by $P_\epsilon^L(z_1,z_2)$, where we have omitted its anti-holomorphic dependence for simplicity, and similarly
for the heat kernel
$k_t(z_1,z_2)$.

By the fact that $P_\epsilon^L(z_1,z_2)$ is a symmetric tensor in $\A^*(\mathbb{H})\otimes\A^*(\mathbb{H})$, we only need to
compute its $(1,0)$ part. For this, we  apply the gauge fixing operator $d^*$ to the $(2,0)$ part of the heat kernel which
is given explicitly by:
$$
 k_t^{scalar}(z_1,z_2)\dfrac{dx_1dy_1}{y_1^2}=\dfrac{\sqrt{2}}{(4\pi
t)^{\frac{3}{2}}}e^{-\frac{1}{4}t}\int^\infty_\rho\dfrac{se^{-\frac{s^2}{4t}}ds}{(\cosh
s-\cosh\rho)^{\frac{1}{2}}}\dfrac{dx_1dy_1}{y_1^2}.\\
$$
Here $k_t^{scalar}(z_1,z_2)$ is the heat kernel of the Laplacian on smooth functions, and $\rho(z_1,z_2)$ denotes the geodesic
distance between $z_1$ and $z_2$
given explicitly by
$$
\rho(z_1,z_2)=\text{arcosh}(1+\dfrac{(x_1-x_2)^2+(y_1-y_2)^2}{y_1y_2}).
$$
In particular, $k_t^{scalar}(z_1,z_2)=k_t^{scalar}(\rho(z_1,z_2))$ is a function of $\rho$. The $(1,0)$ part of  $P_\epsilon^L$ is
therefore given
by (where $d_{z_1}$ is the de Rham differential, $\star_1$ is the Hodge star on the first copy of $\mathbb{H}$)
\begin{align*}
&\int_\epsilon^L dt\left(\star_1 d_{z_1}\star_1 (k_t^{scalar}(z_1,z_2)\dfrac{dx_1dy_1}{y_1^2})\right)
=\int_\epsilon^L dt\left(\star_1 d_{z_1}(k_t^{scalar}(\rho(z_1,z_2)))\right)\\
=&\int_\epsilon^Ldt f(\rho,t) \bracket{\star_1 d_{z_1} \cosh(\rho(z_1,z_2))}\\
=& \int_\epsilon^Ldt f(\rho,t) \star_1
\left(\dfrac{2(x_1-x_2)}{y_1y_2}dx_1+\dfrac{(y_1-y_2)(y_1+y_2)}{y_1^2y_2}dy_1\right)\\
=&\int_\epsilon^Ldt f(\rho,t) \left(\dfrac{2(x_1-x_2)}{y_1y_2}dy_1-\dfrac{(y_1-y_2)(y_1+y_2)}{y_1^2y_2}dx_1\right)
\end{align*}
for some $f(\rho,t)$ clear from the context. By the symmetry property, the full propagator is given by
$$
P_{\epsilon}^L(z_1,z_2)=\int_{\epsilon}^Lf(\rho,t)dt\cdot\left(\dfrac{2(x_1-x_2)}{y_1y_2}(dy_1-dy_2)-\dfrac{(y_1-y_2)(y_1+y_2)}{
y_1y_2
}
\left(\dfrac{dx_1}{y_1}-\dfrac{dx_2}{y_2}\right)\right).
$$
The asymptotic property of $f(\rho,t)$ in equation (\ref{eqn:asymp-propagator})  follows from the general property of heat
kernels, or an explicit evaluation
of $f(\rho,t)$.

\section{Some Feynman graph computations}\label{appendix:Feynman-graph-computation}
\subsubsection*{Proof of Lemma \ref{lem:two-vertex-wheels}}
It is not difficult to see that the proof of the lemma can be reduced to  wheels with two vertices, and we will show that the
Feynman
weights (\ref{eqn:two-vertex-vanish-susy}) associated to the trivalent wheel vanishes. The proof for other wheels with two
vertices is similar.

Let $\alpha_1\otimes g_1$ and $\alpha_2\otimes g_2$ be the inputs on the tails, the Feynman weight
\begin{equation}\label{eqn:Feynman-weight-two-wheel-trivalent}
\figbox{0.2}{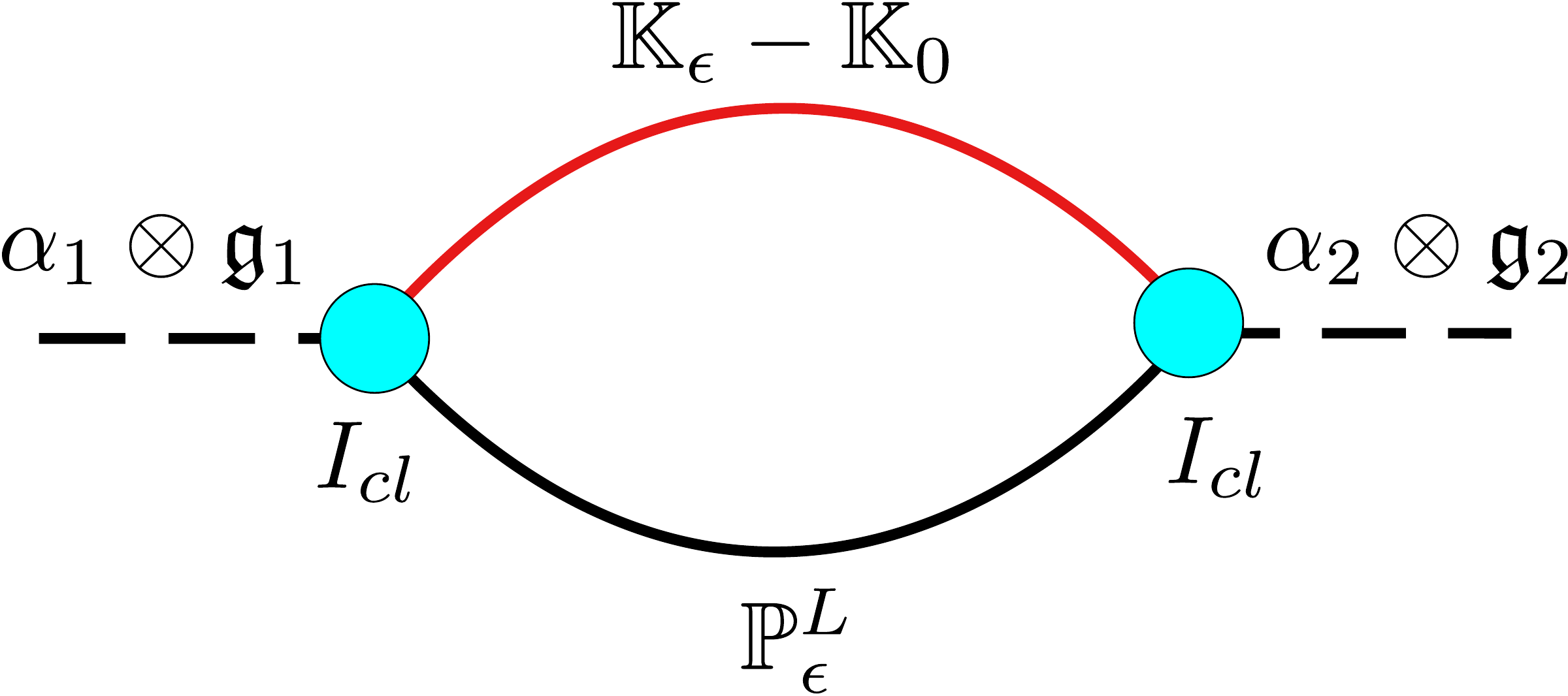}
\end{equation} is the evaluation of
\begin{align*}
&\mathbb{P}_\epsilon^L\otimes(\mathbb{K}_\epsilon-\mathbb{K}_0)\otimes(\alpha_1\otimes g_1)\otimes(\alpha_2\otimes g_2)\\
=&\left(P_\epsilon^L\otimes\sum_{i=1}^n(X_i\otimes X^i+X^i\otimes
X_i)\right)\otimes\left((K_\epsilon-K_0)\otimes\sum_{j=1}^n(X_j\otimes X^j+X^j\otimes
X_j)\right)\otimes(\alpha_1\otimes g_1)\otimes(\alpha_2\otimes g_2)
\end{align*}
under $I_{cl}\otimes I_{cl}$. Here $\{X_i\}$ denotes a basis of $\g_X$ over $\A_X$ (locally) and $\{X^i\}$ denotes the
corresponding dual basis of $\g_X^\vee$. More explicitly,
equation (\ref{eqn:Feynman-weight-two-wheel-trivalent}) is given by
\begin{equation*}
\begin{aligned}
&\Big(\int_{\Sigma_g\times\Sigma_g}P_\epsilon^L(z_1,z_2)(K_\epsilon(z_1,z_2)-K_0(z_1,z_2))\alpha_1\alpha_2\Big)\\
\cdot&\Big(\langle l_2(-),-\rangle\otimes\langle l_2(-),-\rangle\Big)\Bigg(\sum_{i,j=1}^n(-X_i\otimes g_1\otimes
X^j\otimes X_j\otimes g_2\otimes X^i+X_j\otimes g_1\otimes
X^i\otimes X_i\otimes g_2\otimes X^j)\Bigg)\\
=&\Big(\int_{\Sigma_g\times\Sigma_g}P_\epsilon^L(z_1,z_2)(K_\epsilon(z_1,z_2)-K_0(z_1,z_2))\alpha_1\alpha_2\Big)\\
&\cdot\Bigg(\sum_{i,j=1}^n\Big(-\langle l_2(X_i\otimes g_1),X^j\rangle\cdot \langle l_2(X_j\otimes g_2),
X^i\rangle+\langle l_2(X_j\otimes g_1),X^i\rangle\cdot \langle l_2(X_i\otimes g_2), X^j\rangle\Big)\Bigg)\\
=&0.
\end{aligned}
\end{equation*}

\subsubsection*{Proof of Lemma \ref{lem:more-vertex-wheels}}
We first prove the lemma for those cases where $n>3$. As in the proof of Lemma
\ref{lem:naive quantization}, we can replace $\Sigma_g$ by $\mathbb{H}$ with inputs compactly supported, and assume that $\gamma$
is a trivalent wheel. We still use the
notation $K_t$ for the heat kernel on $\mathbb{H}$ for convenience. Without loss of
generality, let us assume that the edge $e$ connects the vertices $v_1$ and $v_n$. Let $\alpha_i\otimes\mathfrak{g}_i$ be the
input on the vertex $v_i$. We will show that the
following two limits exist and are the same
\begin{equation}\label{eqn:one-loop-anomaly-K_ep-K_0}
\lim_{\epsilon\rightarrow
0}W_{\gamma,e}(P_\epsilon^L,K_\epsilon,I_{cl})=\lim_{\epsilon\rightarrow 0}W_{\gamma,e}(P_\epsilon^L,K_0,I_{cl}).
\end{equation}

The LHS of (\ref{eqn:one-loop-anomaly-K_ep-K_0}) is given  explicitly by
\begin{align*}
&W_{\gamma,e}(P_\epsilon^L,K_\epsilon,I_{cl})(\alpha_1,\cdots,\alpha_n)\\
=&\int_{z_1,\cdots,z_n\in\mathbb{H}}P_{\epsilon}^{L}
(z_1 , z_2)\cdots P_{\epsilon}^L(z_{n-1},z_n)
K_\epsilon(z_n,z_1)\alpha_1(z_1,\bar{z}_1)\cdots\alpha_n(z_n,\bar{z}_n) d^2z_1\cdots d^2z_n\\
=&\int_{t_1,\cdots,t_{n-1}=\epsilon}^Ldt_1\cdots dt_{n-1}\big(\int_{z_1,\cdots,z_n\in\mathbb{H}}d^*_{z_1}K_{t_1}(z_1,z_2)\cdots
d^*_{z_{n-1}}K_{t_{n-1}}(z_{n-1},z_n) K_{\epsilon}(z_n,z_1)\\
&\hspace{58mm}\alpha_1\cdots\alpha_nd^2z_1\cdots d^2z_n\big).
\end{align*}
We claim that the integral
\begin{equation}\label{one-loop anomaly}
\int_{z_1,\cdots,z_n\in\mathbb{H}}d^*K_{t_1}(z_1,z_2)\cdots
d^*K_{t_{n-1}}(z_{n-1},z_n) K_{\epsilon}(z_n,z_1)\alpha_1\cdots\alpha_nd^2z_1\cdots d^2z_n
\end{equation}
 is uniformly
bounded by a function of $t_1,\cdots,t_{n-1}$ which is integrable on $[0,L]^{n-1}$. Then (\ref{eqn:one-loop-anomaly-K_ep-K_0})
follows from dominated
convergence theorem.

Proof of the Claim: By the asymptotic expansion (\ref{eqn:asymp-heat-kernel}) (\ref{eqn:asymp-propagator}) of $K_t$ and
$P_\epsilon^L$ respectively, the leading term of the integral (\ref{one-loop anomaly})  is given by
\begin{equation}\label{eqn: one-loop-n>3}
\begin{aligned}
\dfrac{1}{(4\pi)^{n}}\int_{z_1,\cdots,z_{n}\in\mathbb{H}}&\prod_{k=1}^{n-1}b_0(\rho(z_k,z_{k+1}))\left(2(x_k-x_{k+1})(dy_k-dy_{k+1
}
)-\dfrac{y_k^2-y_{k+1}^2}{y_ky_{k+1}}(\dfrac{dx_k}{y_k}-\dfrac{dx_{k+1}}{y_{k+1}})\right)\\
&\left(\prod_{k=1}^{n-1}\dfrac{1}{t_k^2}e^{-\frac{\rho^2(z_{k},z_{k+1})}{4t_k}}\right)
\dfrac{1}{\epsilon}\cdot a_0(z_n,z_1)e^{-\frac{\rho^2(z_{n},z_{1})}{4\epsilon}}\ \alpha_1\ \cdots\ \alpha_n.
\end{aligned}
\end{equation}
We provide the estimates for the above leading term, while higher order terms furnish better convergence property.

We do the same change of coordinates as in the proof of Lemma/Definition \ref{lem:naive quantization}, a procedure after which the
integral
(\ref{eqn: one-loop-n>3}) becomes a sum of integrals of the following form:
\begin{equation}\label{eqn: one-loop-sum}
\begin{aligned}
\int_{\mathbb{H}}du_0dv_0\int_{\mathbb{R}^{2n-2}} du_1dv_1\cdots du_{n-1}dv_{n-1} \Phi\cdot\dfrac{1}{\epsilon
}\left(\prod_{k=1}^{n-1}\dfrac{u_k^{i_k}v_k^{j_k}}{t_k^2}\right)
\cdot\exp\left(-\sum_{i=1}^{n-1}\frac{u_i^2+v_i^2}{4t_i}-\frac
{\rho^2(z_n,z_1)}{
4\epsilon}\right),
\end{aligned}
\end{equation}
where
\begin{itemize}
 \item  For $1\leqslant k\leqslant n-1$, the functions $u_k^{i_k}v_k^{j_k}$ arises from $x_k-x_{k+1}$ and  $y_k-y_{k+1}$, hence
$i_k+j_k\geq1$.
\item $\Phi$ is a smooth function on $\mathbb{H}\times\mathbb{R}^{2n-2}$ with compact support.
\end{itemize}
Now we only need to show that for each fixed $(u_0,v_0)\in\mathbb{H}$, the following integral is bounded above in absolute value
by an
integrable function of $(t_1,\cdots,t_{n-1})$ on $[0,L]^{n-1}$ independent of $\epsilon$
\begin{equation}\label{eqn:Gaussian-estimate}
\int_{\mathbb{R}^{2n-2}} du_1dv_1\cdots du_{n-1}dv_{n-1}\dfrac{1}{\epsilon
}\left(\prod_{k=1}^{n-1}\dfrac{u_k^{i_k}v_k^{j_k}}{t_k^2}\right)
\cdot\exp\left(-\sum_{i=1}^{n-1}\frac{u_i^2+v_i^2}{4t_i}-\frac{\rho^2(z_n,z_1)}{
4\epsilon}\right).
\end{equation}
We show this for the leading term of its Wick expansion.
Notice that for each fixed $(u_0,v_0)\in\mathbb{H}$, the function
$$-\sum_{i=1}^{n-1}\frac{u_i^2+v_i^2}{4t_i}-\frac{\rho^2(z_n,z_1)}{
4\epsilon} $$ takes its maximal value $0$ at the critical point $(u_1,v_1,\cdots,u_{n-1},v_{n-1})=(0,\cdots,0)$. It is not
difficult to see that the Hessian at the critical point is the same as that of the function
$$-\sum_{i=1}^{n-1}\frac{u_i^2+v_i^2}{4t_i}-\frac{\left(\sum\limits_{i=1}^{n-1}u_i\right)^2+\left(\sum\limits_{i=1}^{n-1}
v_i\right)^2}{4\epsilon}. $$
Thus the leading term in the Wick expansion of (\ref{eqn:Gaussian-estimate}) is the same as that of the following
integral
\begin{equation}\label{eqn:Gaussian-integral}
\begin{aligned}
\int_{\mathbb{R}^{2n-2}} du_1dv_1\cdots
du_{n-1}dv_{n-1}&\dfrac{1}{\epsilon}\cdot\left(\prod_{k=1}^{n-1}\dfrac{u_k^{i_k}v_k^{j_k}}{
t_k^2}\right) \cdot\exp\left(-\sum_{i=1}^{n-1}\frac{u_i^2+v_i^2}{4t_i}-\frac{(\sum_{i=1}^{n-1}u_i)^2+(\sum_{i=1}^{n-1}v_i)^2}{
4\epsilon}\right),
\end{aligned}
\end{equation}
which can be evaluated via Gaussian type integral. We rearrange
the coordinates on $\mathbb{R}^{2n-2}$ as $$(u_1,\cdots,u_{n-1},v_1,\cdots,v_{n-1}),$$ and let $t=(t_1,\cdots,t_{n-1})$. The
matrix of
the quadratic form in the exponential is given by:
\begin{equation*}
M(t,\epsilon)=\frac{1}{4}\begin{pmatrix}
              A(t,\epsilon) & 0\\
               0            & A(t,\epsilon)
              \end{pmatrix},
\end{equation*}
in which \begin{equation*}
A(t,\epsilon)=\begin{pmatrix}
\frac{1}{t_1}+\frac{1}{\epsilon} & \frac{1}{\epsilon} & ... & \frac{1}{\epsilon} \\
\frac{1}{\epsilon} & \frac{1}{t_2}+\frac{1}{\epsilon} & ... & \frac{1}{\epsilon} \\
\vdots & \vdots &  \ddots &\vdots \\
\frac{1}{\epsilon} & \frac{1}{\epsilon} & ... & \frac{1}{t_{n-1}}+\frac{1}{\epsilon}
\end{pmatrix}.
\end{equation*}
For convenience, we will also use $M$ for the matrix $M(t,\epsilon)$. It is straightforward to check that
\begin{equation}\label{eqn:det}
\det(M)=\left(\frac{1}{4}\right)^{2(n-1)}\cdot\left(\dfrac{t_1+\cdots+t_{n-1}+\epsilon}{t_1\cdots t_{n-1}\epsilon}\right)^2.
\end{equation}
The standard trick of Feynman integral implies that (\ref{eqn:Gaussian-integral}) equals
\begin{equation}\label{eqn:Wick-expansion}
\begin{aligned}
&\dfrac{1 }{\sqrt{\det M}}\left(\prod_{i=1}^{n-1}\dfrac{1}{t_i^2}\right)
\cdot\dfrac { 1 } { \epsilon }\sum \left(M^{-1}_{\alpha_1,\beta_1}\cdots
M^{-1}_{\alpha_{N},\beta_{N}}\right)
=\dfrac{4^{n-1}}{t_1\cdots
t_{n-1}(t_1+\cdots+t_{n-1}+\epsilon)}\cdot\sum \left(M^{-1}_{\alpha_1,\beta_1}\cdots
M^{-1}_{\alpha_{N},\beta_{N}}\right),
\end{aligned}
\end{equation} where the sum is over
all pairings of $\prod_ { k=1 }^{n-1}(u_k^{i_k}v_k^{j_k})$, and $M^{-1}_{\alpha,\beta}$'s are entries of the inverse matrix of
$M$. $N$ is an integer
no less than $(n-1)/2$.

We claim that on each region of the form $$\{(t_1,t_2,\cdots,t_{n-1})\in[0,L]^{n-1}:0\leqslant
t_{\sigma(1)}\leqslant\cdots\leqslant
t_{\sigma(n-1)}\leqslant L\},$$ where $\sigma\in S_{n-1}$, (\ref{eqn:Wick-expansion}) is uniformly bounded above in absolute value
by an integrable function. This claim finishes the proof of Lemma \ref{lem:more-vertex-wheels}.

We will prove the claim for $\sigma=\text{id}\in S_{n-1}$, the proof for other $\sigma$'s is similar. The following Lemma provides
an
estimate of the entries of  $M^{-1}$.
\begin{lem}\label{lem:estimate-inverse-matrtix}
$|M_{i,j}^{-1}|\leqslant 4\cdot\min\{t_i,t_j\}$
\end{lem}
\begin{proof}
There are two possibilities: $i=j$ or $i\neq j$. By symmetry, we only need to consider $M_{1,1}^{-1}$ and $M^{-1}_{1,2}$. We have
\begin{align*}
M^{-1}_{1,1}=&\det\begin{pmatrix}
 \frac{1}{t_2}+\frac{1}{\epsilon} & \frac{1}{\epsilon} & ... & \frac{1}{\epsilon} \\
\frac{1}{\epsilon} & \frac{1}{t_3}+\frac{1}{\epsilon} & ... & \frac{1}{\epsilon} \\
\vdots & \vdots &  \ddots &\vdots \\
\frac{1}{\epsilon} & \frac{1}{\epsilon} & ... & \frac{1}{t_{n-1}}+\frac{1}{\epsilon}
\end{pmatrix}\cdot \left(\dfrac{t_1+\cdots+t_{n-1}+\epsilon}{t_1\cdots t_{n-1}\epsilon}\right)^{-1}\cdot 4\\
=&\dfrac{t_2+\cdots+t_{n-1}+\epsilon}{t_2\cdots
t_{n-1}\epsilon}\cdot\left(\dfrac{t_1+\cdots+t_{n-1}+\epsilon}{t_1\cdots t_{n-1}\epsilon}\right)^{-1}\cdot 4\\
=&t_1\cdot\dfrac{t_2+\cdots+t_{n-1}+\epsilon}{t_1+\cdots+t_{n-1}+\epsilon}\cdot 4\leqslant 4t_1.
\end{align*}
and
\begin{align*}
M^{-1}_{1,2}=&\det\begin{pmatrix}
\frac{1}{\epsilon} & \frac{1}{\epsilon} & ... & \frac{1}{\epsilon} \\
\frac{1}{\epsilon} & \frac{1}{t_3}+\frac{1}{\epsilon} & ... & \frac{1}{\epsilon} \\
\vdots & \vdots &  \ddots &\vdots \\
\frac{1}{\epsilon} & \frac{1}{\epsilon} & ... & \frac{1}{t_{n-1}}+\frac{1}{\epsilon}
\end{pmatrix}\cdot \left(\dfrac{t_1+\cdots+t_{n-1}+\epsilon}{t_1\cdots t_{n-1}\epsilon}\right)^{-1}\cdot 4\\
=&\dfrac{1}{t_3\cdots
t_{n-1}\epsilon}\cdot\left(\dfrac{t_1+\cdots+t_{n-1}+\epsilon}{t_1\cdots t_{n-1}\epsilon}\right)^{-1}\cdot 4\\
=&\dfrac{t_1t_2}{t_1+\cdots+t_{n-1}+\epsilon}\cdot 4
\leqslant 4\cdot\min\{t_1,t_2\}.
\end{align*}
\end{proof}
With Lemma \ref{lem:estimate-inverse-matrtix}, we can give an estimate of $\dfrac{M^{-1}_{\alpha_1,\beta_1}\cdots
M^{-1}_{\alpha_{N},\beta_{N}}}{t_1\cdots
t_{n-1}(t_1+\cdots+t_{n-1}+\epsilon)}$ in (\ref{eqn:Wick-expansion}):
since $1\in\{\alpha_1,\beta_1,\cdots,\alpha_{N},\beta_{N}\}$,
we can always find a subset $\{l_1,l_2,\cdots,l_{\tilde{N}}\}\subset\{2,3,\cdots,n-1\}$, $\tilde{N}\leq (n-1)/2$, such that
\begin{equation*}
\dfrac{\left|M^{-1}_{\alpha_1,\beta_1}\cdots
M^{-1}_{\alpha_{N},\beta_{N}}\right|}{t_1\cdots
t_{n-1}(t_1+\cdots+t_{n-1}+\epsilon)}\leqslant\dfrac{1}{t_{l_1}\cdots t_{l_{\tilde{N}}}}\dfrac{1}{t_1+\cdots+t_{n-1}}.
\end{equation*}
It is straightforward to check that the function
$$\dfrac{1}{t_{l_1}\cdots t_{l_{\tilde{N}}}}\dfrac{1}{t_1+\cdots+t_{n-1}}$$ is integrable on
$\{(t_1,\cdots,t_{n-1})\in[0,L]^{n-1}:0\leqslant t_1\leqslant\cdots\leqslant t_{n-1}\leqslant L\}$ if $n\geqslant 4$.

The only case left is when $n=3$. Notice that the following Feynman weight is non-trivial only if at least one
$\alpha_i$ is a $0$-form.
\begin{equation*}
 \figbox{0.2}{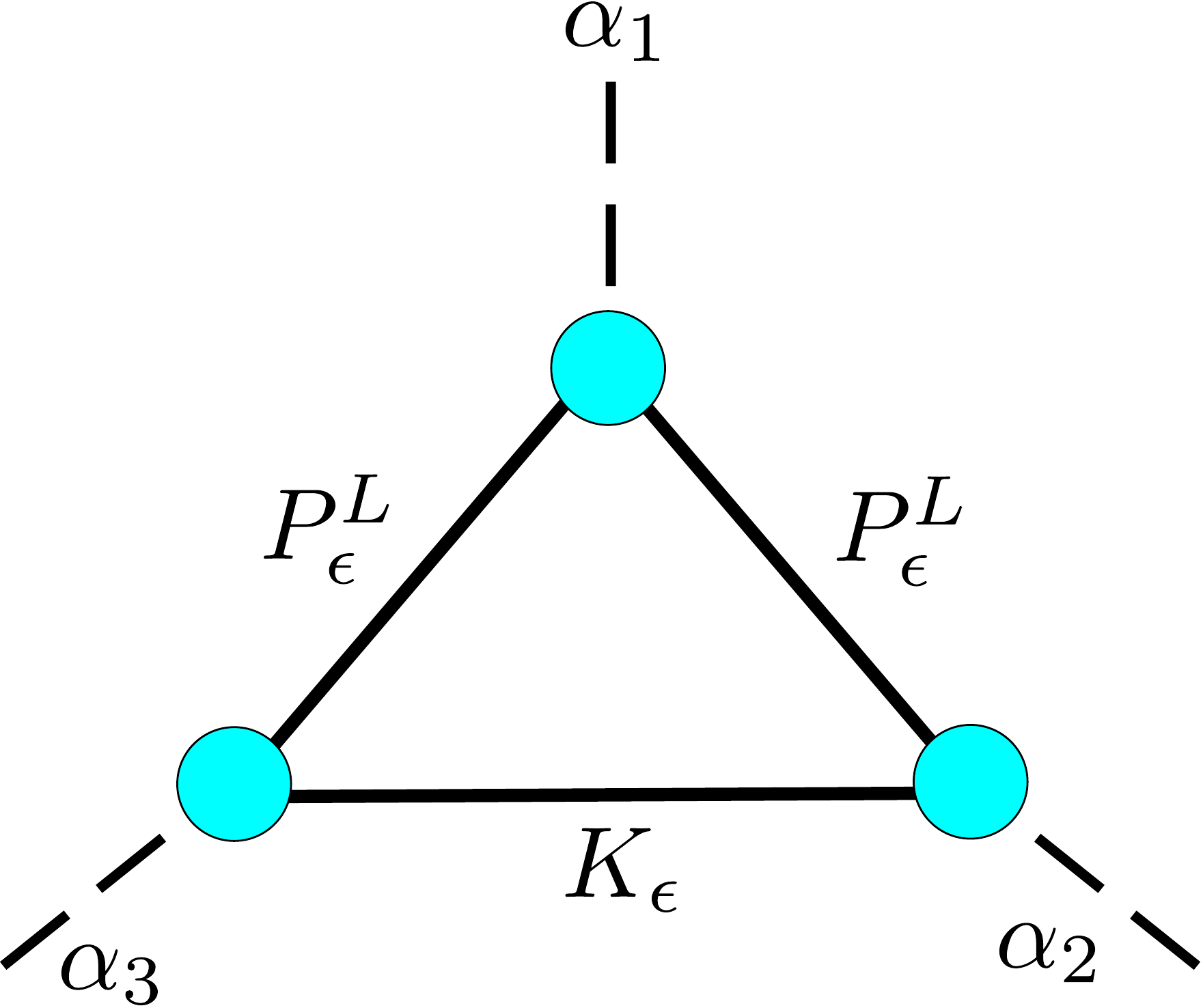}
\end{equation*}

Let $f$ be a compactly-supported function on $\H$. There are the following
two possible configurations of the input on the graph up to automorphisms:
\begin{equation*}
\underbrace{\figbox{0.2}{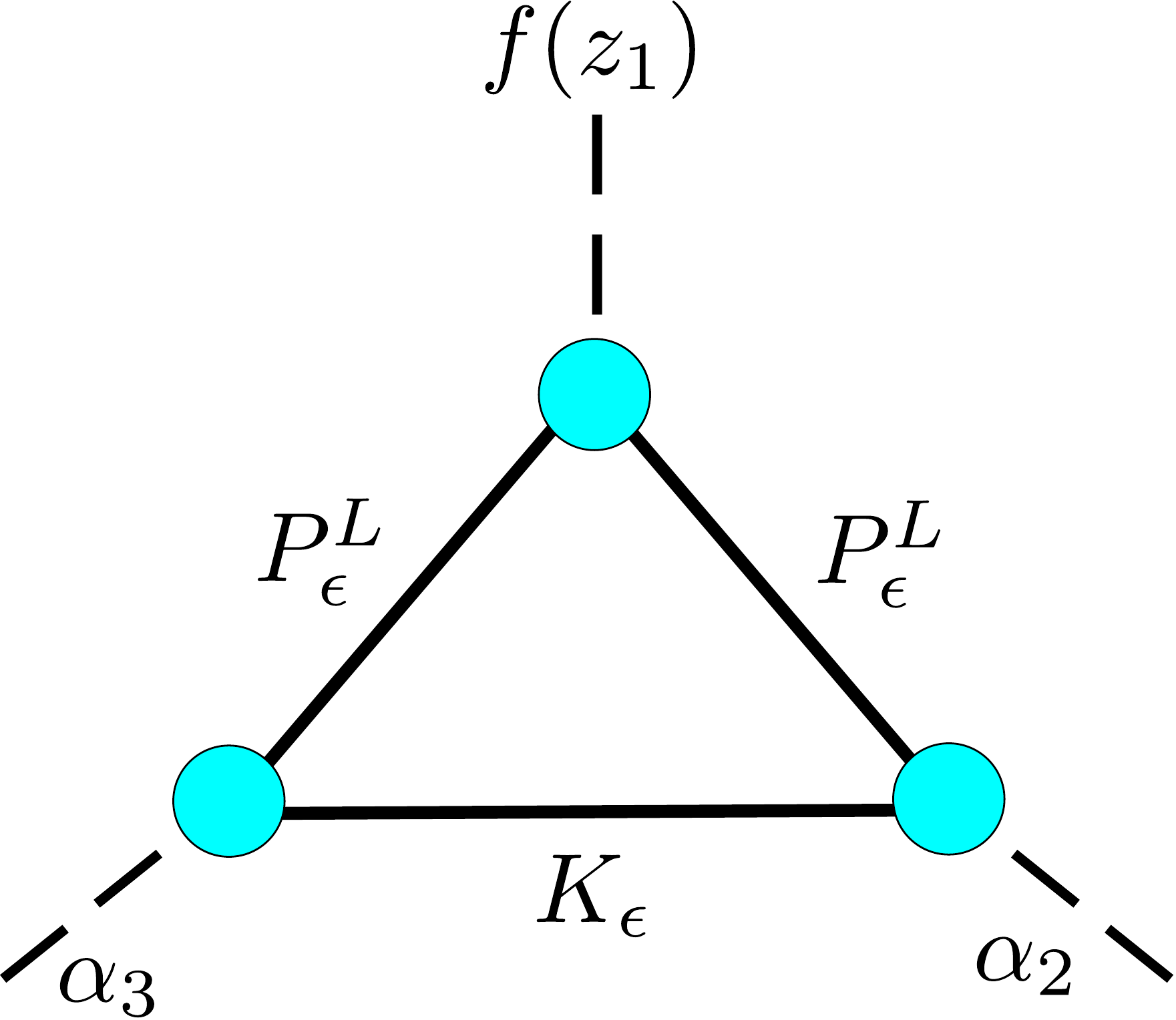}}_\text{(1)}\hspace{30mm}\underbrace{\figbox{0.2}{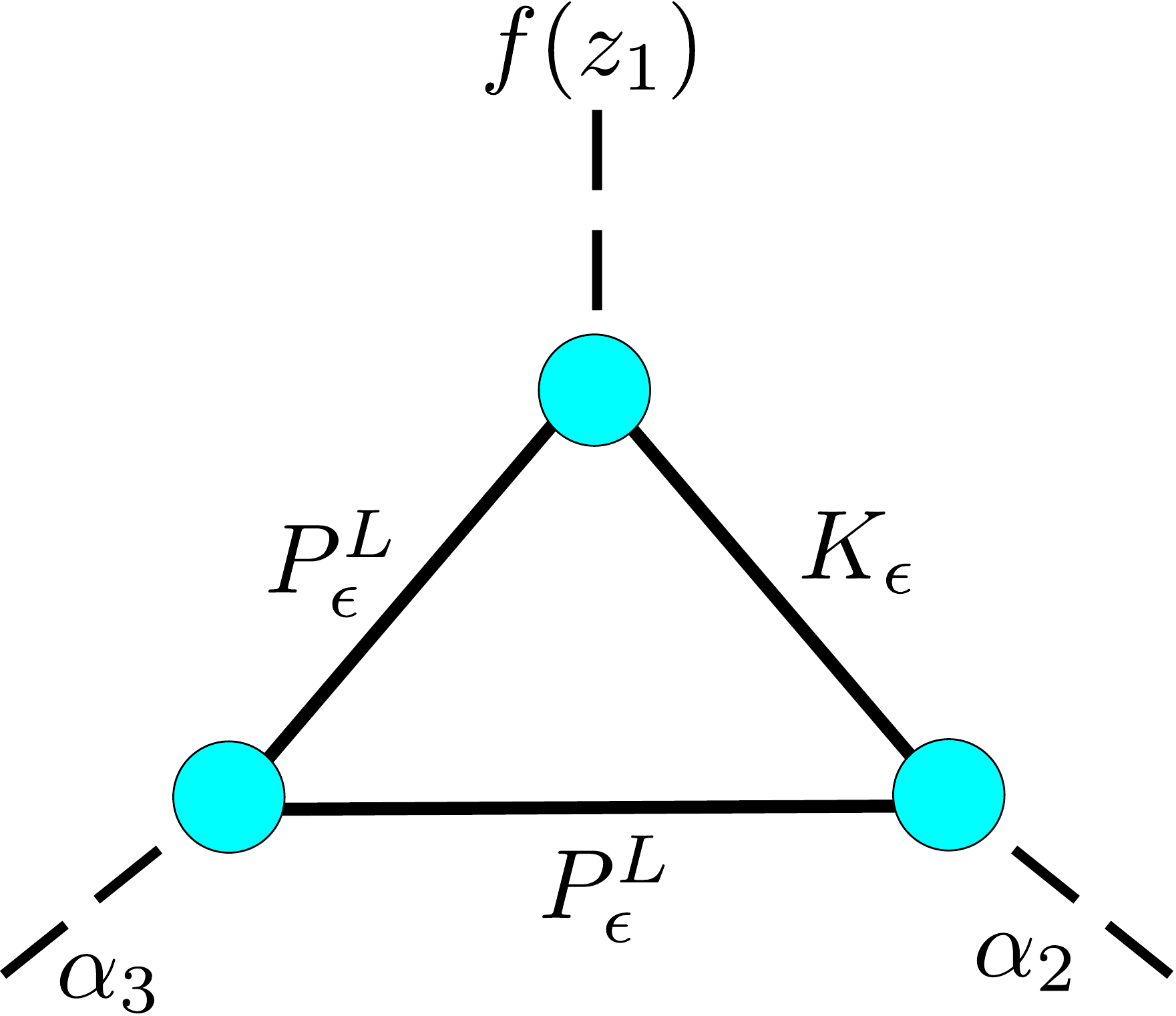}}_\text{(2)}
\end{equation*}

For configuration (1), we write the corresponding Feynman weight as the sum of
\begin{equation}\label{first}
\int_{z_1,z_2,z_3\in\mathbb{H}}P_\epsilon^L(z_1,z_2)\ K_\epsilon(z_2,z_3)\
P_{\epsilon}^L(z_3,z_1)\cdot f(z_2)\ \alpha_2\ \alpha_3
\end{equation}
and
\begin{equation}\label{second}
 \int_{z_1,z_2,z_3\in\mathbb{H}}P_\epsilon^L(z_1,z_2)\ K_\epsilon(z_2,z_3)\
P_{\epsilon}^L(z_3,z_1)\cdot(f(z_1)-f(z_2))\ \alpha_2\ \alpha_3.
\end{equation}

Here (\ref{first}) actually vanishes since
\begin{align*}
 \int_{z_1\in \mathbb{H}} P_{\epsilon}^L(z_3,z_1)P_\epsilon^L(z_1,z_2)=0,
\end{align*}
which amounts to $(d^*)^2=0$. The vanishing holds if we replace $K_\epsilon$ by
$K_0$. For (\ref{second}), we claim that
\begin{equation}\label{configuration1}
\begin{aligned}
&\lim_{\epsilon\rightarrow 0}\int_{z_1,z_2,z_3\in\mathbb{H}}P_\epsilon^L(z_1,z_2) K_\epsilon(z_2,z_3)
P_{\epsilon}^L(z_3,z_1)\cdot (f(z_1)-f(z_2)) \alpha_2 \alpha_3 \\
=&\lim_{\epsilon\rightarrow 0}\int_{z_1,z_2,z_3\in\mathbb{H}}P_\epsilon^L(z_1,z_2) K_0(z_2,z_3)
P_{\epsilon}^L(z_3,z_1)\cdot (f(z_1)-f(z_2)) \alpha_2 \alpha_3.
\end{aligned}
\end{equation}
To prove the claim, we apply the same argument for the case of $n\geqslant 4$. The
leading term of (\ref{second}) is similar to (\ref{eqn:Wick-expansion}), except that the function $f(z_1)-f(z_2)$ in
(\ref{second}) contributes one more $u_i$ or $v_i$ than in (\ref{eqn:Wick-expansion}) (so $N\geq 2$ when $n=3$, hence $\tilde
N=0$). Thus the leading
term is bounded above
by a constant times  $$\dfrac{1}{t_1+t_2},$$ which clearly has a finite
integral on $[0,L]\times[0,L]$. All together, we have
\begin{equation*}
 \lim_{\epsilon\rightarrow 0}\left(\figbox{0.2}{threevertices1}\right)=
\lim_{\epsilon\rightarrow 0}\left(\figbox{0.2}{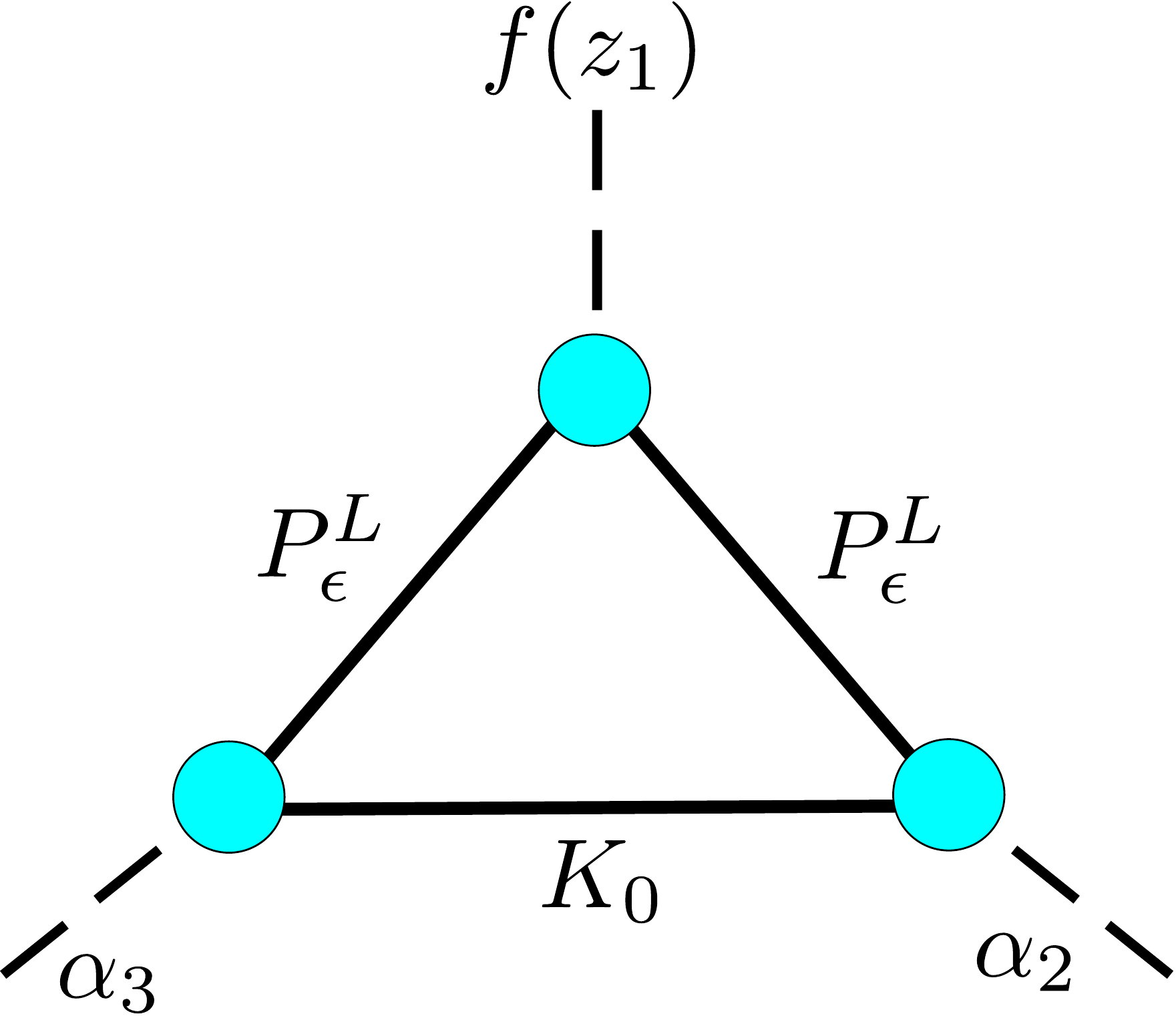}\right)
\end{equation*}

For  configuration (2), a similar argument as above shows
\begin{align*}
\lim_{\epsilon\rightarrow 0}\int_{z_1,z_2,z_3\in\mathbb{H}}K_\epsilon(z_1,z_2) P^L_\epsilon(z_2,z_3)
P_{\epsilon}^L(z_3,z_1)\cdot(f(z_1)-f(z_2))\alpha_2\alpha_3\\
=\lim_{\epsilon\rightarrow 0}\int_{z_1,z_2,z_3\in\mathbb{H}}K_0(z_1,z_2) P^L_\epsilon(z_2,z_3)
P_{\epsilon}^L(z_3,z_1)\cdot(f(z_1)-f(z_2))\alpha_2\alpha_3.
\end{align*}
On the other hand,
\begin{align*}
\int_{z_1,z_2,z_3\in\mathbb{H}}K_\epsilon(z_1,z_2) &P^L_\epsilon(z_2,z_3)
P_{\epsilon}^L(z_3,z_1)\cdot f(z_2)\alpha_2(z_2)\alpha_3(z_3)\\
&=\pm\int_{z_2,z_3\in\mathbb{H}}P_\epsilon^L(z_2,z_3)P_{2\epsilon}^{\epsilon+L}(z_3,z_2)f(z_2)\alpha_2(z_2)\alpha_3(z_3).
\end{align*}
Then similar to Lemma/Definition \ref{lem:naive quantization}, the above limit exists, and
\begin{align*}
&\lim_{\epsilon\rightarrow
0}\int_{z_2,z_3\in\mathbb{H}}P_\epsilon^L(z_2,z_3)P_{2\epsilon}^{\epsilon+L}(z_3,z_2)f(z_2)\alpha_2\alpha_3\\
=&\lim_{\epsilon\rightarrow 0}\int_{z_2,z_3\in\mathbb{H}}P_\epsilon^L(z_2,z_3)P_{\epsilon}^{L}(z_3,z_2)f(z_2)\alpha_2\alpha_3\\
=&\lim_{\epsilon\rightarrow
0}\int_{z_1,z_2,z_3\in\mathbb{H}}P_\epsilon^L(z_2,z_3)P_{\epsilon}^{L}(z_3,z_1)K_0(z_1,z_2)f(z_2)\alpha_2\alpha_3.
\end{align*}
Altogether we have
\begin{equation*}
 \lim_{\epsilon\rightarrow 0}\left(\figbox{0.2}{threevertices2}\right)=
\lim_{\epsilon\rightarrow 0}\left(\figbox{0.2}{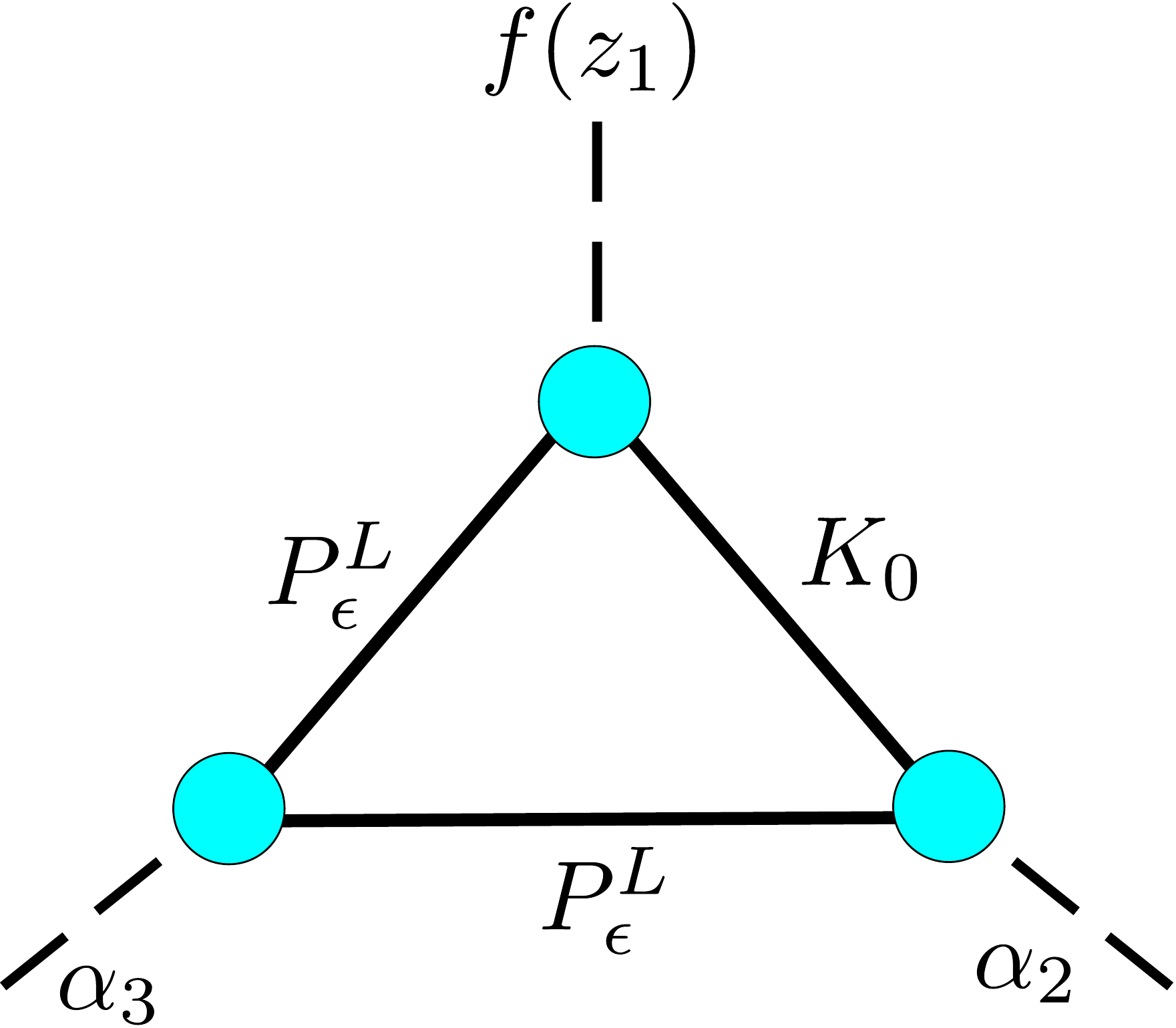}\right)
\end{equation*}

\section{One-loop anomaly}\label{appendix:one-loop-anomaly}
In this section, we give a general formula of the one-loop anomaly for perturbative QFT in Costello's formalism. Let $\mathcal{E}$
be
the space of fields of a perturbative QFT whose classical interaction is $I\in\mathcal{O}(\mathcal{E})$. Let $P_\epsilon^L$ denote the regularized propagator.

Let us first give an explicit description of the one-loop naive quantization $I_{naive}^{(1)}[L]$. Let $\Gamma^{Wheel}$ denote the
set of Feynman diagrams given by wheels (without trees attached), which are the essential part of 1-loop diagrams requiring
regularization by counter-terms.
We fix a renormalization scheme which allows us to decompose any graph integral uniquely into its `'smooth
part" and `'singular part" in the sense of \cite{Kevin-book}. Let $\gamma\in \Gamma^{Wheel}$, we will write
\begin{align}\label{smooth-sing-decomp}
  W_\gamma(P_\epsilon^L, I)=W_\gamma(P_\epsilon^L, I)^{sm}+W_{\gamma}(P_\epsilon^L, I)^{sing}
\end{align}
for the corresponding decomposition \cite[Theorem 9.5.1]{Kevin-book}.

\begin{lem} Let $\gamma\in \Gamma^{Wheel}$, then $W_{\gamma}(P_\epsilon^L, I)^{sing}$ is a local functional on $\E$ independent of $L$.
\end{lem}
\begin{proof} Since ${\pa \over \pa L}P_\epsilon^L$ is a smooth kernel which does not depend on $\epsilon$,
$
{\pa \over \pa L}W_\gamma(P_\epsilon^L, I)
$
behaves like a tree diagram. Therefore
$$
   \lim_{\epsilon\to 0}{\pa \over \pa L}W_\gamma(P_\epsilon^L, I) \quad \text{exists}.
$$
Hence $W_{\gamma}(P_\epsilon^L, I)^{sing}$ is independent of $L$. By \cite[Theorem 9.3.1]{Kevin-book}, $W_{\gamma}(P_\epsilon^L, I)^{sing}$ has a
small $L$ asymptotic expansion in terms of local functionals. Since it does not depend on $L$, it follows that $W_{\gamma}(P_\epsilon^L, I)^{sing}$
is local.
\end{proof}

By the algorithm in \cite{Kevin-book}, $W_{\gamma}(P_\epsilon^L, I)^{sing}$ is the counter-term associated to $\gamma$, and
$$
    \lim_{\epsilon\to 0}W_\gamma(P_\epsilon^L, I)^{sm} \quad \text{exists}.
$$
The following Proposition now follows easily from the Feynman diagram analysis and the regularization process described in \cite{Kevin-book}.
\begin{prop}\label{1-loop naive} The one-loop naive quantization is given by
$$
   I_{naive}^{(1)}[L]=\lim_{\epsilon\to 0} \sum_{\gamma_1\in trees,  v\in V(\gamma_1), \gamma_2\in \Gamma^{Wheel}}W_{\gamma_1,v}(P_\epsilon^L, I,
W_{\gamma_2}(P_\epsilon^L, I )^{sm}),
$$
where the summation is over all connected tree diagrams $\gamma_1$ with a specified  vertex $v$, and a wheel diagram $\gamma_2$.
$W_{\gamma_1,v}(P_\epsilon^L, I, W_{\gamma_2}(P_\epsilon^L, I )^{sm})$ is the Feynman graph integral on $\gamma_1$, where we put
$I$  on those vertices
not being $v$, put $W_{\gamma_2}(P_\epsilon^L, I )^{sm}$ on the vertex $v$, and put $P_\epsilon^L$ on all internal edges. %for the propagators.
\end{prop}
Pictorially,
\begin{equation}\label{eqn:effective-interaction}
I_{naive}^{(1)}[L]=\left(\figbox{0.18}{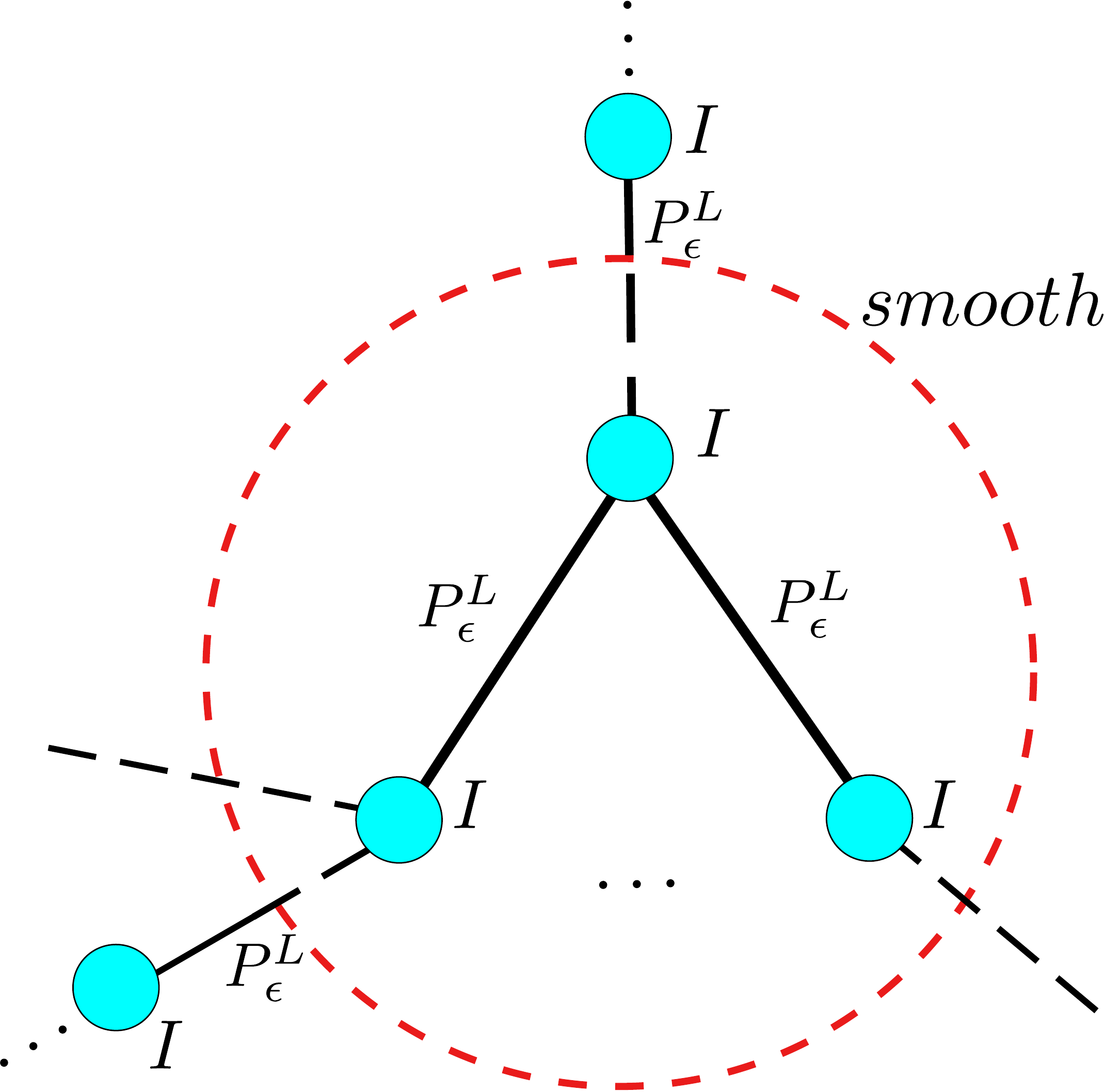}\right)
\end{equation}
\begin{rmk}
In the above picture, we are taking the sum of weights of all one-loop graphs.
\end{rmk}
Let
$$
  \bracket{Q+\hbar \Delta_L}e^{I_{naive}^{(0)}[L]/\hbar+I_{naive}^{(1)}[L]}=(O_1[L]+O(\hbar))e^{I_{naive}^{(0)}[L]/\hbar+I_{naive}^{(1)}[L]},
$$
where $O_1[L]$ is the leading term in the $\hbar$-expansion. By the construction in \cite[Chapter 5]{Kevin-book}, $O_1[L]$ is the
anomaly for solving quantum mater equation at one-loop. Moreover, $O_1[L]$ satisfies a version of classical renormalization group
flow, and
$$
 O_1:=\lim_{L\to 0}O_1[L]
$$
exists as a local functional. Our goal is to give a formula for computing $O_1$ in terms of graphs.

Let
$$
  I^{CT}(\epsilon)=\sum_{\gamma\in \Gamma^{Wheel}}W_{\gamma}(P_\epsilon^L, I)^{sing}
$$
denote the 1-loop counter-terms. Proposition \ref{1-loop naive} can be formally written as \cite{Kevin-book}
$$
e^{I_{naive}^{(0)}[L]/\hbar+I_{naive}^{(1)}[L]+O(\hbar)}=\lim_{\epsilon\to 0}e^{\hbar{\pa\over \pa P_\epsilon^L}}e^{I/\hbar-I^{CT}(\epsilon)}.
$$
Therefore
\begin{align*}
&(O_1[L]+O(\hbar))e^{I_{naive}^{(0)}[L]/\hbar+I_{naive}^{(1)}[L]+O(\hbar)}\\
=&\bracket{Q+\hbar \Delta_L}\lim_{\epsilon\to 0}e^{\hbar{\pa\over \pa P_\epsilon^L}}e^{I/\hbar-I^{CT}(\epsilon)}
=\lim_{\epsilon\to 0}e^{\hbar{\pa\over \pa P_\epsilon^L}}\bracket{Q+\hbar \Delta_\epsilon} e^{I/\hbar-I^{CT}(\epsilon)}\\
=&\lim_{\epsilon\to 0}e^{\hbar{\pa\over \pa P_\epsilon^L}}\bracket{\hbar^{-1}(\{I,I\}_\epsilon-\{I,I\}_0)+\Delta_\epsilon I-Q I^{CT}(\epsilon)-\{I,
I^{CT}(\epsilon)\}_\epsilon+O(\hbar)}e^{I/\hbar-I^{CT}(\epsilon)}.
\end{align*}
It follows that

\begin{align}\label{1-loop-anomaly-graph}
\begin{split}
O_1[L]=&\lim_{\epsilon\to 0} \Bigg( \sum_{\substack{\gamma: \text{1-loop connected},\\ v\in V(\gamma)}}W_{\gamma,v}(P_\epsilon^L, I,
\{I,I\}_\epsilon-\{I,I\}_0)\\
& \quad +\sum_{\gamma: tree, v\in V(\gamma)}W_{\gamma, v}(P_\epsilon^L,I, \Delta_\epsilon I-Q I^{CT}(\epsilon)-\{I,
I^{CT}(\epsilon)\}_\epsilon) \Bigg)
 \end{split}
\end{align}

\begin{lem} \label{lemma-Q-CT}
$$
QI^{CT}(\epsilon)=-\{I, I^{CT}(\epsilon)\}_0+\sum_{\substack{\gamma\in \Gamma^{Wheel}, \sharp E(\gamma)>1,\\ e\in
E(\gamma)}}W_{\gamma}(P_\epsilon^L, K_\epsilon-K_0,
I)^{sing}%-(\Delta_\epsilon I)^{sing}.
+(\Delta_\epsilon I)^{sing}.$$
\end{lem}
\begin{proof} It is easy to see that $Q$ preserves the decomposition \eqref{smooth-sing-decomp}, hence
$$
   Q I^{CT}(\epsilon)=Q \bracket{\sum_{\gamma\in \Gamma^{Wheel}}W_{\gamma}(P_\epsilon^L, I)^{sing}}=\sum_{\gamma\in
\Gamma^{Wheel}}(QW_{\gamma}(P_\epsilon^L, I))^{sing}.
$$
By the identity $(Q\otimes 1+1\otimes Q)P_\epsilon^L=K_\epsilon-K_L$  and the classical master equation,
\begin{align*}
  \bracket{\sum_{\gamma\in \Gamma^{Wheel}}QW_{\gamma}(P_\epsilon^L, I)}
  =&-\fbracket{I,\sum_{\gamma\in \Gamma^{Wheel}}W_{\gamma}(P_\epsilon^L, I)}_0
  -\sum_{\substack{\gamma\in \Gamma^{Wheel}, \sharp E(\gamma)>1,\\e\in E(\gamma)}}W_{\gamma}(P_\epsilon^L, K_0, I)\\
  &-\sum_{\gamma\in \Gamma^{Wheel}, e\in E(\gamma)} W_{\gamma}(P_\epsilon^L, K_L-K_\epsilon, I)\\
  =&-\fbracket{I,\sum_{\gamma\in \Gamma^{Wheel}}W_{\gamma}(P_\epsilon^L, I)}_0
+\sum_{\substack{\gamma\in \Gamma^{Wheel}, \sharp E(\gamma)>1,\\ e\in E(\gamma)}}W_{\gamma}(P_\epsilon^L, K_\epsilon-K_0, I)%-\Delta_\epsilon I\\
+\Delta_\epsilon I\\&-\sum_{\gamma\in \Gamma^{Wheel}, e\in E(\gamma)}W_{\gamma}(P_\epsilon^L, K_L, I).
\end{align*}
Since the last term is smooth as $\epsilon\to 0$, it follows that
$$
QI^{CT}(\epsilon)=-\{I, I^{CT}(\epsilon)\}_0+\sum_{\substack{\gamma\in \Gamma^{Wheel}, \sharp E(\gamma)>1,\\ e\in
E(\gamma)}}W_{\gamma}(P_\epsilon^L,
K_\epsilon-K_0, I)^{sing}%-(\Delta_\epsilon I)^{sing}.
+(\Delta_\epsilon I)^{sing}.
$$
\end{proof}

\begin{thm}\label{thm-1-loop-anomaly} The 1-loop anomaly $O_1$ is given by
$$
  O_1=\lim_{\epsilon\to 0}\sum_{\gamma\in \Gamma^{Wheel}, e\in E(\gamma)}W_{\gamma}(P_\epsilon^L, K_\epsilon-K_0, I)^{sm}+(\Delta_\epsilon I)^{sm}
$$
\end{thm}
\begin{proof}
The term
$$
\sum_{\substack{\gamma: \text{1-loop connected},\\ v\in V(\gamma)}}W_{\gamma,v}(P_\epsilon^L, I, \{I,I\}_\epsilon-\{I,I\}_0)
$$
in equation (\ref{1-loop-anomaly-graph}) can be expressed as the sum of the following two types of Feynman weights:
\begin{equation}\label{eqn:one-loop-anomaly-picture}
\figbox{0.18}{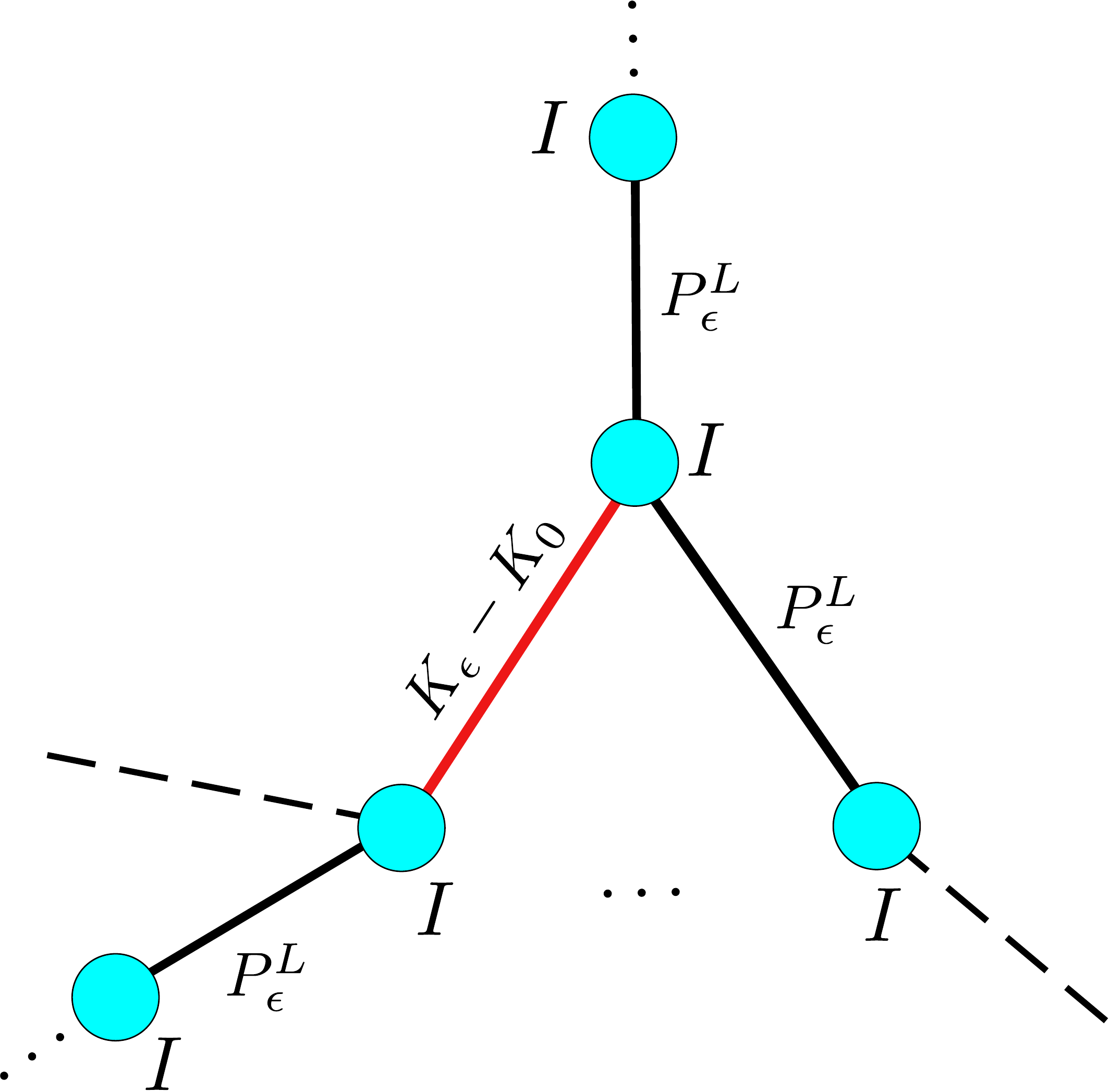}\hspace{9mm}\figbox{0.18}{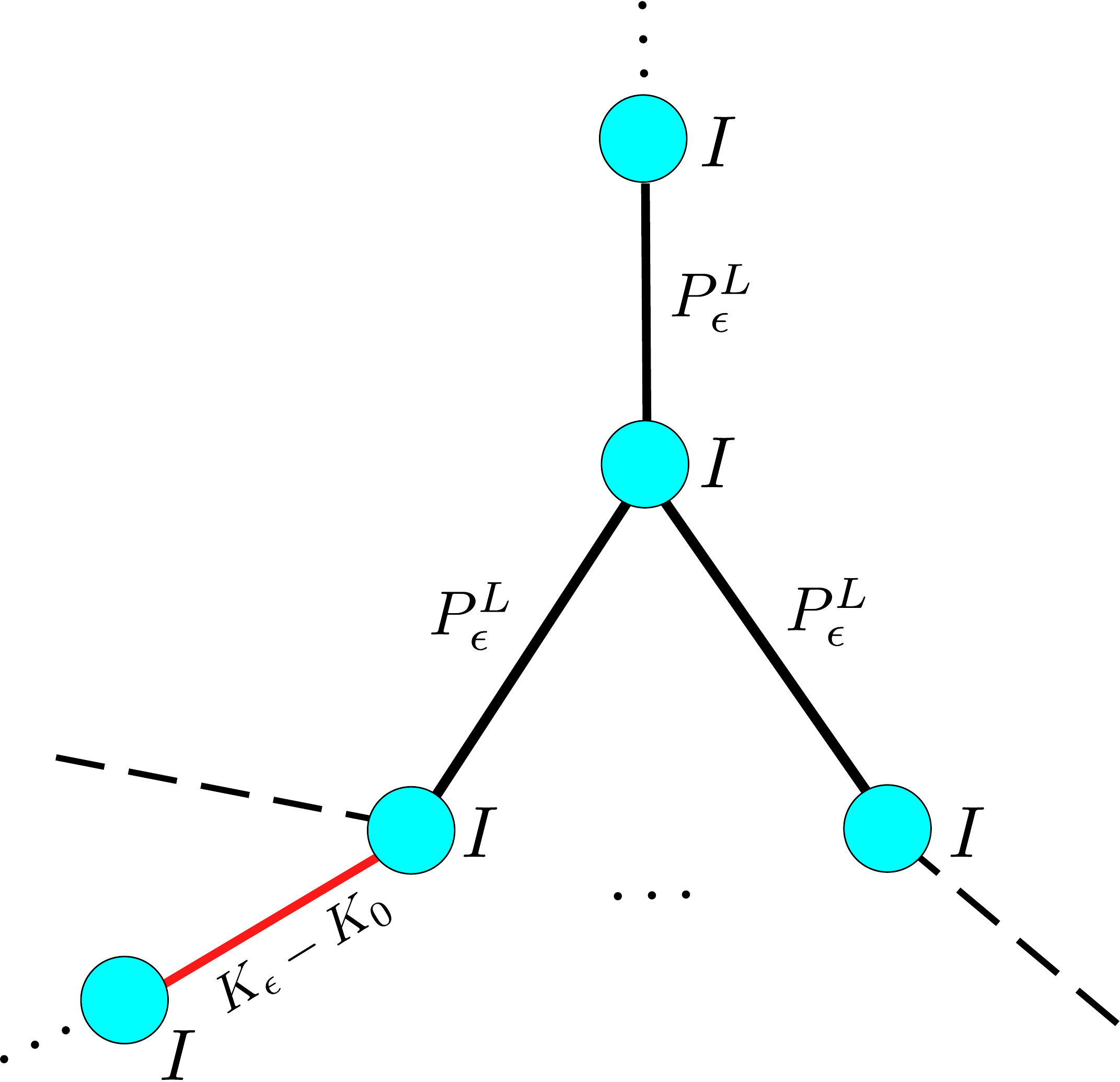}
\end{equation}
In the left picture $K_\epsilon-K_0$ is labeled on the wheel (the red edge) while in the right picture it is labeled on the external tree. It is not difficult to see that the
right picture, together with the term
$$
\sum_{\gamma: tree, v\in V(\gamma)}W_{\gamma, v}(P_\epsilon^L,I, \{I,
I^{CT}(\epsilon)\}_0-\{I,I^{CT}(\epsilon)\}_\epsilon)
$$
contributes
$$
\figbox{0.18}{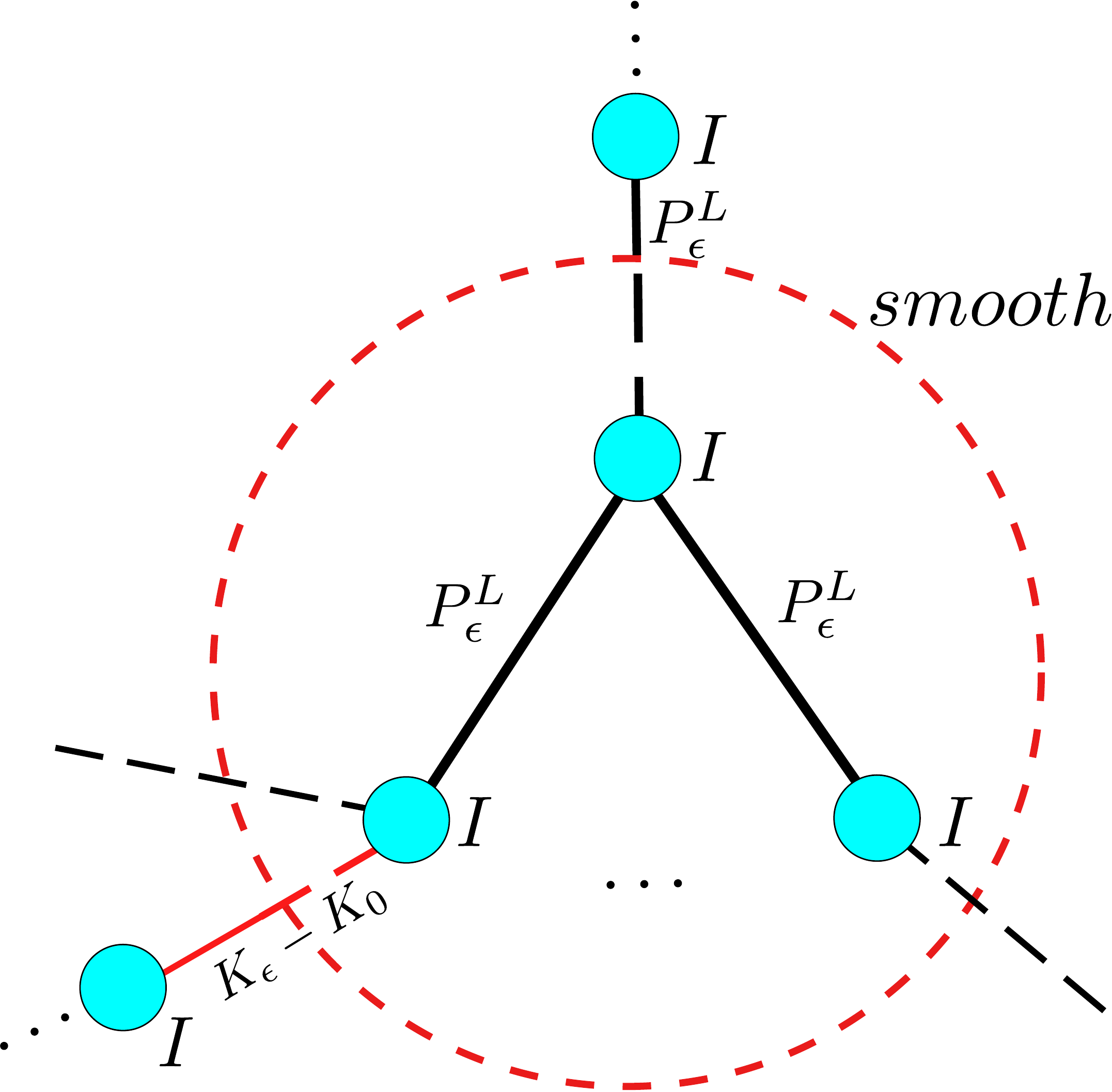}
$$ whose limit vanishies as $\epsilon\rightarrow 0$. The theorem then follows easily from equation (\ref{1-loop-anomaly-graph}), Lemma  \ref{lemma-Q-CT} and $O_1=\lim\limits_{L\to 0}O_1[L]$ (which kills all external trees).

\end{proof}
Theorem \ref{graph:one-loop anomaly} is now a graphic expression of Theorem \ref{thm-1-loop-anomaly}.

%{\color{red} Please modify the statement in Theorem \ref{graph:one-loop anomaly} since it is not stated correctly. Let's state is for $O_1$ instead
%of $O_1[L]$}.

\section{Chevalley-Eilenberg complex vs de Rham complex of jet bundles}\label{appendix:L_infty}
 The main objective of this section is to give an explicit description of the isomorphism in Proposition \ref{prop:iso-CE-jet}. We will also review
modules over $L_\infty$ algebras and the corresponding Chevalley-Eilenberg differential for the purpose of our discussion.

\subsubsection*{$L_\infty$ algebras and their modules}
Let us first recall the definition of $L_\infty$ algebras.
\begin{defn}\label{def:L-infty-algebra}
 Let $A$ be a commutative differential graded algebra and let $A^\sharp$ denote the underlying graded algebra.  A curved
$L_\infty$
algebra over $A$ consists of a locally free finitely generated graded $A^\sharp$-module $V$, together with a cohomological degree
$1$
and square zero derivation: $$d:\widehat{\text{Sym}}_{A^\sharp}(V^\vee[-1])\rightarrow\widehat{\text{Sym}}_{A^\sharp}(V^\vee[-1])
$$
such that the derivation $d$ makes $\widehat{\text{Sym}}_{A^\sharp}(V^\vee[-1])$ into a dga over the dga $A$.
Here $V^\vee$ denotes the $A^\sharp$-linear dual of $V$. We can decompose the derivation $d$ into components:
$$
d_n: V^\vee[-1]\rightarrow\text{Sym}^n_{A^\sharp}(V^\vee[-1]), n\geq 0.
$$
The structure maps of the curved $L_\infty$ algebra $V$ are defined by dualizing $d_n$ with a degree shift:
$$l_n:=d_n^*:\wedge^nV[n-2]\rightarrow V.$$
\end{defn}
The components $d_n$ of the derivation $d$ can be represented by the following ``corollas'', which should be read from bottom to
top:
the bottom line denotes the input of $d_n$ and the top lines denote the outputs.

%{\color{red}{should change $\g_X$ in the picture into V}. Similarly for the later pictures}

\begin{equation*}
 \figbox{0.26}{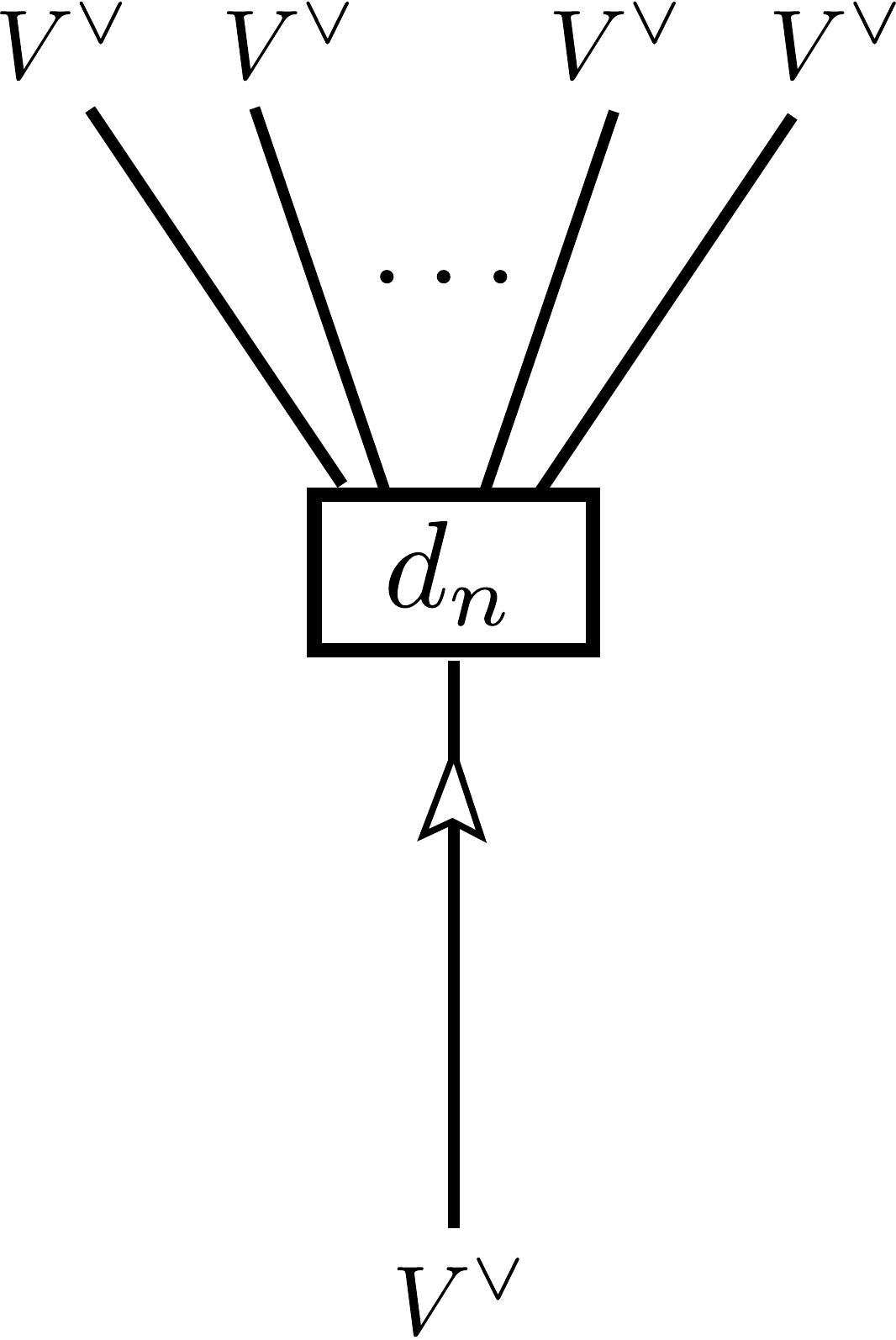}
\end{equation*}
Modules over $L_\infty$ algebras are defined in a similar fashion:
\begin{defn}
Let $A$ and $V$ be the same as in Definition \ref{def:L-infty-algebra}. An $A^\sharp$-module $M$  is
called a module over the $L_\infty$ algebra $V$ if there is a differential
$$
d_M:\widehat{\text{Sym}}_{A^\sharp}(V^\vee[-1])\otimes_{A^\sharp} M \rightarrow
\widehat{\text{Sym}}_{A^\sharp}(V^\vee[-1])\otimes_{A^\sharp} M
$$
making $\widehat{\text{Sym}}_{A^\sharp}(V^\vee[-1])\otimes_{A^\sharp} M$ a differential graded module over
$\widehat{\text{Sym}}_{A^\sharp}(V^\vee[-1])$.
\end{defn}
It is clear from the definition that the differential $d_M$ is determined by its components
$$
(d_M)_n:M\rightarrow\text{Sym}_{A^\sharp}^n(V^\vee[-1])\otimes_{A^\sharp} M,
$$
which we represent by the following picture:
\begin{equation*}
\figbox{0.26}{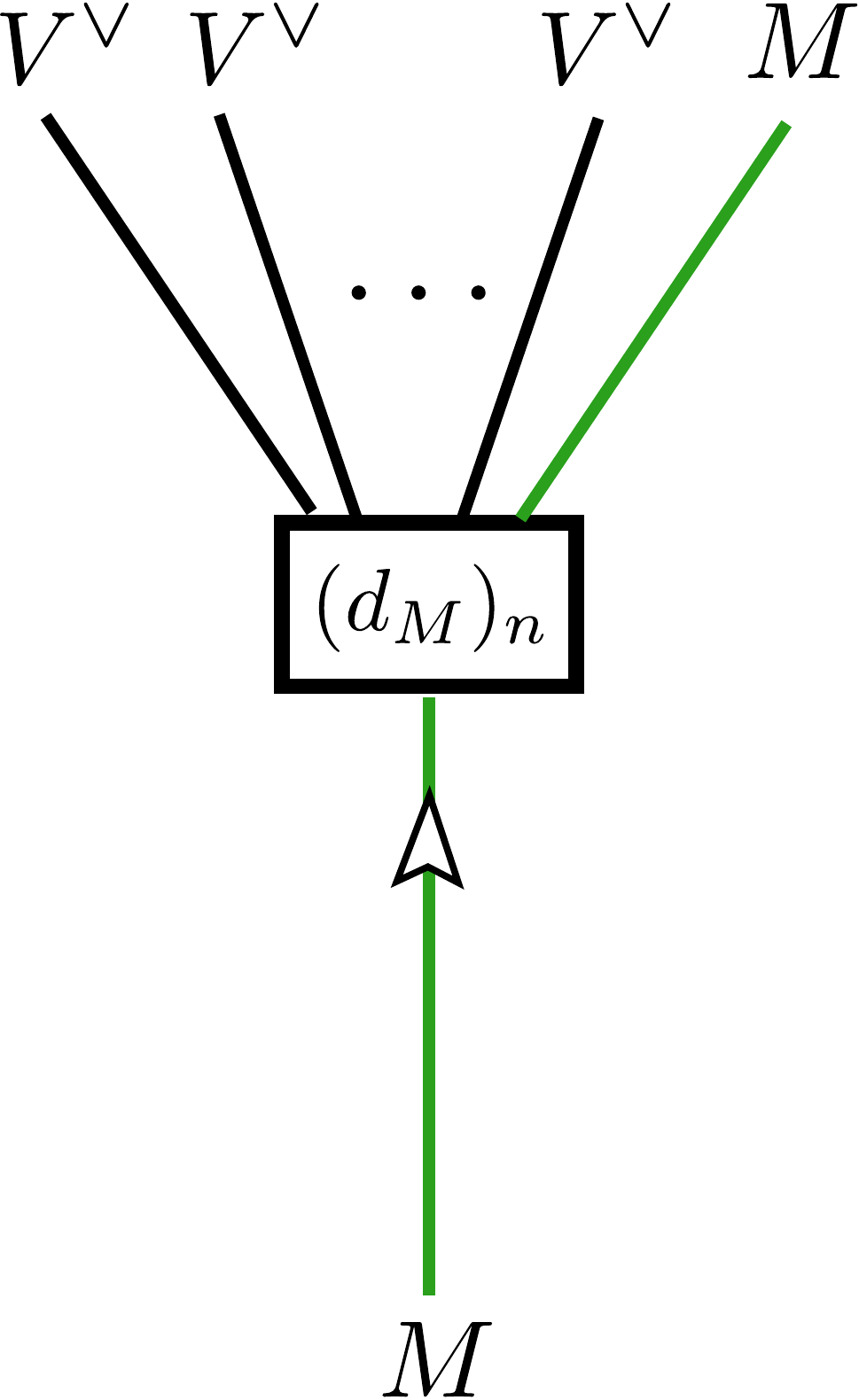}
\end{equation*}

\begin{eg}
$M=V^\vee$ has a naturally induced structure of $L_\infty$-module over $V$. We define the map $d_M$ by the following composition:
\begin{equation}\label{eqn:def-d-V}
M=V^\vee\rightarrow
V^\vee[-1]\overset{d_V}{\longrightarrow}\widehat{\text{Sym}}_{A^\sharp}(V^\vee[-1])\overset{d_{dR}}{\longrightarrow}\widehat{\text{Sym}}_{A^\sharp}
(V^\vee[-1])\otimes V^\vee=\widehat{\text{Sym}}_{A^\sharp}(V^\vee[-1])\otimes M.
\end{equation}
 This differential can be represented as follows:
\begin{equation*}
 \figbox{0.26}{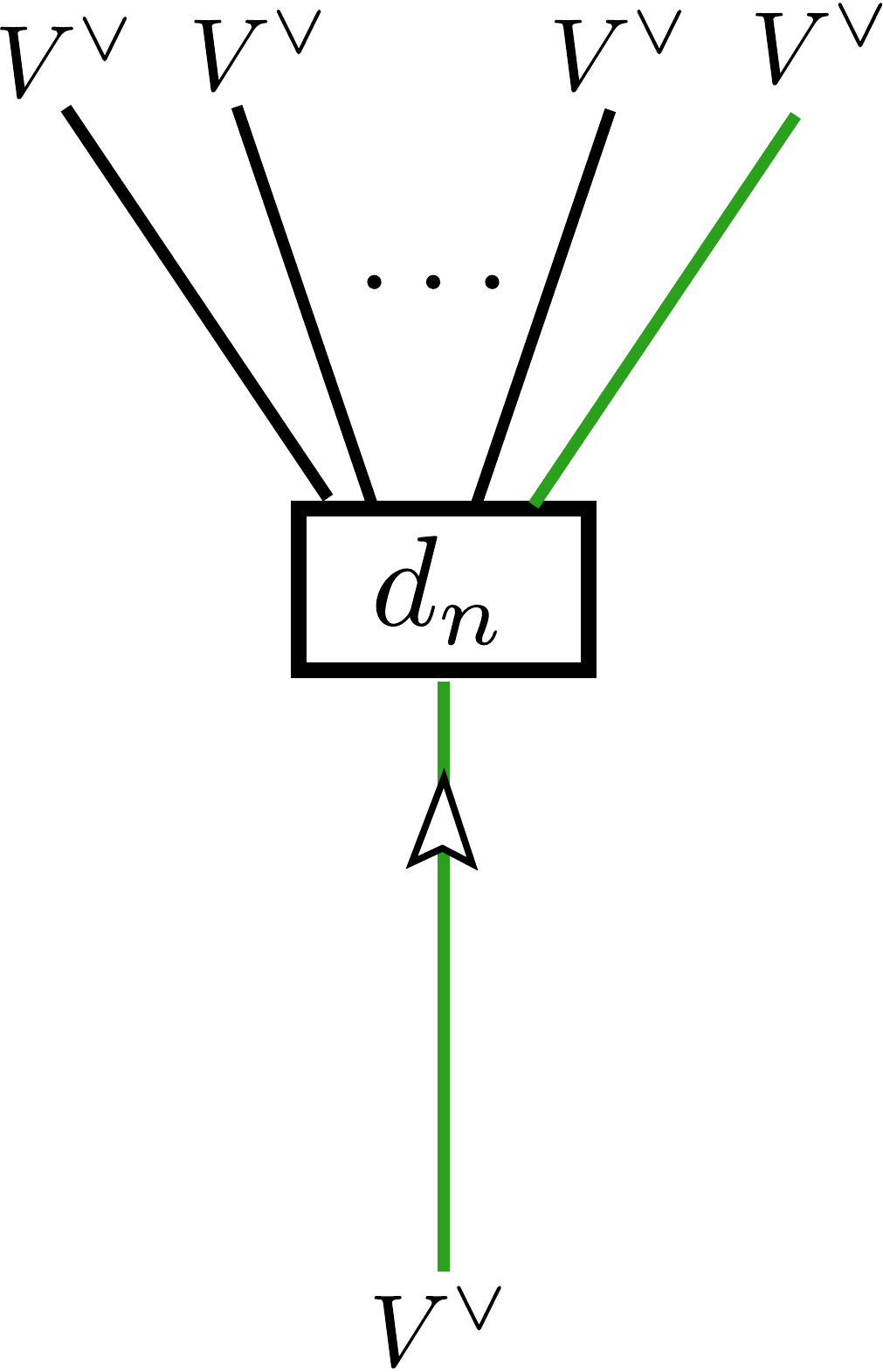}
\end{equation*}
where the green lines denote the module $M$, and the black lines denote the components in $\widehat{\text{Sym}}_{A^\sharp}(V^\vee[-1])$. Notice
that the only difference between $d_M$ and $d_V$ is that there are green lines in the graphical representation of $d_M$. Thus it is clear
that the identity $d_M^2=0$ follows from $d_V^2=0$, and that the effect of the operator $d_{dR}$ in equation (\ref{eqn:def-d-V}) is
exactly ``picking out the green line''.
\end{eg}
\begin{eg}
$N=V$. We define the differential $d_N$ by the following graphics:
\begin{equation}\label{graph:cotangent-module}
\figbox{0.26}{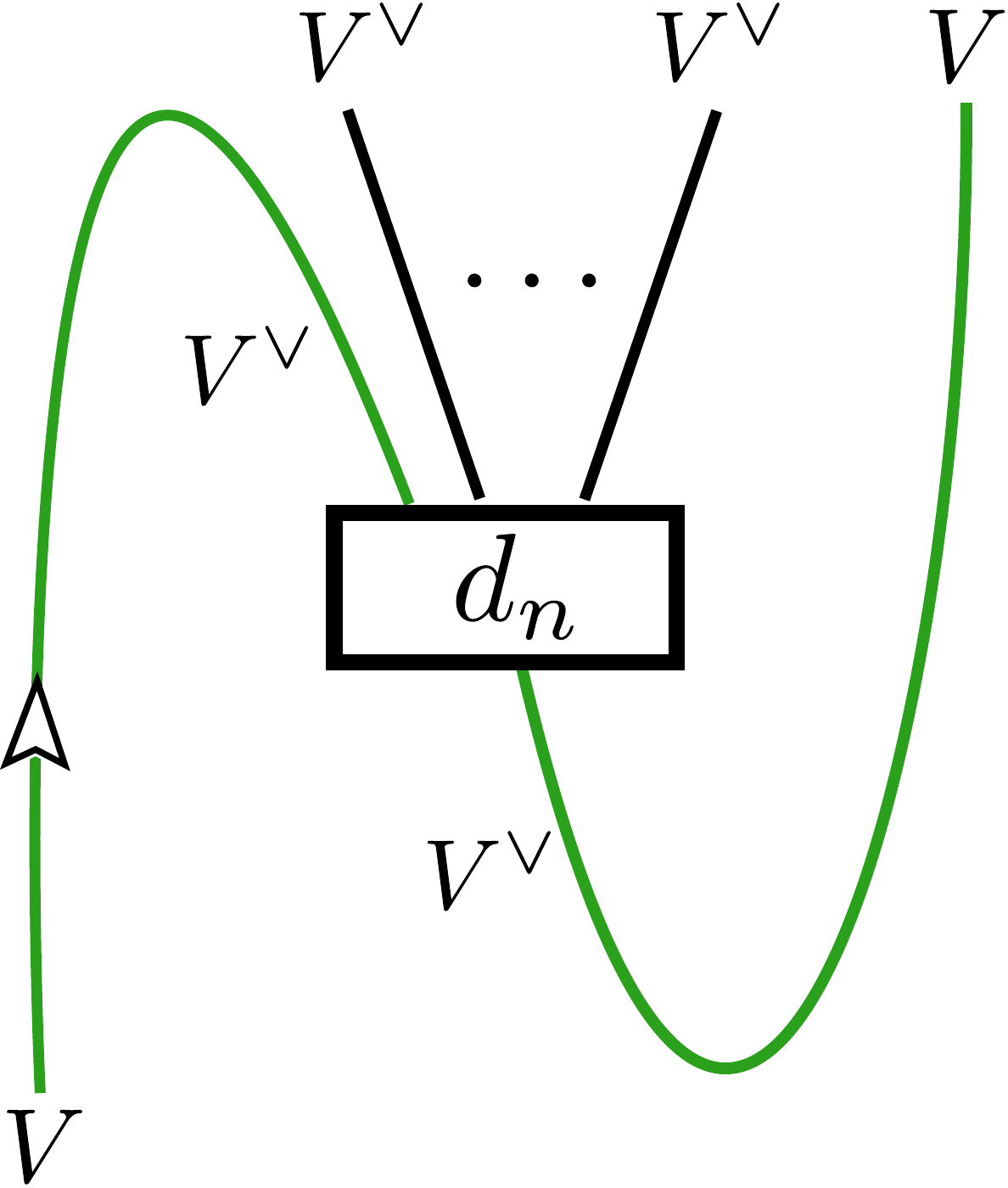}
\end{equation} where the downward ``elbow''
\begin{equation*}
\figbox{0.26}{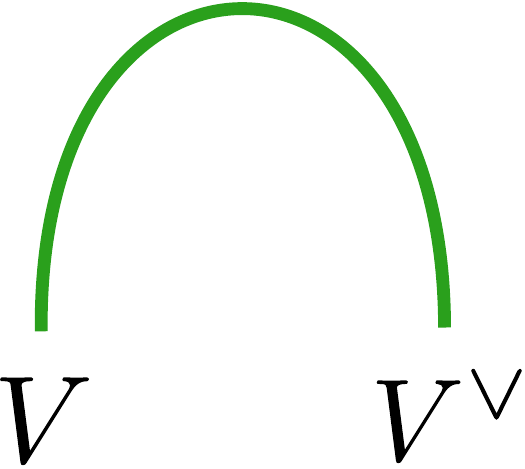}
\end{equation*} in (\ref{graph:cotangent-module}) denotes the evaluation map
$$\langle -,-\rangle: V\otimes V^\vee\rightarrow A^\sharp$$ and the reversed ``elbow'' denotes the
coevaluation map.

Again, $d_N^2=0$ follows
from the identity $d_V^2=0$.

\end{eg}

\subsubsection*{Proof of Proposition \ref{prop:iso-CE-jet}}
Let $X$ be a complex manifold, and let  $\mathfrak{g}_X$ be the curved $L_\infty$ algebra over $A=\mathcal{A}_X$ encoding the
complex
geometry of $X$. By the construction of $\mathfrak{g}_X$, there is an isomorphism of cochain complexes
$$\rho^*:\left(\mathcal{A}_X\otimes_{\mathcal{O}_X}\text{Jet}^{hol}
_X(\mathcal{O}_X),d_{D_X}\right)\overset{\sim}{\rightarrow}\left(C^*(\mathfrak{g}_X),d_{CE}\right).$$
We have the following proposition:
\begin{prop}\label{proposition:T-dual-to-jet}
The extension of the map (\ref{eqn:splitting-T}) over $\A_X$:
$$\mathfrak{g}_X^\vee[-1]\cong\mathcal{A}_X\otimes_{\mathcal{O}_X}\Omega_X^1\overset{\partial_{dR}\circ
T}{\longrightarrow}\mathcal{A}_X\otimes_{
\mathcal { O } _X}\text{Jet}_X^{hol}(\Omega_X^1)
$$
%where $T$ is the $\A_X^\sharp$-linear extension of (\ref{eqn:splitting-T}) and $\partial_{dR}$ is the operator on jet bundles
%induced by the
%holomorphic de Rham differential $\partial:\mathcal{O}_X\rightarrow\Omega_X^1$,
gives rise to an isomorphism of cochain complexes
$$\tilde{T}:\left(C^*(\mathfrak{g}_X)\otimes\mathfrak{g}_X^\vee[-1],d_{CE}\right)\overset{\sim}{\rightarrow}\left(\mathcal{A}
_X\otimes_{\mathcal{O}_X}
\text { Jet }_X^{hol} (\Omega_X^1),d_{D_X}\right).$$
\end{prop}
\begin{proof}
It is clear from the definition of $\tilde{T}$ that the following diagram commutes:
 \begin{equation}\label{eqn:right-de rham-CE}
  \xymatrix{
  C^*(\mathfrak{g}_X) \ar[d]^{(\rho^*)^{-1}} \ar[r]^{d_{dR}} & C^*(\mathfrak{g}_X)\otimes\mathfrak{g}_X^\vee[-1]
\ar[d]^{\tilde{T}} \\
\mathcal{A}_X\otimes_{\mathcal{O}_X}\text{Jet}_X^{hol}(\mathcal{O}_X) \ar[r]^{\partial_{dR}} &
\mathcal{A}_X\otimes_{\mathcal{O}_X}\text{Jet}_X^{hol}(\Omega^1_X)}
  \end{equation}
Here $d_{dR}$ is the de Rham differential of the algebra $C^*(\mathfrak{g}_X)$, and we have identified
$C^*(\mathfrak{g}_X)\otimes\mathfrak{g}_X^\vee[-1]$ with 1-forms. Consider the following diagram:
\begin{equation*}
 \xymatrix{
& \mathcal{A}_X\otimes_{\mathcal{O}_X}\text{Jet}_X^{hol}(\Omega_X^1) \ar[dd]^<<<<<<{d_{D_X}}|!{"2,1";"2,3"}\hole
 &&
\mathcal{A}_X\otimes_{\mathcal{O}_X}\text{Jet}_X^{hol}(\mathcal{O}_X) \ar[ll]_{\partial_{dR}}\ar[dd]^{d_{D_X}}\\
 C^*(\mathfrak{g}_X)\otimes\mathfrak{g}_X^\vee[-1] \ar[dd]^{d_{CE}}  \ar[ru]^{\tilde{T}} &&
C^*(\mathfrak{g}_X)\ar[dd]^>>>>>>>{d_{CE}} \ar[ll]^>>>>>>>>>>{d_{dR}}
\ar[ur]_{(\rho^*)^{-1}} \\
& \mathcal{A}_X\otimes_{\mathcal{O}_X}\text{Jet}_X^{hol}(\Omega_X^1)  &&
\mathcal{A}_X\otimes_{\mathcal{O}_X}\text{Jet}_X^{hol}(\mathcal{O}_X)\ar[ll]_<<<<<{\partial_{dR}}|!{"2,3";"4,3"}\hole\\
 C^*(\mathfrak{g}_X)\otimes\mathfrak{g}_X^\vee[-1]\ar[ur]^{\tilde{T}} &&
C^*(\mathfrak{g}_X)\ar[ll]^{d_{dR}}\ar[ur]_{(\rho^*)^{-1}}
  }
\end{equation*}
It is straightforward to check that all the squares commute except the left vertical one:
\begin{itemize}
\item The commutativity of the top and the bottom squares follows from (\ref{eqn:right-de rham-CE}).
\item The  front vertical square commutes  by the definition of the Chevalley-Eilenberg differential on
$C^*(\mathfrak{g}_X)\otimes\mathfrak{g}_X^\vee[-1]$,
\item The commutativity of the back vertical square follows from the fact that  $\partial_{dR}$ and $d_{D_X}$ commute with each
other,
\item The right vertical square commutes by the definition of $\g_X$.
\end{itemize}
Since $d_{dR}$ is surjective, a simple diagram chase shows the commutativity of the left vertical square,
which implies that the Chevalley-Eilenberg differential $d_{CE}$ on $C^*(\mathfrak{g}_X)\otimes\mathfrak{g}_X^\vee[-1]$ is
identified with $d_{D_X}$
on
$\mathcal{A}_X\otimes_{\mathcal{O}_X}\text{Jet}^{hol}_X(\Omega_X^1)$ under $\tilde{T}$.
\end{proof}
Now we prove the following proposition:
\begin{prop}\label{proposition:T_X-to-jet}
Let $K$ be the smooth homomorphism
\begin{align*}
 K: T_X&\rightarrow \cinfty(X)\otimes_{\OO_X}\text{Jet}_X^{hol}(T_X[-1])
\end{align*} such that
\begin{equation}\label{eqn:definition-T}
v(\alpha)=\langle K(v),T(\alpha)\rangle,
\end{equation}
 for all $\alpha\in\Omega_X^1,v\in T_X$. Then the extension of $K$ over
$C^*(\mathfrak{g}_X)\cong\mathcal{A}_X\otimes_{\mathcal{O}_X}\text{Jet}_X^{hol}(\mathcal{O}_X)$:
 $$
 \tilde{K}:C^*(\mathfrak{g}_X)\otimes\mathfrak{g}_X\rightarrow\mathcal{A}_X\otimes_{\mathcal{O}_X}\text{Jet}_X^{hol}(T_X[-1])
 $$
 is an isomorphism of cochain complexes. In particular, $d_{D_X}\circ\tilde{K}=\tilde{K}\circ d_{CE}$.
\end{prop}
\begin{proof}
It is obvious from equation (\ref{eqn:definition-T}) that $\tilde{K}$ is both injective and surjective. We now show that
$\tilde{K}$ commutes with
differentials. After translating equation (\ref{graph:cotangent-module}) into homomorphisms, it is clear that the
Chevalley-Eilenberg differential
$d_{CE}$ on $C^*(\mathfrak{g}_X)\otimes\mathfrak{g}_X$ is given by the following composition:
\begin{equation*}
 \xymatrixcolsep{3pc}\xymatrix{
\mathfrak{g}_X \ar[r]^-{id\otimes coev}& \mathfrak{g}_X\otimes\mathfrak{g}_X^\vee\otimes\mathfrak{g}_X \ar[r]^-{id\otimes
d_{CE}\otimes id} &\mathfrak{g}_X\otimes\mathfrak{g}_X^\vee \otimes C^*(\mathfrak{g}_X)\otimes \mathfrak{g}_X \ar[r]^-{ev\otimes
id\otimes id}&
C^*(\mathfrak{g}_X)\otimes \mathfrak{g}_X.
}
\end{equation*}
We pick local holomorphic coordinates $\{z^i\}$ on $X$. The image of $\{\widetilde{\partial_{z^i}}\}$ under $d_{CE}$ is given by
$$\widetilde{\partial_{z^i}}\mapsto\widetilde{\partial_{z^i}}\otimes \widetilde{dz^j}\otimes\widetilde{\partial_{z^j}}\mapsto
\widetilde{\partial_{z^i}}\otimes
(\tilde{T}^{-1}\circ d_{D_X}\circ\tilde{T})(\widetilde{dz^j})\otimes
\widetilde{\partial_{z^j}}\mapsto\langle\widetilde{\partial_{z^i}},(\tilde{T}^{-1}\circ
d_{D_X}\circ\tilde{T})(\widetilde{dz^j})\rangle\otimes\widetilde{\partial_{z^j}}.$$
We have the following identities:
\begin{align*}
 &\ \langle\widetilde{\partial_{z^i}},(\tilde{T}^{-1}\circ
d_{D_X}\circ\tilde{T})(\widetilde{dz^j})\rangle\otimes\widetilde{\partial_{z^j}}\\
 \overset{(1)}{=}&\ \langle
\tilde{K}(\widetilde{\partial_{z^i}}),(d_{D_X}\circ\tilde{T})(\widetilde{dz^j})\rangle\otimes\widetilde{\partial_{z^j}}\\
 \overset{(2)}{=}&\ \langle
(d_{D_X}\circ\tilde{K})(\widetilde{\partial_{z^i}}),\tilde{T}(\widetilde{dz^j})\rangle\otimes\widetilde{\partial_{z^j}}\\
 \overset{(3)}{=}&\ \langle (\tilde{K}^{-1}\circ
d_{D_X}\circ\tilde{K})(\widetilde{\partial_{z^i}}),\widetilde{dz^j}\rangle\otimes\widetilde{\partial_{z^j}}\\
 =&\ (\tilde{K}^{-1}\circ d_{D_X}\circ\tilde{K})(\widetilde{\partial_{z^i}}),% Why the negative sign goes away?
\end{align*} where the identities $(1)$ and $(3)$ follow from equation (\ref{eqn:definition-T}) and identity $(2)$ follows from
the fact that
$d_{D_X}$ is a derivation
with respect to the pairing $\langle -,-\rangle$.
\end{proof}
It is clear that the wedge product of the map $\tilde{K}$ gives the desired isomorphism in Proposition \ref{prop:iso-CE-jet}.

%%%%%%%%%  References %%%%%%%%%%%%%

\end{document}